\documentclass[a4paper, 10pt]{article}
\usepackage[a4paper,hmargin={2.5cm,2.5cm},vmargin={3cm,3cm}]{geometry}
\usepackage[english,activeacute]{babel}

\usepackage{amsmath,amssymb}
\usepackage{amsthm}
\usepackage{amsxtra}
\usepackage{hyperref}
\usepackage{babel,blindtext}

\usepackage{graphicx}

\usepackage{wrapfig}
\usepackage{tikz}
\usetikzlibrary{arrows,calc,decorations.pathreplacing}
\definecolor{light-gray1}{gray}{0.90}
\definecolor{light-gray2}{gray}{0.80}
\definecolor{light-gray3}{gray}{0.60}
\definecolor{light-gray4}{gray}{0.70}

\title{Dispersive estimates for rational symbols and local well-posedness of the nonzero energy NV equation. II.}
\date{}
\author{Anna Kazeykina\footnote{Laboratoire de Math\'ematiques d'{}Orsay UMR 8628, B\^at. 425 Facult\'e des Sciences d'{}Orsay,
Universit\'e Paris-Sud F-91405 Orsay Cedex France. E-mail: anna.kazeykina at math.u-psud.fr} \quad and \quad  Claudio Mu\~noz\footnote{CNRS and Departamento de Ingenier\'ia Matem\'atica DIM, Universidad de Chile, Chile. E-mail: claudio.munoz at math.u-psud.fr, cmunoz at dim.uchile.cl}}
\begin{document}

\maketitle
\markboth{A}{Anna Kazeykina and Claudio Mu\~noz}
\renewcommand{\sectionmark}[1]{}
\begin{abstract}
We continue our study on the Cauchy problem for the two-dimensional Novikov-Veselov (NV) equation, integrable via the inverse scattering transform for the {\color{black}two dimensional} Schr\"odinger operator at a fixed energy parameter. This work is concerned with the case of positive energy. For the solution of the linearized equation we derive smoothing and Strichartz estimates by combining two different frequency regimes. At non-low frequencies we also derive improved smoothing estimates with gain of almost one derivative. We combine the linear estimates with the Fourier decomposition method and $ X^{s,b} $ spaces to obtain local well-posedness of NV at positive energy in $H^s$, $s>\frac12$. Our result implies, in particular, that {\it at least} for $s>\frac12$, NV does not change its behavior from semilinear to quasilinear as energy changes sign, in contrast to the closely related Kadomtsev-Petviashvili equations.

As a supplement, we provide some new explicit solutions of NV at zero energy which exhibit an interesting behaviour at large times.
\end{abstract}

\newtheorem{lemma}{Lemma}
\newtheorem{thm}{Theorem}[section]
\newtheorem{cor}[thm]{Corollary}
\newtheorem{lem}[thm]{Lemma}
\newtheorem{prop}[thm]{Proposition}
\newtheorem{prope}[thm]{Property}
\newtheorem{ax}{Axiom}
\newtheorem*{thma}{{\bf Theorem A}}
\newtheorem*{thmap}{{\bf Theorem A'}}
\newtheorem*{thmb}{{\bf Theorem B}}
\newtheorem*{thmc}{{\bf Theorem C}}
\newtheorem*{cord}{{\bf Corollary D}}
\theoremstyle{definition}
\newtheorem{defn}{Definition}[section]

\theoremstyle{remark}
\newtheorem{rem}{Remark}[section]
\newtheorem*{notation}{Notation}
\newtheorem{Cl}{Claim}

\numberwithin{equation}{section}

\newcommand{\secref}[1]{\S\ref{#1}}
\newcommand{\lemref}[1]{Lemma~\ref{#1}}
\newcommand{\ds}[1]{\displaystyle{#1}}

\newcommand{\Com}{\mathbb{C}}
\newcommand{\const}{\mathrm{const}}
\newcommand{\dist}{\mathrm{dist}}
\newcommand{\sgn}{\mathrm{sgn}}
\renewcommand{\Re}{\mathrm{Re}}
\renewcommand{\Im}{\mathrm{Im}}
\newcommand{\re}{\operatorname{Re}}
\newcommand{\ima}{\operatorname{Im}}
\newcommand{\interior}{\mathrm{int}}
\providecommand{\abs}[1]{\left|#1 \right|}
\providecommand{\norm}[1]{\left\| #1 \right\|}

\newcommand{\meas}{\operatorname{meas}}
\newcommand{\Hess}{\operatorname{Hess}}
\newcommand{\R}{\mathbb{R}}
\newcommand{\Ss}{\mathbb{S}}
\newcommand{\N}{\mathbb{N}}
\newcommand{\Z}{\mathbb{Z}}
\newcommand{\T}{\mathbb{T}}
\newcommand{\Q}{\mathbb{Q}}
\newcommand{\la}{\lambda}
\newcommand{\pd}{\partial}
\newcommand{\al}{\alpha}
\newcommand{\ve}{\varepsilon}
\newcommand{\bt}{\beta}
\newcommand{\ga}{\gamma}
\newcommand{\de}{\delta}
\newcommand{\te}{\theta}

\newcommand{\be}{\begin{equation}}
\newcommand{\ee}{\end{equation}}
\newcommand{\ba}{\begin{equation*}}
\newcommand{\ea}{\begin{equation*}}
\newcommand{\bea}{\begin{eqnarray}}
\newcommand{\eea}{\end{eqnarray}}
\newcommand{\bee}{\begin{eqnarray*}}
\newcommand{\eee}{\end{eqnarray*}}
\newcommand{\ben}{\begin{enumerate}}
\newcommand{\een}{\end{enumerate}}
\newcommand{\nonu}{\nonumber}
\setlength\arraycolsep{1.5pt}

\newcommand{\arctanh}{\operatorname{arctanh}}
\newcommand{\fr}{\operatorname{Fr}}
\newcommand{\diam}{\operatorname{diam}}
\newcommand{\supp}{\operatorname{supp}}
\newcommand{\adh}{\operatorname{Adh}}
\newcommand{\pv}{\operatorname{pv}}
\newcommand{\vv}[1]{\partial_x^{-1}\partial_y{#1}}

\tableofcontents

\section{Introduction}  

\medskip

\subsection{Main result} In the present paper we continue our work on the Cauchy problem for the Novikov-Veselov (NV) equation, in two dimensions, with a fixed energy parameter $E\in \R$:
\be\label{NV}
\begin{aligned}
& \partial_t v  = 8 (\partial_z^3 + \partial^3_{\bar z}) v +2\partial_z (vw) +2 \partial_{\bar z}(v \bar w) -2E(\partial_z w + \partial_{\bar z} \bar w ),  \\
& \partial_{\bar z} w =  -3 \partial_z v,\\
& v|_{t=0} =v_0 ~ \hbox{given,}
\end{aligned}
\ee
where
\[
\begin{aligned}
& z =x+iy =(x,y) \in \Com,  \quad \bar z=x-iy,\\
& \partial_z = \frac12 (\partial_x - i \partial_y),  \quad \partial_{\bar z} = \frac12 (\partial_x + i \partial_y), \\
& v=v(t,x,y)  \in \R , \quad w=w(t,x,y)\in \Com.
\end{aligned}
\]

It was shown in \cite{A} (using techniques and results from Carbery, Kenig and Ziesler \cite{CKZ}, and Molinet and Pilod \cite{MP}) that if $ E = 0 $, then equation \eqref{NV} is locally well-posed in $ H^s =H^s(\R^2)$, $ s > \frac{1}{2} $. In this case, the symbol of the equation is a cubic polynomial in two dimensions. In \cite{KM}, we dealt with the case where $E$ is a fixed negative parameter and, consequently, the symbol of the linearized equation is no longer a polynomial, but a non-smooth rational function of two variables. In this case a reasonable decoupling of the stationary-phase type integral, related to the linearized equation, is not available. We were able to derive dispersive smoothing estimates for this symbol and to show local well-posedness of NV with $ E < 0 $ for initial data in the same Sobolev space $H^s$, $s>\frac12$. The purpose of this paper is to extend this result to the more involved case $E>0$, where the techniques used in our work \cite{KM} on the negative energy case do not apply directly and thus they have to be improved and generalized.

\medskip

Additionally, and complementing some previous results, we consider the global well-posedness problem for the case $E=0$. For this problem, we give some examples and results on the unusual behavior that NV equations may have at large time.

\medskip

The main result of the present work is the following theorem.
\begin{thm}\label{LWP_pos}
Assume $E>0$ in \eqref{NV}. The Novikov-Veselov equation \eqref{NV} is locally well-posed in $H^s(\R^2)$, for any $s>\frac 12$. Moreover, the lifespan of solution (if finite) is at least proportional to $E^{ \alpha }$ for some $ \alpha > 0 $.
\end{thm}

The precise, quantitative version of this Theorem is formulated as Thereom \ref{Cauchy} below. This new result, combined with \cite{A,KM}, states that NV equations are well-posed in $H^s$, $s>\frac12$. Note that a reasonable improvement of the $s>\frac 12$ regularity would require some new ideas, in particular, a deeper understanding of the Fourier-based resonance function of NV. 

\medskip

We should mention that we have very recently learned that a similar result is written in \cite{A}, but we have not been able to verify some of the estimates written in that paper. Moreover, we believe that at positive energies one cannot obtain estimates as good as those for negative energies \cite{KM}, because of some deep differences in the behavior of the NV symbol at the level of the degenerate stationary points. For more details, see e.g. the discussions after Proposition \ref{all_freq_prop} and equation \eqref{new_type}.

\medskip

For the sake of completeness, we write equations \eqref{NV} in terms of real-valued coordinates $(x,y)$. If we identify $w=w_1+iw_2 $ with the vector field $w=(w_1,w_2)$, with $w_1,w_2$ real-valued, then the second equation in \eqref{NV} becomes
\[
\partial_y w_1+\partial_x w_2 =3~\partial_y v, \quad \partial_y w_2- \partial_x w_1 =3~ \partial_x v,
\]
and the first one in \eqref{NV} reads
\[
\partial_t v = 2\left[ \partial_x (\partial_x^2 v-3 \partial_y^2v) +\nabla .(vw) -E\nabla .w \right].
\]
See \cite[Appendix A]{KM} for a derivation of this fact. However, for most of this paper we will work with the $z$-$\bar z$ formulation \eqref{NV}, which presents several advantages for derivation of precise estimates.

\medskip

Note also that one has  $w = -3\partial_{\bar z}^{-1} \partial_z v$, where $\partial_{\bar z}^{-1} \partial_z$ can be defined via the Fourier transform $\mathcal F$, acting on $L^2(\R^2)$, $\xi=(\xi_1,\xi_2)\in\R^2\backslash\{(0,0)\}$:
\be\label{W_V}
\mathcal F[\partial_{\bar z}^{-1} \partial_z f] (\xi_1,\xi_2)= \Big(\frac{\xi_1 -i \xi_2}{\xi_1 + i \xi_2}\Big)\mathcal F[f](\xi_1,\xi_2).
\ee
so that the first equation can be written as follows:
\[
\partial_t v = 4 \, \Re \{ \partial_{z} (4 \, \partial_z^2 v - 3\, v\partial_{\bar z}^{-1} \partial_z v +3\, E\partial_{\bar z}^{-1} \partial_z v) \} .
\]

\subsection{The NV equation} The NV equation \eqref{NV} is a \emph{completely integrable} model in two dimensions \cite{NV, KBook}. The origin of NV is mathematical rather than physical \cite{NV}: it arises as the integrable equation obtained by assuming that the  associated scattering problem corresponds to the standard, stationary Schr\"odinger equation in two dimensions, with fixed energy parameter $E$:
\be\label{Schrodinger}
( -\Delta_{x,y} + v(t,x,y) -E ) \psi = 0.
\ee
Recall that $ \Delta_{x,y} =\partial_x^2 + \partial_y^2 = 4 \partial_{\bar z}\partial_z. $
In that sense, NV is the most natural (from the mathematical point of view) model that generalizes the Korteweg-de Vries (KdV) equation to the two dimensional case. 

\medskip

The Novikov-Veselov equation was first obtained in an implicit form by S.V. Manakov in \cite{Manakov}.
It has the following operator representation
\be\label{Manakov_triple_0}
\partial_t  L = [L,A] +BL,
\ee
where
\[
L := -\Delta_{x,y} + v(t,x,y)  -E, \quad A:= -8(\partial_z^3 +\partial_{\bar z}^3)  -2(w \partial_z +\bar w \partial_{\bar z}),
\]
and
\[
B:= 2(\partial_z w + \partial_{\bar z}\bar w),
\]
with $ w $ given by (\ref{NV}) and $[\cdot,\cdot]$ denoting the standard commutator. Equation (\ref{Manakov_triple_0}) corresponds to the compatibility condition for the system
\[
L \varphi =0, \quad \partial_t \varphi = A\varphi.
\]
Representation (\ref{Manakov_triple_0}) is called Manakov triple representation for (\ref{NV}) and can be considered as a generalization of the Lax pair representation for KdV (see \cite{Lax}), to the $ ( 2 + 1 ) $-dimensional case.

\medskip

Compared with other dispersive models coming from physics, NV equations lack signed conserved quantities which control the long time dynamics, at least at a suitable level of regularity. For this reason, a good understanding of some particular explicit solutions is essential. The behavior of solutions to NV seems to strongly depend on the sign of the energy parameter $ E $: we will see that NV at $E<0$ ($ NV_- $) corresponds to a sort of defocusing case, whilst NV with $E>0$ ($ NV_+ $) exhibits the focusing behavior. The case $E=0$ is different in several aspects that we explain below.

\medskip

The NV equations have several interesting connections with other integrable models. In particular, it was shown (see \cite{Bog,Perry2}) that there exists a two-dimensional generalization of the well-known Miura transformation which maps solutions of the modified Novikov-Veselov equation (a two-dimensional, integrable generalization of the modified KdV equation and a member of the integrable hierarchy of the Davey-Stewartson equations) towards solutions of NV. 

\medskip

It can also be formally shown (see \cite{Gr2}, for example, and Appendix B in \cite{KM} for more details) that when the parameter of energy $ E $ tends to $ \pm \infty $, the NV equations, rescaled appropriately, become in the limit KP equations. Moreover, there is a corresponding convergence of scattering problems for the KP and NV equations.

\medskip

In the case of negative energy, $NV_-$ is in some sense reminiscent of the KPII equation (see, for example, \cite{K2, K} for results on $ NV_- $ and \cite{BS1, BS2, BM, Kis} for related results on KPII). KPII was proved to be globally well-posed in $ L^2( \mathbb{R}^2 ) $ by Bourgain in \cite{B1}. We have shown in \cite{KM} the local well-posedness of $ NV_- $ in $ H^s $, $ s > \frac{1}{2} $. A better understanding of the complicated rational symbol of $NV_-$ could help to lower the index of regularity. The global well-posedness for initial data satisfying a small-norm assumption is implied by inverse scattering theory (see \cite{Nov}). However, the existence of global solutions for large initial data is still an open problem.

\medskip

At positive energy, $ NV_+ $ exhibits a regime similar in some aspects to that of KPI (see \cite{KN, KN2, K2} for results on $ NV_+ $ and \cite{MZBIM, MST, BS1, BS2} for related results on KPI). Note that from the point of view of behavior of the solution to the Cauchy problem, KPI is essentially different from KPII: KPI is essentially a quasilinear equation, while KPII is  semilinear. The quasilinear behavior of KPI is a consequence of the results in \cite{MST1}: KPI, in contrast to KPII, cannot be solved via Picard iterations on the Duhamel formulation. This fact implies, in particular, that the flow of KPI has no enough  smoothness to be solved using an iteration scheme.

\medskip

In order to obtain a rigorous proof of the fact that KP equations present an  asymptotic regime of NV equations at high energies $|E|$, a suitable well-posedness theory for NV equations is necessary, with explicit bounds of the solution and its lifespan depending on $E$ (see Fig. \ref{Dicotomy}).
In particular, KPI and KPII being, in some sense, ``limit cases'' of the NV equation, it is natural to ask whether a similar ``bifurcation'' of the local behavior occurs for NV when the sign of $E$ changes. The main result of this paper, Theorem \ref{LWP_pos}, gives a negative answer to this question. More precisely, we see that NV is always semilinear for any value of parameter $ E $, provided the initial data has $1/2$ derivatives in $L^2$.

\medskip

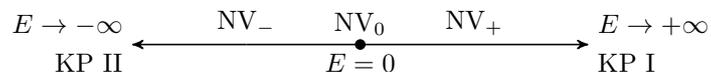
\begin{figure}[!h]
\begin{center}
\begin{tikzpicture}[
	>=stealth',
	axis/.style={semithick,->},
	coord/.style={dashed, semithick},
	yscale = 1,
	xscale = 1]
	\newcommand{\xmin}{-3};
	\newcommand{\xmax}{3};
	\newcommand{\ymin}{-3};
	\newcommand{\ymax}{3};
	\newcommand{\ta}{3};
	\newcommand{\fsp}{0.2};
	\draw [axis] (\xmin/2,0) -- (\xmax,0) node [above right] {$E\to +\infty$};
	\draw [axis] (\xmin/2,0) -- (\xmin,0) node [above left] {$E\to -\infty$};
	\draw (0,0) node [above] {NV${}_0$};
	\draw (3,0) node [below right] {KP I};
	\draw (-3,0) node [below left] {KP II};
	\draw (-1.5, 0) node [above] {NV${}_-$};
	\draw (1.5,0) node [above] {NV${}_+$};
	\draw (0,0) node [below] {$E=0$};
	\fill (0,0)  circle[radius=2pt];
\end{tikzpicture}
\end{center}
\caption{Schematic diagram of the behavior of NV equations as the energy $E$ tends to infinity, with the second variable properly rescaled (see \cite[Appendix B]{KM} for more details). Note that NV equations are semilinear in nature, but the limit to KPI as $E\to +\infty$, if rigorous, should destroy this property.}\label{Dicotomy}
\end{figure}

\subsection{Ideas of the proof}

In the proof of Theorem \ref{LWP_pos}  we follow the ideas of Bourgain, Saut and Tzvetkov \cite{ST}, and extended by Molinet and Pilod in \cite{MP} for the study of the Zakharov-Kuznetsov (ZK) equation and applied in \cite{A} for NV in the case $ E = 0 $. The proofs in the above mentioned works are essentially based on the use of a sharp $L^4$ Strichartz smoothing estimate, previously obtained by Carbery, Kenig and Ziesler \cite{CKZ}, in the case of linear dynamics arising from polynomial symbols. In the case of NV equation with $ E \neq 0 $, however, the symbol is not a polynomial, but a rational function, bounded but no longer smooth, much in the spirit of KP equations. 

\medskip

In \cite{KM}, we overcame the difficulty of showing a Carbery-Kenig-Ziesler type estimate by proving dispersive estimates with smoothing for the solution of the linearized equation, in the spirit of Kenig, Ponce and Vega \cite{KPV_Indiana}, and Saut \cite{Saut}. The dispersive estimate is proved via a stationary-phase type procedure for a two-dimensional phase depending on a complex parameter. This type of procedure has been previously used in \cite{KN,K} in the framework of the Inverse Scattering Transformation (IST) approach to NV. An important role in our estimations is played by certain changes of variables, arising naturally in the IST method for NV. The role of those changes of variables is to transform the equation for stationary points into an algebraic equation of one complex variable, simplifying significantly the subsequent manipulations with the phase.

\medskip

However, when dealing with the case $E>0$, such change of variable is only partially successful. The case $E>0$ presents a more involved problem because a similar change of variable is only valid for a particular region of the Fourier space, $|\xi|>2$ (see \eqref{ChVar} for more details). The remaining case, characterized as the low frequency regime, has to be treated using a different approach.

\medskip

From the point of view of IST, the existence of two different regimes in the case $ E > 0 $, as oppposed to a single regime in the case $ E < 0 $, can be explained by the following fact: at $ E < 0 $ the solution $ v $ of NV can be reconstructed via IST from one type of scattering data which correspond to $ \hat v $ (the Fourier transform of $ v $) in the small norm approximation, whereas at $ E > 0 $ the solution $ v $ is reconstructed from two types of scattering data: one of them gives $ \hat v $ for $ | \xi | > 2 $ in the small-norm approximation, and the other one gives $ \hat v $ for $ | \xi | \leq 2 $ (see \cite{Gr2} for more details).

\medskip

In order to show the dispersive estimates in the low frequency regime, we perform a detailed analysis of all possible configurations of stationary points of the phase that could appear as time evolves. An important problem here is the fact that when the complex paramter, arising in the stationary phase, belongs to a certain curve on the complex plane, then the corresponding points are degenerate (up to the third order of degeneracy, i.e. the stationary points can be third-order roots of the derivative of the phase). In addition, even in the case of absence of stationary points, the values of the complex parameter can be arbitrarily close to the ``degeneracy curve''.
Moreover, the symbol $|\xi|^\al$, representing the number of derivatives that we would like to obtain in our estimates, does not vanish in the degenerate points, in contrast to the case $E<0$, which leads to a slower decay of the stationary-phase type integral.

\medskip

In order to overcome the aforementioned difficulties, in the low frequency regime we first perform a new change of variables (see \eqref{ChVar2}) which allows, as in the high frequency case, to transform the equation for stationary points into an algebraic equation of one complex variable. Then we classify the stationary points of the new phase function according to how denerate they are, in terms of a parameter $\omega$ and some respective angles (see \eqref{Case1}-\eqref{Case4}).  As a third step we define two new parameters, $\omega_1$ and $\omega_2$ (see Definitions \ref{omega1} and \ref{omega2}), which take into account how close stationary points are to each other. The fourth step is to analyse these two quantities. Among the properties that we prove, we show that the values $\omega_1$ and $\omega_2$ can only be attained at some particular configurations of stationary points, and moreover, if a particular distance condition holds (implying that the stationary points are not too far from the unit circle), then both $\omega_1$ and $\omega_2$ must be bounded by the original parameter $\omega$. This simple but not less important fact allows us to estimate the degree of degeneracy of the stationary phase in the dispersive estimate.

\medskip

Although our final dispersive estimate is not as good as the original one in the negative energy case, it is enough to close the standard global Strichartz estimates. The dispersive estimates allow us to obtain an $ L^{ \infty } $ estimate for the linearized solution, at all frequencies and with decay slightly slower than $ 1/t^{3/4} $. However, if  we want to get sharp smoothing estimates, we need another new estimate, this time for large frequencies. Proposition \ref{Smoothing_Large} allows us to get the desired decay estimate at large frequencies. This dispersive estimate allows us to gain almost one derivative of the linearized solutions in $L^\infty$ with decay slightly slower than $1/t$.  Note that the estimates are weaker than in the case $ E < 0 $ (\cite{KM}), but are sufficient to close standard Bourgain's bilinear estimates for $s>\frac12$.

\medskip

In addition to linear estimates, we perform a standard Fourier localization between low and high frequencies. We recall that the resonance function that appears when dealing with the interaction of low-high to high frequencies is treated by estimating the zero level set via reasonable lower bounds on the partial derivatives of the resonance function (see \cite{ST}). Such estimates are simple to establish, but they carry a loss of accuracy (probably 1/2 of derivative) that could be avoided by dealing directly with the resonance function as Bourgain did in \cite{B1}; however, such a task is substantially more difficult given the complexity of the linear NV symbol (see \eqref{w0} for more details). However, it is worth to mention that the nonzero energy NV symbol has some useful boundedness properties near the origin, unlike standard KP equations. The low-low to low interaction is treated using a Strichartz estimate without smoothing valid for all frequencies. To treat the high-high to low/high interaction we use the smoothing estimate in $L^4$ with $1/4^-$ gain of derivative, valid for non-low frequencies.

\subsection{About the global well-posedness problem} 

We finish this introductory section with some words about the global existence problem. 
%
It turns out that in the case of NV equations, such a problem is deeply related to the behavior of scattering solutions of the associated Schr\"odinger operator \eqref{Schrodinger}. The results of Novikov \cite{Nov} imply, via the inverse scattering transform, the global existence of solutions to $ NV_- $, $ NV_+ $ for initial data with suitable spectral properties.  Those spectral properties are satisfied, in particular, under some small-norm assumptions on initial data.

\medskip

Very recently, Music and Perry \cite{MusicPerry} showed that in the case of zero energy, if the initial datum $ v_0$ has enough regularity in weighted Sobolev spaces and is such that the associated Schr\"odinger operator $ -\Delta_{x,y} + v_0 $ is critical or subcritical (i.e., nonnegative), then $ NV_0 $ ($E=0$) has a global solution. Schottdorf \cite{Schottdorf} showed in his Ph.D. thesis that the {\bf modified} NV equation has small global solutions in $L^2(\R^2)$, by making use of suitable Koch-Tataru spaces (see also a previous work of Perry \cite{Perry2}). These global solutions can only be translated, via the Miura transformation (see \cite{Perry2}), to potentials (and therefore solutions of the standard $NV_0$ equation), for which the associated Schr\"odinger operator is nonnegative.

\medskip

In the following we consider {\bf only the case $E=0$}. Recall that when the non-negativity assumption on the Schr\"odinger operator $ -\Delta_{x,y} + v_0 $ is not satisfied, the solution may develop a blow-up in finite time. An exemplifying case is the following one:
for any $a,c,d\in \R$ such that $a+c (x^3+y^3) +d(x^2+y^2)^2>0$ everywhere, the function
\[
v(t,x,y)= -2 \Delta_{x,y} \log (a -24c t +c (x^3+y^3) +d(x^2+y^2)^2)
\]
solves $NV_0$, decays like $r^{-3}$ at infinity ($r=\sqrt{x^2+y^2}$), and it blows up at finite time (see also \cite{TS2008}).
The reader may also consult the work by Adilkhanov and Taimanov \cite{AT}, where the discrete spectrum of these solutions is numerically computed. Setting $a=d=1$ (scaling symmetry), we can define, for $|c|$ small,
\be\label{Q2c}
\begin{aligned}
Q_{2,c}(t,x,y) & := -2 \Delta_{x,y} \log (1 -24c t +c (x^3+y^3) + (x^2+y^2)^2) \\
& =\frac{-32(x^2 + y^2)}{
   (1 + (x^2 + y^2)^2 + c(-24t + x^3 + y^3))^2} \\
& \quad +4c \frac{ x(-3 + 192tx + x^4) - 3(1 + x^4)y - 2(-96t + x^3)y^2 - 2x^2y^3 - 3xy^4 + y^5}{
   (1 + (x^2 + y^2)^2 + c(-24t + x^3 + y^3))^2} \\
& \quad    +6c^2 \frac{x^4 - 2x^3y - 2xy^3 + y^4 + 48t(x + y)}{
   (1 + (x^2 + y^2)^2 + c(-24t + x^3 + y^3))^2} .
\end{aligned}
\ee
A simple computation (involving the computation of some critical points) shows that $ 1+c (x^3+y^3) + (x^2+y^2)^2 >0$ for all $(x,y)\in \R^2$ provided $c\in (-c_0 ,c_0)$, with $c_0:=\frac4{3^{3/4}}>1$. We have then the following simple 

\begin{prop}[\cite{TS2008,KM}]\label{thm1}
For any $c\in (-c_0,c_0)$, the smooth rational one-parameter family $Q_{2,c}$ defined in \eqref{Q2c} solves $NV_0$, it blows up in finite positive time if $c \in (0,c_0)$, and scatters to zero as $t\to +\infty$ if $c\in (-c_0,0)$. Exactly the opposite behavior holds for $t\to -\infty$.
\end{prop}

It is interesting to notice that the particular case $c=0$ describes a static \emph{lump} solution
\be\label{Q20}
Q_{2,0} := -2 \Delta_{x,y} \log (1 + (x^2+y^2)^2)  =\frac{-32(x^2+y^2)}{(1 + (x^2+y^2)^2)^2} =\frac{-32|z|^2}{(1 + |z|^4)^2}, 
\ee
which decays as $r^{-6}$. However, if $c\neq 0$, $Q_{2,c}$ decays as $r^{-3}$. Proposition \ref{thm1} can be recast as the instability of the lump solution $Q_{2,0}$ under well-localized perturbations. For a similar behavior the reader may confront this result with similar ones in Merle-Rapha\"el-Szeftel \cite{MRS} on the instability of pseudo-conformal nonlinear Schr\"odinger blow-up solution. A more general family of threshold solutions can be obtained by using the scaling of the equation:
\be\label{v_la}
v_\la(t,x,y):= \la^2 v(\la^3 t,\la x, \la y).
\ee

\medskip

It is also interesting to obtain a better understanding of the nature of blow up solutions  of the form $Q_{2,c}$, $c>0$ in terms of a critical ``norm''. Indeed, the quantity
\be\label{Integral_L1}
\int v(t,x,y)dxdy
\ee
is formally conserved by the flow, and {\bf scaling invariant}, in the sense that if $v$ is solution to \eqref{NV}, and $\la>0$,
$v_\la$ in \eqref{v_la} is also solution to \eqref{NV} and \eqref{Integral_L1} remains invariant. With this in mind, note that
\[
\int Q_{2,0} =-16\pi , \qquad  \int Q_{2,c} <  -16\pi  \quad \hbox{ if } \quad c \neq   0 \hbox{ small}. 
\]
The second fact above can be easily obtained by numerical integration. The lump $Q_{2,0}$ is not the simplest lump solution for $NV_0$. The function
\be\label{Q10}
Q_{1,0}(x,y):=  -2 ~\Delta_{x,y} \log(1+x^2+y^2)= \frac{-8}{(1+x^2+y^2)^2} = \frac{-8}{(1+|z|^2)^2},
\ee
is solution of $NV_0$ and
\[
\int Q_{1,0} = - 8\pi.
\]
Moreover, there is a continuous branch of perturbations of $Q_{1,0}$ that are stationary solutions of NV with zero energy:
\[
Q_{1,a,b} (x,y) := -2 ~\Delta_{x,y} \log(1+ ax+by +x^2+y^2) = \frac{-2(4-a^2-b^2)}{(1+ ax+by +x^2+y^2)^2},
\]
provided $1+ ax+by +x^2+y^2>0$, that is, $a^2+b^2<4$.

\medskip

Let us mention that in \cite{Simon}, B. Simon showed that, if $v_0$ satisfies 
\[
\int |v_0(x,y)|^{1+\ve} <\infty, \quad \int (1+x^2+y^2)^\ve |v_0(x,y)| <\infty, \quad \hbox{for some $\ve>0$},
\]
then $-\Delta + \la v_0$ has a bound state for all $\la>0$ small if and only if  $\displaystyle{ \int v_0 \leq 0}$.
%
%
%
%
%
%
%
In this paper, following the approach in \cite{Chang}, we prove the following 
\begin{thm}\label{Blow}
For each $n\geq 3$ there exists a negative solution $Q_{n,0} =Q_{n,0}(t,z,\bar z)$ of $NV_0$, decaying as $ | z |^{ - 2( n + 1 ) } $, such that $\displaystyle{\int Q_{n,0} =-8n\pi}$, and $Q_{n,0}$ is globally well-defined, but blows up in infinite time: $ \|Q_{n,0}(t)\|_{L^2(\R^2)} \to +\infty$ as $t\to \pm \infty$.
\end{thm}
This result can be recast as $(i)$ there exist infinite time blow up solutions, and $(ii)$ there exists a more complex zoology of simple solutions for $NV_0$, that are not present in standard, physically-motivated models. In particular, when asking for a suitable resolution into lumps of an initial datum, what are the lump components that should appear? Additionally, this theorem reveals that globally defined solutions do not necessarily have bounded $L^2$-norm. In some sense, this fact justifies the use of \eqref{Integral_L1} as a not-being-perfect, but reasonable tool to measure the size of solutions.

\medskip

By using the scaling invariance of the equation in  \eqref{v_la}, we can make any solution and any initial data arbitrarily small in the $L^\infty$ norm. Even if Simon's theorem is not directly applicable, because the potential $Q_{n,0}(0,z,\bar z)$ is $L^1$ scaling invariant, using a suitable ensemble of almost constant cutoff functions that approximates the constant 1, we conclude the following result:

\begin{cor}
The operator $-\Delta + Q_{n,0}(0,z,\bar z)$ has always a negative bound state.
\end{cor}

In particular, the solutions $Q_{n,0}$ do not satisfy the conditions imposed in Music and Perry's  \cite{MusicPerry} work.

\medskip

We finish this introduction with some words about the case $E\neq 0$. The straightforward extension of the previous results to the nonzero case does not seem to work, and some different behavior may be expected. We believe that global well-posedness does hold for the case $E<0$, not depending on the size of the initial data. However, the panorama for $E>0$ could be even more complicated than the case $E=0$. 
%
%

\subsection{Organization of this paper} This paper is organized as follows. In Section \ref{disp_section}, we state some smoothing estimates for the linear NV equation. These bounds include not only global, but also some large frequency estimates needed for the proof of the main theorem. Section \ref{Sect_3} is devoted to the computation of some oscillatory integrals in a region of the Fourier space outside the ball $B(0,2)$, which represents the nonsingular regime. Section \ref{Sect_4} is the first part of a series of sections dealing with the estimates inside the ball $B(0,2)$ (low frequency regime). In particular, in this Section we prove several auxiliary lemmas on how stationary points are distributed. Section \ref{Sect_5} deals with the actual proof of the decay estimate with smoothing for the NV symbol. In Section \ref{large_freq} we prove the smoothing estimates leading to Proposition \ref{Smoothing_Large}. Section \ref{Sect_7} deals with the global and large-frequency Strichartz estimates needed in the subsequent section. Finally, Sections \ref{Sect_8}  and \ref{Sect_9} are devoted to the proof of bilinear estimates (for which we use the estimates obtained in the previous section), and the local well-posedness result. In Section \ref{Sect_10} we prove Theorem \ref{Blow}.

\medskip

{\bf Acknowledgments.} We would like to thank J.-C. Saut and H. Koch for stimulating discussions and insights. A. Kazeykina is grateful to the University of Chile for the kind hospitality during her stay in Santiago in July 2015 when a part of this work was done. C. Mu\~noz would like to thank the Laboratoire de Math\'ematiques d'Orsay where part of this work was done. He is partially funded by Fondecyt 1150202 Chile, Fondo Basal CMM (U. Chile), and Millennium Nucleus Center for Analysis of PDE NC130017.

\medskip

{\bf Notations.} In the following text the notation $ A \lesssim B $ means that there exists a constant $ c > 0 $ (depending on parameters specified in each context) such that $ A \leq c B $. Similarly, the notation $ A \gtrsim B $ means that there exists a constant $ c > 0 $ such that $ A \geq c B $. Finally, the notation $ A \eqsim B $ means that there exist constants $ c_1 > 0 $, $ c_2 > 0 $ such that $ c_2 B \leq A \leq c_2 B $.

\medskip

In the text of the paper we identify $ x =( x_1, x_2 ) \in \mathbb{R}^2 $ with $ x = x_1 + i x_2 \in \mathbb{C} $. 

\bigskip

\section{Smoothing estimates for positive energies}\label{disp_section}

\subsection{A global estimate}

The aim in the following sections is to estimate, for $ E > 0 $, the integral
\be\label{I}
I = \int\limits_{ \mathbb{C} } | \xi |^{ \alpha }  e^{ i t \tilde S( u, \xi ) } d \Re \xi d \Im \xi,
\ee
where $\displaystyle{ u = \frac xt }$ and
\begin{equation}
\label{st_phase}
\widetilde S( u, \xi ) = ( \xi^3 + \bar \xi^3 )\left( 1 - \frac{ 3 E }{ \xi \bar\xi } \right) + \frac{ 1 }{ 2 }( \bar u \xi + u \bar \xi ),
\end{equation}
is the symbol associated to the NV equation. In Sections \ref{disp_section}--\ref{large_freq} we assume for that $ E = 1 $.  A scaling argument allows to reduce the general case $ E > 0 $ to the particular case $ E = 1 $ (see Section \ref{Sect_7} for details).

We will prove the following result.
\begin{prop}
\label{all_freq_prop}
For any $0\leq\alpha\leq\frac{1}{4}$ and any $ \varepsilon > 0 $ small, one has
\be\label{xi_alpha}
|I| \lesssim \frac{ 1 }{ t^{ 3 / 4 - \varepsilon } }
\ee
uniformly in $u\in \R^2$.
\end{prop}
The constant involved in the above estimate depends on $\al$ and $ \varepsilon $ only; however, the decay does not improve as $\al$ increases as in our previous paper \cite{KM}. This is one of the main differences between the cases of positive and negative energies, and, as we will se later, this estimate will not be enough to close some Fourier decomposition estimates. 

\medskip

Similarly to the case of negative energy treated in \cite{KM}, we will start by making a change of variables that allows to treat \eqref{xi_alpha}. This change of variables is motivated by the form of the linearized ``inverse scattering solution'' of the NV equation. Recall that at $ E > 0 $ the inverse scattering solution is constructed from two types of scattering data: the classical scattering amplitude which in Born approximation represents the Fourier transform of the solution in the ball 
\[
 B_2 = \big\{  \xi \in \R^2 ~:~ | \xi | \leq 2 \big\} ,
\]
 and Faddeev's scattering data which in Born approximation represent the Fourier transform of the solution outside the ball $ B_2 $. Thus we will perform different changes of variables, one inside the ball $ B_2 $, and other outside this ball.

More precisely, we shall split \eqref{I} into two pieces:
\be\label{I_split}
\begin{aligned}
I & =  \int\limits_{|\xi| > 2 } | \xi |^{ \alpha }  e^{ i t \tilde S( u, \xi ) } d \Re \xi d \Im \xi + \int\limits_{|\xi| \leq 2 } | \xi |^{ \alpha }  e^{ i t \tilde S( u, \xi ) } d \Re \xi d \Im \xi \\
& =: I_{out} + I_{in}.
\end{aligned}
\ee
The purpose of Sections 3, 4 and 5 is to estimate both $ I_{in} $ and $ I_{out} $. A delicate analysis of the behaviour of stationary points of $\tilde S$ will be essential for getting these estimates.

\subsection{Improved smoothing estimate for large frequencies}

Estimate \eqref{xi_alpha} is not enough to close some key estimates in the iteration scheme; for this reason, we will need an additional ``localized'' estimate.

\medskip

In the following lines we prove an additional smoothing estimate for the linear dynamics in the case of large frequencies. 
Fix $ R > 2 $. Take $ \psi_R \in C^{ \infty }_0( \mathbb{R}^2, [ 0, 1 ] ) $ to be a bump function such that 
\be\label{Phi_R}
\psi_R( \xi ) = 0 \text{ for } | \xi | \leq R, \quad \psi_R( \xi ) = 1 \text{ for } | \xi | \geq R + 1. 
\ee
Now define
\begin{equation}\label{I_R}
I_R := \int\limits_{\mathbb{C} } | \xi |^{ \alpha } \psi_R( \xi )  e^{ i t \tilde S( u, \xi ) } d \Re \xi d \Im \xi.
\end{equation}
Then we have

\begin{prop}\label{Smoothing_Large}
For any $ 0 \leq \alpha < 1 $ and any $ \varepsilon > 0 $ small the following estimate is valid
\begin{equation}
\label{large_disp}
| I_R | \lesssim \frac{ 1 }{ t^{ 1 - \varepsilon } },
\end{equation}
uniformly on $ u \in \mathbb{C} $. The implicit constant depends on $ \alpha $, $ \varepsilon $ and $ R $ only.
\end{prop}

\medskip

Proposition \ref{Smoothing_Large} is proved in Section \ref{large_freq}. In Section \ref{Sect_3} we estimate $I_{out}$ of \eqref{I_split}. Sections \ref{Sect_4} and \ref{Sect_5} are devoted to the estimate of the integral $ I_{in} $ of \eqref{I_split}.

\medskip

In Sections \ref{Sect_3}--\ref{Sect_5} the implicit constants arising in the estimates depend on $ \alpha $ only.

\bigskip

\section{The integral outside the ball $B_2$}
\label{Sect_3}

The purpose of this Section is to estimate $I_{out}$ in \eqref{I_split}, in order to show estimate \eqref{xi_alpha}.

\subsection{Main ideas}
Outside the ball $ B_2 $ we perform the following change of variables :
\be\label{ChVar}
\xi = \lambda + \frac{ 1 }{ \overline{\lambda} }, \quad \bar \xi = \bar \lambda + \frac{ 1 }{ \lambda }.
\ee
Denote 
\be\label{f}
 f( \lambda ) = \lambda + \frac{ 1 }{ \overline{\lambda} } =\xi .
\ee 
Note that
\begin{equation*}
f \colon \{ | \lambda | > 1 \} \longrightarrow \{ | \xi | > 2 \},
\end{equation*}
is a bijective smooth map. Note also that 
\[
 \frac{ D( \xi, \bar \xi ) }{ D( \lambda, \bar \lambda ) } = 1 - \frac{ 1 }{ \lambda^2 \bar \lambda^2 } .
\]
We may remark that 
\be\label{simply}
\frac{ \bar \lambda }{ \lambda } = \frac{ \bar \xi }{ \xi },
\ee
so that from \eqref{st_phase} we have the following equality
\be\label{S_u}
\begin{aligned}
i t \tilde S( u, f( \lambda ) ) & = i t \left\{ \left( \lambda^3 + \frac{ 1 }{ \lambda^3 } + \bar \lambda^3 + \frac{ 1 }{ \bar \lambda^3 } \right) + \frac{ 1 }{ 2 } \left( \left( \lambda + \frac{ 1 }{ \bar \lambda } \right) \bar u + \left( \bar \lambda + \frac{ 1 }{ \lambda } \right) u \right) \right\} \\
& : = i t S( u, \lambda).
\end{aligned}
\ee
Thus, we need to estimate the following integral
\begin{equation}
\label{new_type}
I_{out} = \iint\limits_{ \mathbb{C} \backslash B_1( 0 ) } \frac{ | \lambda \bar \lambda + 1 |^{ \alpha } ( | \lambda |^4 - 1 ) }{ | \lambda |^{ \alpha + 4 } } e^{ i t S( u, \lambda ) } d \Re \lambda d \Im \lambda.
\end{equation}
We shall see that the main difference of the treated problem with the case of negative energy is that the term coming from $ | \xi |^{ \alpha } $, i.e. $ | \bar \lambda \lambda + 1 |^{ \alpha } $, does not vanish in the degenerate stationary points.

\medskip

Note that the fact that $ | \lambda \bar \lambda - 1 |^{ \alpha } $  vanishes in the degenerate stationary points was not used in \cite{KM} in the proof of the corresponding estimate for small $ t $ in the case of negative energy.
Therefore, repeating the reasonings that were carried out for the case of negative energy (see the proof of \cite[Lemma 2.1]{KM}), we will obtain the following

\begin{lemma}\label{i_out_lemma}
Consider the integral $I_{out}$ defined in \eqref{new_type}. For all $ 0 \leq \alpha < 1 $,
\begin{equation*}
| I_{out} | \lesssim \frac{ 1 }{ t^{ \frac{ \alpha + 2 }{ 3 } } }
\end{equation*}
uniformly on $ u \in \mathbb{C} $, $ 0 < t \leq r $, where $ r = e^{ - \frac{ 3 }{ 1 - \alpha } } $, and 
\begin{equation*}
| I_{out} | \lesssim \frac{ | \ln t | }{ t^{ 3 / 4 } }
\end{equation*}
uniformly on $ u \in \mathbb{C} $,  $ t > r  $.
\end{lemma}

\begin{cor} For all $ 0 \leq \alpha \leq \frac{1}{4} $ and any $ \varepsilon > 0 $ small
\begin{equation*}
| I_{out} | \lesssim \frac{ 1 }{ t^{ 3 / 4 - \varepsilon } }
\end{equation*}
uniformly on $ u \in \mathbb{C} $.
\end{cor}

The proof of this result follows the ideas of the proof of Lemma 2.1 in \cite{KM}, with small differences (in the case of large $t$) coming from the fact that now the term $ | \bar \lambda \lambda + 1 |^{ \alpha } $ does not help to improve the estimates. For the sake of completeness, we include a sketch of proof of this result. For more details, the reader can consult \cite[Lemma 2.1]{KM}.

\subsection{Review on NV stationary points}

We recall the following parametrically defined sets of the complex plane  (see Fig. \ref{u_curve_figure}). Let
\be\label{mU}
\mathcal{U} := \{ \tilde u\in \Com \ : \  \tilde u = 6 ( 2 e^{ - i \varphi } + e^{ 2 i \varphi } ), \; \varphi \in [ 0, 2 \pi ) \},
\ee
and
\be\label{UU}
\mathbb{U} := \{ \tilde u\in \Com \ : \  \tilde u =  6t  (2 e^{ - i \varphi } + e^{ 2 i \varphi }), ~ t\in [0,1] , ~ \varphi \in [ 0, 2 \pi ) \},
\ee
be the (closed) region enclosed by the curve $ \mathcal{U} $. These sets will be essential to understand the stationary points of the phase function $S(u,\la)$.
\begin{figure}[!h]
\begin{center}
\includegraphics[width=100mm]{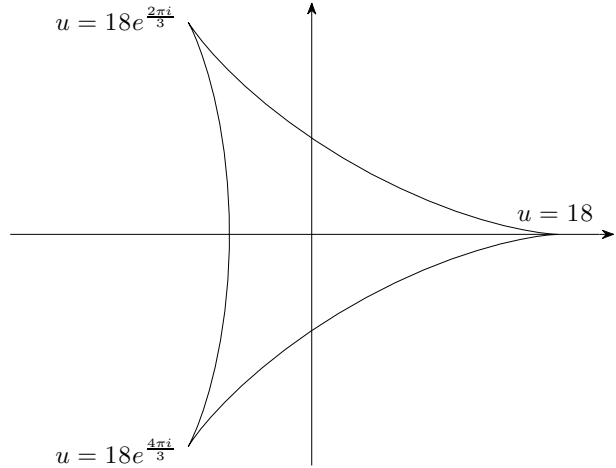}
\caption{The curve $ \mathcal{U} $ and its enclosed region $\mathbb U$ in the complex plane. Note that $\mathcal U$ (and its interior and exterior) is invariant with respect to the transformations $z \mapsto \bar z$ and $z \mapsto ze^{2ik\pi/3}$, $k=0,1,2$.}
\label{u_curve_figure}
\end{center}
\end{figure}

\medskip

With these definitions in mind, let us describe the properties of the stationary points of the function $ S( u, \lambda ) $ defined in \eqref{S_u}. These points satisfy the equation\footnote{Here the symbol $\overset{!}{=}$ means that equality to zero is satisfied for stationary points only.}
\begin{equation}
\label{asp_stationary}
S_{ \lambda } (u,\la)= \frac{ \bar u }{ 2 } - \frac{ u }{ 2 \lambda^2 } - 3 \lambda^2 + \frac{ 3 }{ \lambda^4 } \overset{!}{=}  0,
\end{equation}
where $S_\la$ stands for the partial derivative with respect to $\la$.

Additionally, the degenerate stationary points obey the equation
\begin{equation}
\label{asp_degenerate}
S_{ \lambda \lambda }(u,\la) = \frac{ u }{ \lambda^3 } - 6 \lambda - \frac{ 12 }{ \lambda^5 } \overset{!}{=} 0.
\end{equation}
We denote $ \zeta = \lambda^2 $, and define $Q( u, \zeta ) := S_{ \lambda }(u,\la)$, namely
\be\label{Q_u}
Q( u, \zeta ) = \frac{ \bar u }{ 2 } - \frac{ u }{ 2 \zeta } - 3 \zeta + \frac{ 3 }{ \zeta^2 }.
\ee
Clearly, for each $ \zeta =\zeta(u)$, root of the function $ Q( u, \zeta ) $, there are two corresponding stationary points of $ S( u, \lambda ) $, given by $ \lambda = \pm \sqrt{ \zeta } $.
Since $Q( u, \zeta ) $ has only three roots counting multiplicity (say $\la_0^2(u)$, $\la_1^2(u)$ and $\la_2^2(u)$), in terms of the variable $\la^2$, the function $ S_{ \lambda }( u, \lambda ) $ can be represented in the following compact form
\begin{equation}
\label{asp_sprime_representation}
S_{ \lambda }( u, \lambda ) = - \frac{ 3 }{ \lambda^4 } ( \lambda^2 - \lambda_0^2( u ) ) ( \lambda^2 - \lambda_1^2( u ) ) ( \lambda^2 - \lambda_2^2( u ) ).
\end{equation}
Concerning the behavior of the roots $\la_j(u)$, defined in \eqref{asp_sprime_representation}, we have the following result, see \cite{KN} for a proof.\footnote{Note that in \cite{KN} the parametrization of the set $\mathbb U$ was slightly different, here we give a more precise one.}

Indeed, the stationary points of the phase satisfy equation (3.7) of \cite{KM} and degenerate stationary points satisfy in addition equation (3.8) of \cite{KM}. Thus their description is given by the following Lemma.
\begin{lemma}[Description of stationary points, see also \cite{KM}]\label{st_points_lemma}
\rule{1pt}{0pt}
Assume that $\tilde u\in \Com$ is a fixed parameter. Then the following are satisfied.
\begin{enumerate}
\item\label{1} If $ \tilde u = 18 e^{ \frac{ 2 \pi i k }{ 3 } } $, $ k = 0, 1, 2 $ (see the vertices of $\mathcal U$ in Fig. \ref{asp_sprime_representation}), then
\begin{equation*}
\lambda_0( \tilde u ) = \lambda_1( \tilde u ) = \lambda_2( \tilde u ) = e^{ -\frac{ \pi i k }{ 3 } },
\end{equation*}
and $ S( \tilde u, \lambda ) $ has two degenerate stationary points, corresponding to a third-order root of the function $ Q( \tilde u, \zeta ) $, $ \zeta_0 = e^{ -\frac{ 2 \pi i k }{ 3 } } $.

\item\label{2} If $ \tilde u \in \mathcal{U} $ $($i.e. $\tilde  u = 6( 2 e^{ - i \varphi } + e^{ 2 i \varphi } )$ $)$ and $ \tilde u \neq 18 e^{ \frac{ 2 \pi i k }{ 3 } } $,  for $ k = 0, 1, 2 $, then\footnote{Note that here we enumerate the stationary points in a slightly different way in comparison with the same Lemma stated in \cite{KM}. This is done for convenience of presentation in Section \ref{Sect_4}.}
\begin{equation*}
\lambda_0( \tilde u ) = \lambda_2( \tilde u ) = e^{ i \varphi / 2 }, \quad \lambda_1( \tilde u ) = e^{ - i \varphi }.
\end{equation*}

Thus $ S( \tilde u, \lambda ) $ has two degenerate stationary points, corresponding to a second-order root of the function $ Q( \tilde u, \zeta ) $, $ \zeta_0 = e^{ i \varphi } $, and two non--degenerate stationary points corresponding to a first-order root, $ \zeta_1 = e^{ - 2 i \varphi } $.

\item\label{3} If $ \tilde u \in \interior  ~ \mathbb{U} $, then
\begin{equation*}
\lambda_i( \tilde u ) = e^{ i \varphi_i }, \quad \text{and} \quad \lambda_i( \tilde u ) \neq \lambda_j( \tilde u ) \quad \text{for} \quad i \neq j.
\end{equation*}
In this case the stationary points of $ S( \tilde u, \lambda ) $ are non-degenerate and correspond to the roots of the function $ Q( \tilde u, \zeta ) $ with absolute value equals 1.

\item\label{4} Finally, if $\tilde  u \in \mathbb{C} \backslash \mathbb{U} $, then
\be\label{la_j}
\lambda_0( \tilde u ) = ( 1 + \omega ) e^{ i \varphi / 2 }, \quad \lambda_1( \tilde u ) = e^{ - i \varphi }, \quad  \lambda_2(\tilde  u ) = ( 1 + \omega )^{ -1 } e^{ i \varphi / 2 },
\ee
for certain $ \varphi \in \R$ and $ \omega > 0 $.

In this case the stationary points of the function $ S(\tilde  u, \lambda ) $ are non-degenerate, and correspond to the roots of the function $ Q(\tilde  u, \zeta ) $ that can be expressed as $ \zeta_0 = ( 1 + \tau ) e^{ i \varphi } $, $ \zeta_1 = e^{ - 2 i \varphi } $, $ \zeta_2 = ( 1 + \tau )^{ - 1 } e^{ i \varphi } $, and $ ( 1 + \tau ) = ( 1 + \omega )^2 $.

\end{enumerate}
\end{lemma}

\subsection{Estimate of $I_{out}$}

For small $ t $ the integral $ I_{ out } $ is estimated as in the case of $ E < 0 $ (see \cite[Subsection 3.5]{KM}). In particular, we obtain that 
\[
 | I_{ out } | \lesssim \frac{ 1 }{ t^{ \frac{\alpha + 2 }{ 3 } } }.
\]  

For the case of large $ t $, similarly to the case of negative energy, we are brought to consider three different cases:
\begin{enumerate}
\item $ u \in \mathbb{C} \backslash \mathbb{U} $ and $ | \lambda_0 | = 1 + \omega \geq 2 $;
\item $ u \in \mathbb{U} $;
\item $ u \in \mathbb{C} \backslash \mathbb{U} $ and $ | \lambda_0 | = 1 + \omega < 2 $.
\end{enumerate}

Case 1 is treated absolutely in the same way as Case 1 for the negative energy. In the Cases 2 and 3 the integrals $ I_{ 2 }^{ j, + } $, $ I_3^+ $ (over $ \mathbb{C} \backslash \Omega = \{ \lambda \colon | \lambda | \geq 2 \} $) are also treated in the same way as for $ E < 0 $. For integrals $ I_1 $, $ I_{ 2, j }^- $, $ I_3^- $ the variable of integration belongs to $ \Omega = \{ \lambda \colon | \lambda | < 2 \} $ and on $ \Omega $ we have that
\begin{equation*}
| \lambda |^2 + 1 \simeq 1 \simeq ( | \lambda |^2 - 1 )^0.
\end{equation*}
Thus applying to integrals $ I_1 $, $ I_{ 2, j }^- $, $ I_3^- $ the reasoning of the negative energy with $ \alpha = 0 $ we obtain that
\begin{equation*}
| I_1 | + | I_{ 2, j }^- | + | I_3^- | \lesssim \frac{ | \ln t | }{ t^{ 3/4 } }.
\end{equation*}
Finally, we conclude that 
\[
 | I | \lesssim \frac{ | \ln t | }{ t^{ 3/4 } },
\]
as expected.

\bigskip

\section{The integral inside the ball $ B_2 $: auxiliary lemmas}\label{Sect_4}

In this Section we estimate the second part of the integral in \eqref{I_split}, only present in the case $E>0$.

\subsection{Preliminaries}

In this section our purpose is to estimate the integral
\be\label{J}
I_{in}= \int_{B_2}  | \xi |^{ \alpha }  e^{ i t \tilde S( u, \xi ) } d \Re \xi d \Im \xi.
\ee
Here $B_2$ is the ball of radius 2 centered at the origin. 

\medskip

Note that $ | I_{ in } | \leq \pi 2^{ \alpha + 2 } $. Thus, for small $ t $ (e.g. $ | t | \leq r $ for any $ r > 0 $) the estimate (\ref{xi_alpha}) holds. Consequently, we only need to consider the case of large $ t $ ($ | t | > r $ for some $ r > 0 $).

\medskip

Additionally, note that inside the ball $ B_2 $ variable $ \xi $ cannot be represented in the form \eqref{ChVar}. 
Thus we need a different approach in this case.

\medskip

In the ball $ B_2 $ we shall perform the following change of variables
\be\label{ChVar2}
\xi = \lambda + \lambda', \quad \lambda, \lambda' \in  \Ss^1.
\ee
Note that this change of variables is not a bijection unless we provide a better description of the domain for each $\la,\la'$. Let us write 
\begin{equation}
\label{lambda_phi}
\lambda = e^{ i \varphi_1 }, \quad \lambda' = e^{ i \varphi_2 }.
\end{equation}
Denote 
\be\label{f2}
 f( \varphi_1, \varphi_2 ) = e^{ i \varphi_1 } + e^{ i \varphi_2 } = (\cos \varphi_1 + \cos \varphi_2) + i(\sin \varphi_1 +\sin \varphi_2).
\ee
Now, note that
\be\label{D1}
f \colon D_1 = \left\{\begin{matrix} 0 < \varphi_1 < 2 \pi, \\ \varphi_1 < \varphi_2 <  \varphi_1 +\pi \end{matrix} \right\} \longrightarrow \{ 0 < | \xi | < 2 \},
\ee
and
\be\label{D2}
 f \colon D_2 = \left\{\begin{matrix} 0 < \varphi_1 < 2 \pi, \\  \varphi_1 - \pi < \varphi_2 < \varphi_1 \end{matrix} \right\}  \longrightarrow \{ 0 < | \xi | < 2 \}, 
\ee
are bijective smooth maps, see Figs. \ref{Biy1} and \ref{Biy2} for more details. 

\bigskip


\begin{figure}[!h]
\minipage{0.45\textwidth}
\begin{tikzpicture}[
	>=stealth',
	axis/.style={semithick,->},
	coord/.style={dashed, semithick},
	yscale = 0.7,
	xscale = 0.7]
	\newcommand{\xmin}{-4};
	\newcommand{\xmax}{4};
	\newcommand{\ymin}{-4};
	\newcommand{\ymax}{4};
	\newcommand{\ta}{3};
	\newcommand{\fsp}{0.2};
	\filldraw[color=light-gray2] (0,0) circle (4);
	\filldraw[color=light-gray1] (0,0) circle (2);
	\draw [axis] (\xmin,0) -- (\xmax,0) node [right] {$\re \la$};
	\draw [axis] (0,\ymin) -- (0,\ymax) node [below left] {$\ima \la$};
	\draw (2.2,-0.3) node [left] {$1$};
	\draw (4.2,-0.3) node [left] {$2$};
	\draw [thick] (1.4,1.4) -- (0,0);
	\draw [dashed] (1.4,1.4) -- (2.8,2.8);
	\draw [dashed] (1.4,1.4) -- (2.8,1.4);
	\draw [thick] (1.4,1.4) -- (1.95,3.05);
	\draw (1.4,1.4) node [below] {$\la$};
	\draw (2.1,3.4) node [left] {$\la +\la'$};
	\draw (1.5,0.6) node [below] {$\varphi_1$};
	\draw (2.7,2.8) node [left] {$\varphi_2$};
	\fill (1.4,1.4)  circle[radius=2pt];
	\fill (1.95,3.05)  circle[radius=2pt];
	\draw (2.5,1.44) arc (0:75:1);
	\draw (1,0) arc (0:45:1);
\end{tikzpicture}
\caption{The representation of the pair $\la +\la'$ described in \eqref{D1}. Note that the angle $\varphi_2$ varies from the value $\varphi_1$ up to $\varphi_1 +\pi$.}\label{Biy1}
\endminipage\hfill
\minipage{0.45\textwidth}
\begin{tikzpicture}[
	>=stealth',
	axis/.style={semithick,->},
	coord/.style={dashed, semithick},
	yscale = 0.7,
	xscale = 0.7]
	\newcommand{\xmin}{-4};
	\newcommand{\xmax}{4};
	\newcommand{\ymin}{-4};
	\newcommand{\ymax}{4};
	\newcommand{\ta}{3};
	\newcommand{\fsp}{0.2};
	\filldraw[color=light-gray2] (0,0) circle (4);
	\filldraw[color=light-gray1] (0,0) circle (2);
	\draw [axis] (\xmin,0) -- (\xmax,0) node [below right] {$\re \la$};
	\draw [axis] (0,\ymin) -- (0,\ymax) node [below left] {$\ima \la$};
	\draw (2.2,-0.3) node [left] {$1$};
	\draw (4.2,-0.3) node [left] {$2$};
	\draw [thick] (1.4,1.4) -- (0,0);
	\draw [dashed] (1.4,1.4) -- (2.8,2.8);
	\draw [dashed] (1.4,1.4) -- (2.6,1.4);
	\draw [thick] (1.4,1.4) -- (2.9,0.42);
	\draw (1.4,1.4) node [below] {$\la$};
	\draw (2.9,0.42) node [right] {$\la +\la'$};
	\draw (1.5,0.6) node [below] {$\varphi_1$};
	\draw (3,0.7) node [above] {$\varphi_2$};
	\fill (1.4,1.4)  circle[radius=2pt];
	\fill (2.9,0.42)  circle[radius=2pt];
	\draw (2.5,1.44) arc (0:-37:1);
	\draw (1,0) arc (0:45:1);
\end{tikzpicture}\caption{The pair $\la +\la'$ described in \eqref{D2}. Note that the angle $\varphi_2$ varies from the value $\varphi_1 -\pi$ up to $\varphi_1$.}\label{Biy2}
\endminipage
\end{figure}

Note also that a simple computation shows that the Jacobian obeys the relation
\[
 \frac{ D( \Re \xi, \Im \xi ) }{ D( \varphi_1, \varphi_2 ) } = \sin( \varphi_2 - \varphi_1 ) .
\]
One can also check that the function $\tilde S$ defined in \eqref{st_phase} has now the following representation
\begin{equation}\label{SS}
\begin{aligned}
i t \tilde S( u, \xi ) ) & = i t \left\{ \left( \lambda^3 + \bar \lambda^3 + \lambda'^3 + \bar \lambda'^3 \right) + \frac{ 1 }{ 2 } \left( \left( \lambda + \lambda' \right) \bar u + \left( \bar \lambda + \bar \lambda' \right) u \right) \right\}\\
&  : = i t S( u, \lambda, \lambda').
\end{aligned}
\end{equation}
Thus, using the bijective character of $f$ on each $D_j$, we can write that for each $j=1,2$,
\[
I_{in}= \int\limits_{ D_j } | \lambda + \lambda' |^{ \alpha } e^{ i t S( u, \lambda, \lambda' ) } | \sin( \varphi_1 - \varphi_2 ) | d \varphi_1 d \varphi_2  =: I_{in,j}.
\]
Therefore, adding these two identities we get
\be\label{J_new}
\begin{aligned}
I_{in} & = \frac{ 1 }{ 2 } ( I_{in,1} + I_{in,2} )\\
&  = \frac{ 1 }{ 2 } \int\limits_{ 0 }^{ 2 \pi } \int\limits_{ 0 }^{ 2 \pi } | \lambda + \lambda' |^{ \alpha } e^{ i t S( u, \lambda, \lambda' ) } | \sin( \varphi_1 - \varphi_2 ) | d \varphi_1 d \varphi_2.
\end{aligned}
\ee
In what follows we are going to use the above representation of $ I_{in} $.

Note that
\begin{equation*}
S_{ \varphi_1 } = i ( S_{ \lambda } \lambda - S_{ \bar \lambda } \bar \lambda ) = i \left( 3 \lambda^3 + \frac{ 1 }{ 2 } \bar u \lambda - \frac{ 3 }{ \lambda^3 } - \frac{ u }{ 2 \lambda } \right),
\end{equation*}
and
\[
S_{ \varphi_1,\varphi_1 } =i  \la S_{ \varphi_1,\la}=  -\la \left( 9 \lambda^2 + \frac{ 1 }{ 2 } \bar u  + \frac{ 9 }{ \lambda^4 } +  \frac{ u }{ 2 \lambda^2 } \right).
\]
Thus the stationary points of $ S $ are the roots of equation (3.7) in \cite{KM}, but with $ u $ replaced by $ -u $, and (additionally) they are in $ \Ss^1 $. For the sake of completeness, \cite[equation (3.7)]{KM} is given by
\[
 3 \lambda^2 - \frac{ \bar u }{ 2 } - \frac{ 3 }{ \lambda^4 } +\frac{ u }{ 2 \lambda^2 }   =  0.
\]
One can also show that the degenerate stationary points satisfy additionally equation (3.8) of the same paper (as in the case of negative energy), but with $u$ replaced by $-u$:
\[
\frac{ u }{ \lambda^3 } + 6 \lambda + \frac{ 12 }{ \lambda^5 } = 0.
\]
Therefore, from Lemma \ref{st_points_lemma} applied to $-u$, we have only three possibilities for critical points: cases 1, 2 and 3 of the lemma.

\medskip

Following the scheme similar to the one presented for the case of negative energies \cite{KM}, we can obtain the following result:
\begin{lem}\label{Lemma_I_in}
One has, uniformly in $u=x/t$,
\be\label{I_in}
| I_{in} | \lesssim \frac{ 1 }{ t^{ 3/4 } }.
\ee
\end{lem}
For the proof of this result, we proceed in several steps detailed in the following subsections.

\subsection{Setting of the problem}

Recall that the stationary points of the phase $ S $ of (\ref{SS}) are in $ \mathbb{S}^1 $. However, if $ -u \in \mathbb{C} \backslash \mathbb{U} $, and thus the phase $ S $ does not have stationary points, the derivatives $ S_{ \varphi_1 } $, $ S_{ \varphi_2 } $ can be very close to zero due to the existence of zeros of $ S_{ \varphi_1 } $, $ S_{ \varphi_2 } $ viewed as functions of $ \lambda \in \mathbb{C} $, $ \lambda' \in \mathbb{C} $. Thus we need to study, more generally, stationary points of $ S $ on $ \mathbb{C} $.

Let 
 \be\label{deg_points}
  \lambda_k^* = e^{ - \frac{ i \pi k }{ 3 } } , \quad  k \in N_5 = \{  0, 1, \ldots, 5 \},
 \ee
 i.e. $ \lambda_k^* $ are the degenerate stationary points of $ S $ (see \eqref{SS}) corresponding to the particular case $ u_k^* = {\color{red}-}18 e^{ \frac{ 2 \pi i k }{ 3 } } $.
Let $ \lambda_j=\la_j(-u) $, $ j \in N_5 $ denote the stationary points of the phase $ S $ for a given $ u $ with $ \lambda_j $, $ j = 0, 1, 2 $ given by Lemma \ref{st_points_lemma}, and  $ \lambda_{ j + 3 } = - \lambda_j $, $ j = 0, 1, 2 $. In other words:
\[
\begin{cases}
\begin{matrix}
~\la_0 ~ ~& \la_1 ~~ & \framebox(13,14){$\la_2$}~ ~ & \la_3 ~ & \la_4 ~&  \framebox(13,14){$\la_5$} \\
~\downarrow  ~ & \downarrow  ~ & \downarrow  ~ & \downarrow  ~ & \downarrow  ~& \downarrow \\
~\la_0 ~ ~& \la_1 ~~ &  \framebox(13,14){$\la_2$} ~~ & -\la_0 ~ & -\la_1 ~ &  \framebox(19,14){$-\la_2$}
\end{matrix}
\end{cases}
\]
(The pair (2,5) is inside a box for later reasons.) From this arrangement we see that the size of the difference between any two points is at most equals 2 if the $\la_j$ are on the unit circle.  For the record, we have four different scenarios for the set of $(\la_j)_{j=0,1,2}$:
\ben
\item Case 1. $\la_0=\la_1=\la_2 = e^{-i\pi k/3}$, for some $k=0,1,2$. In this case,
\be\label{Case1}
(\la_0,\la_1,\la_2 ~ ; ~ \la_3,\la_4,\la_5) = e^{-i\pi k/3}(1,1,1~ ; ~ -1,-1,-1).
\ee
\item Case 2. $\la_0=\la_2 = e^{i\varphi/2}$, $\la_1= e^{-i\varphi} \neq \la_0,\la_2$.  Here,
\be\label{Case2}
(\la_0,\la_1,\la_2 ~ ; ~ \la_3,\la_4,\la_5) = (e^{i\varphi/2},e^{-i\varphi},e^{i\varphi/2} ~ ; ~  -e^{i\varphi/2},-e^{-i\varphi},-e^{i\varphi/2}).
\ee
\item Case 3. $\la_j=e^{i \varphi_j}$, $j=0,1,2$, $\la_j\neq \la_k$ if $j\neq k$. In this case:
\be\label{Case3}
(\la_0,\la_1,\la_2 ~ ; ~ \la_3,\la_4,\la_5) = (e^{i\varphi_0},e^{i\varphi_1},e^{i\varphi_2} ~ ; ~  -e^{i\varphi_0},-e^{i\varphi_1},-e^{i\varphi_2}).
\ee
\item Case 4. $\la_0= (1+\omega)e^{i \varphi/2}$, $\la_2= (1+\omega)^{-1}e^{i \varphi/2}$, and $\la_1= e^{-i \varphi}$, for some $\varphi \in \R$. Here we have
\be\label{Case4}
(\la_0,\la_1,\la_2 ~ ; ~  \la_3,\la_4,\la_5) = \Big( (1+\omega)e^{i \varphi/2}, e^{-i\varphi}, \frac{e^{i \varphi/2}}{1+\omega} ~ ; ~  -(1+\omega)e^{i \varphi/2} ,-e^{-i\varphi},-\frac{e^{i \varphi/2}}{1+\omega} \Big).
\ee
\een

\begin{defn}\label{omega1}
Let $ \omega_1 $ denote the minimal distance between two stationary points, \emph{excluding} the pair $ ( \lambda_2, - \lambda_2 ) = ( \lambda_2, \lambda_5 ) $:
\begin{equation*}
\omega_1 = \min_{ P_1 } | \lambda_i - \lambda_j  |, \quad P_1 = \Big\{ (i, j  ) \in N_5^2, \quad i < j,  \quad ( i, j ) \neq ( 2, 5 ) \Big\}.
\end{equation*}
\end{defn}
Let $ i^m $, $ j^m $ be such that $ \omega_1 = | \lambda_{i^m} - \lambda_{j^m} | $, $ i^m<j^m $. Note that then for $ k^m = ( i^m+3 )\mod 6 $, $ l^m = ( j^m + 3) \mod 6 $ we also have that
\begin{equation*}
\omega_1 = | \lambda_{k^m} - \lambda_{l^m} |.
\end{equation*} 
This is just the fact that if $ \lambda_{i^m}$ and  $\lambda_{j^m}$ realize the value $\omega_1$, then $- \lambda_{i^m}$ and  $-\lambda_{j^m}$ also realize the same value.

\medskip 

\begin{defn}\label{omega2}
We define $ \omega_2 $ to be the second minimal distance between two stationary points (excluding pairs $ ( \lambda_{i^m}, \lambda_{j^m}  ) $, $ ( \lambda_{k^m}, \lambda_{l^m}) $, $ ( \lambda_2, \lambda_5 ) $):
\begin{align*}
& \omega_2 = \min_{ P_2 } | \lambda_i - \lambda_j |, \quad \hbox{where} \\
& P_2=\Big\{ ( i, j  ) \in N_5^2 \;  | \;  i < j, \; ( i, j ) \neq ( 2, 5 ),  ( i^m, j^m ),  ( k^m, l^m ),  ( l^m, k^m ) \Big\}.
\end{align*}
\end{defn}

(Actually, only one of the pairs $( k^m, l^m )$ and $ ( l^m, k^m )$ is needed, the other does not satisfy the condition ``$i<j$''.)  

\subsection{Preliminary lemmas}

\begin{lemma}
\label{elem_lemma}
One has that 
$ \omega_1 < 2 $, $ \omega_2 < 2 $, and $ \omega_1 \leq \omega_2 $.  
\end{lemma}

\begin{proof}
It is clear that we only have to prove that $ \omega_2 < 2 $. We will treat cases \eqref{Case1}--\eqref{Case4} separately.

\medskip

In the case represented in \eqref{Case1}, it is also direct that $\omega_1=\omega_2=0.$  We also have $ \omega_1 =|\la_0-\la_1| = \omega_2=|\la_0 -\la_2|$. 

\medskip

Let us now consider the case in \eqref{Case2}. Readily we have $ \omega_1=0$, obtained with the difference $|\la_0-\la_2|$, repeated by $|\la_3-\la_5|$.
Therefore, the set of points $(i,j) \in P_2$ for which  the difference $|\la_i-\la_j|$ is different from 2 is reduced to the array
\[
i ~\overbrace{  \begin{cases} \begin{matrix}
(0,1) & {\bf X } & {\bf X} & (0,4) & {\bf X} \\
& (1,2) & (1,3) & {\bf X} & (1,5)\\
& & {\bf X} & (2,4) & {\bf X} \\
& & & (3,4) & {\bf X }\\
& & & & (4,5)
\end{matrix}
\end{cases}}^{\ds j}
\]
Now we claim that $|\la_0-\la_1|<2$ or $|\la_0-\la_4|<2$ (which implies $ \omega_2<2$). Indeed, if  we have $|\la_0-\la_1| =2$, then $|\la_0-\la_4| =|\la_0+ \la_1|=0$, leading to $ \omega_2 =0$. In order to prove this fact, notice that 
\[
4 =|\la_0-\la_1|^2 =2 -2\re (\la_0\overline{\la_1}),  
\]
which implies $\re (\la_0\overline{\la_1}) =-1$. Therefore,
\[
|\la_0-\la_4|^2 = |\la_0+ \la_1 |^2 = 2 + 2 \re (\la_0\overline{\la_1})  =0.
\]
In conclusion, we have proved $ \omega_2<2$ for the second case \eqref{Case2}.

\medskip

Now, we deal with the case \eqref{Case3}. 
Here, the argument is more geometric. Seen as points on the circle $\Ss^1$, the adjacent points $\la_j$'s define sectors of this circle (see Fig. \ref{Angles}).

\begin{figure}[!h]
\begin{center}
\begin{tikzpicture}[
	>=stealth',
	axis/.style={semithick,->},
	coord/.style={dashed, semithick},
	yscale = 1,
	xscale = 1]
	\newcommand{\xmin}{-4};
	\newcommand{\xmax}{4};
	\newcommand{\ymin}{-3};
	\newcommand{\ymax}{3};
	\newcommand{\ta}{3};
	\newcommand{\fsp}{0.2};
	\filldraw[color=light-gray1] (0,0) circle (2);
	\draw [axis] (\xmin,0) -- (\xmax,0) node [right] {$\re \la$};
	\draw [axis] (0,\ymin) -- (0,\ymax) node [below left] {$\ima \la$};
	\draw (2.2,-0.2) node [left] {$1$};
	\draw [dashed] (1.4,1.4) -- (-1.4,-1.4);
	\draw [dashed] (-1.4,1.4) -- (1.4,-1.4);
	\draw [dashed] (1.9,0.5) -- (-1.9,-0.5);
	\draw (-1.5,1.44) node [above] {$\la_2$};
	\draw (1.5,-1.44) node [below] {$\la_5$};
	\draw (1.9,0.5) node [right] {$\la_1$};
	\draw (-1.9,-0.5) node [left] {$\la_4$};
	\draw (1.5,1.4) node [above] {$\la_0$};
	\draw (-1.5,-1.4) node [below] {$\la_3$};
	\fill (1.4,1.4)  circle[radius=1pt];
	\fill (-1.4,-1.4)  circle[radius=1pt];
	\fill (-1.4,1.4)  circle[radius=1pt];
	\fill (1.4,-1.4)  circle[radius=1pt];
	\fill (1.9, 0.5)  circle[radius=1pt];
	\fill (-1.9,-0.5)  circle[radius=1pt];
	\draw (0.5,0.5) arc (45:135:0.7);
	\draw (1,0.25) arc (15:55:0.8);
\end{tikzpicture}
\end{center}
\caption{The sectors defined by the points $\la_j$ in the Case described in \eqref{Case3}.}\label{Angles}
\end{figure}
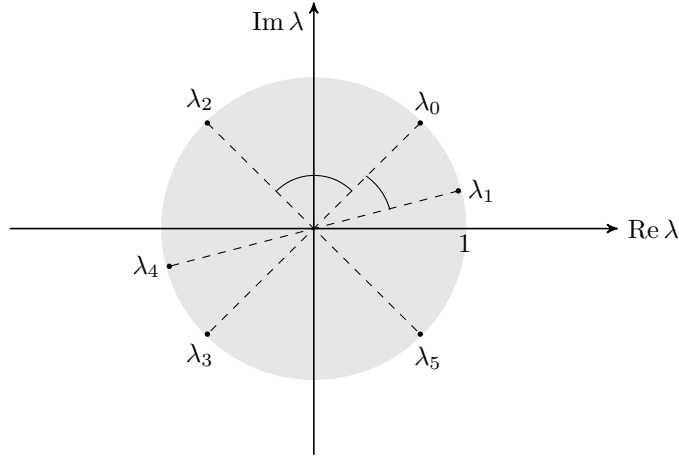

We first claim that there are at least two adjacent points that define a sector of an angle less than 60 degrees. Otherwise, the sum of all sector angles will give more than 360 degrees. The distance between two adjacent points on the circle forming a sector of the angle of 60 degrees is 1. Since $\la_2$ and $\la_5$ are not adjacent points, we conclude that  $ \omega_1<2$, attained by two adjacent points (this choice is maybe not unique). Secondly, after discarding the distance between the opposite points, equal to the same value $\omega_1$, we see that the remaining points satisfy $ \omega_2<2$, because there is at least one pair that  is not composed of two opposed points.

\medskip

Finally, we consider the Case 4 in \eqref{Case4}. Since we are considering mutual distances, after a rotation by $-\varphi/2$ angles, we can assume that (see Fig. \ref{4b})
\be\label{Case4p}
(\la_0,\la_1,\la_2 ~ ; ~  \la_3,\la_4,\la_5) = \Big( (1+\omega), e^{i\varphi_0}, \frac1{1+\omega} ~ ; ~  -(1+\omega) ,-e^{i\varphi_0},-\frac{1}{1+\omega} \Big),
\ee
for some $\varphi_0 \in \R$. Therefore we have
\[
\min_{\pm} \Big|e^{i\varphi_0}  \pm \frac1{1+\omega} \Big| \leq  \max_{\pm} \Big|e^{i\varphi_0}  \pm \frac1{1+\omega} \Big| \leq 1+ \frac1{1+\omega} <2,
\]
which implies that $ \omega_1 \leq  \omega_2 <2$.

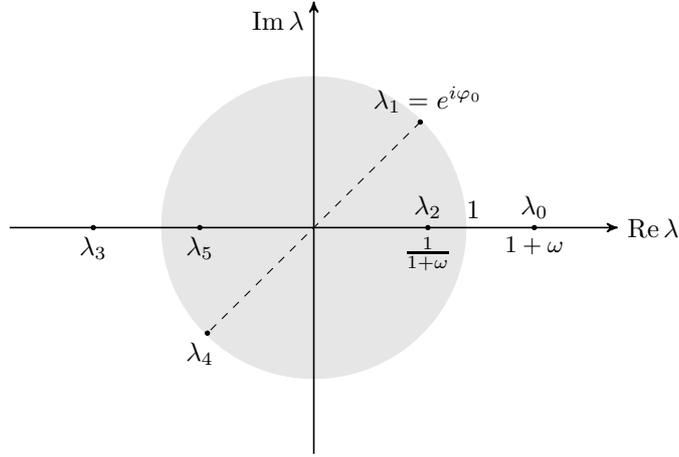
\begin{figure}[!h]
\begin{center}
\begin{tikzpicture}[
	>=stealth',
	axis/.style={semithick,->},
	coord/.style={dashed, semithick},
	yscale = 1,
	xscale = 1]
	\newcommand{\xmin}{-4};
	\newcommand{\xmax}{4};
	\newcommand{\ymin}{-3};
	\newcommand{\ymax}{3};
	\newcommand{\ta}{3};
	\newcommand{\fsp}{0.2};
	\filldraw[color=light-gray1] (0,0) circle (2);
	\draw [axis] (\xmin,0) -- (\xmax,0) node [right] {$\re \la$};
	\draw [axis] (0,\ymin) -- (0,\ymax) node [below left] {$\ima \la$};
	\draw (2.1,0) node [above] {$1$};
	\draw [dashed] (1.4,1.4) -- (-1.4,-1.4);
	\draw (1.5,0) node [above] {$\la_2$};
	\draw (-1.5, 0) node [below] {$\la_5$};
	\draw (2.9,0) node [above] {$\la_0$};
	\draw (-2.9,0) node [below] {$\la_3$};
	\draw (1.5,1.4) node [above] {$\la_1 =e^{i\varphi_0}$};
	\draw (-1.5,-1.4) node [below] {$\la_4$};
	\draw (1.5,0) node [below] {$\frac 1{1+\omega}$};
	\draw (2.9,0) node [below] {\small $1+\omega$};
	\fill (1.4,1.4)  circle[radius=1pt];
	\fill (-1.4,-1.4)  circle[radius=1pt];
	\fill (-1.5,0)  circle[radius=1pt];
	\fill (1.5,0)  circle[radius=1pt];
	\fill (2.9, 0)  circle[radius=1pt];
	\fill (-2.9,-0)  circle[radius=1pt];
\end{tikzpicture}
\end{center}
\caption{The roots described in \eqref{Case4p} and \eqref{Case4pq}.}\label{4b}
\end{figure}

\end{proof}

The next lemma shows that we can always choose $ \omega_1$ and $ \omega_2$ with ``the same base point''.

\begin{lemma}[See Figs. \ref{2cases} and \ref{2cases1}]
\label{aux_lemma}
There exist $ j_0, j_1, j_2 \in N_5 $ such that 
\be\label{trivote}
 \omega_1 = | \lambda_{ j_0 } - \lambda_{ j_1 } | \quad  \hbox{and} \quad   \omega_2 = | \lambda_{ j_0 } - \lambda_{ j_2 } | .
 \ee
\end{lemma}

\begin{proof}
For the case in \eqref{Case1} the result is obvious. The case \eqref{Case2} requires more care. One has $ \omega_1=0=|\la_0-\la_2| =|\la_3-\la_5|$,
\[
|\la_0-\la_3|=|\la_0-\la_5|=|\la_1-\la_4|=|\la_2-\la_3|=2,
\]
\be\label{subcase2a}
|\la_0-\la_1|=|\la_1-\la_2|=|\la_3-\la_4|=|\la_4-\la_5| = |e^{i\varphi/2} - e^{-i\varphi}|<2,
\ee
and
\be\label{subcase2b}
|\la_0-\la_4|=|\la_1-\la_3|=|\la_1-\la_5|=|\la_2-\la_4|=|e^{i\varphi/2} + e^{-i\varphi}|<2.
\ee
One of the last two values must be $ \omega_2$ (which one, it will depend on $\varphi\in [0,2\pi)$), a value which is attained either by $|\la_0-\la_1|$ or by $|\la_0-\la_4|$. This proves the result in this case.

\medskip

Now we deal with the more general case \eqref{Case3}. Recall that each $\la_j$, $j=0,1,2$ is different, and recall that $ \omega_1>0$. Let $ i_1, i_2, i_3, i_4 \in N_5 $ be such that $ \omega_1 = | \lambda_{ i_1 } - \lambda_{ i_2 } | $ and $ \omega_2 = | \lambda_{ i_3 } - \lambda_{ i_4 } | $. Then for 
\[
j_1 = ( i_1 + 3 ) \mod 6 ,  \quad  j_2 = ( i_2 + 3 ) \mod 6,
\]
\[
j_3 = ( i_3 + 3 ) \mod 6 , \quad j_4 = ( i_4 + 3 ) \mod 6 ,
\]
we also have that $ \omega_1 = | \lambda_{ j_1 } - \lambda_{ j_2 } | $ and $ \omega_2 = | \lambda_{ j_3 } - \lambda_{ j_4 } | $, because
\[ 
\lambda_{ j_1 } =-\lambda_{ i_1 },\quad \lambda_{ j_2 } =-\lambda_{ i_2 },\quad  \lambda_{ j_3 } = -\lambda_{ i_3}, \quad \hbox{and}\quad\lambda_{ j_4 } = -\lambda_{ i_4}.
\]
Note that on the whole we have 6 complex numbers $\la_i$, $ i \in N_5 $. Among the $ 8 $ numbers $ \lambda_{ i_k } $, $ \lambda_{ j_k } $, $ k =1, 2, 3, 4 $, there are at least two pairs of coinciding numbers. Moreover, a number from the group $ A =\{ \lambda_{ i_1 }, \lambda_{ i_2 }, \lambda_{ j_1 }, \lambda_{ j_2 } \} $ (concerning the computation of $\omega_1$) can only coincide with a number from the group $ B = \{ \lambda_{ i_3 }, \lambda_{ i_4 }, \lambda_{ j_3 }, \lambda_{ j_4 } \} $ (related to the computation of $\omega_2$) and vice versa. Indeed, take for example $ \lambda_{ i_1 } $. It is clear that it cannot coincide with $ \lambda_{ i_2 } $ ($w_1>0$), and with $ \lambda_{ j_1 } $. If it coincides with $ \lambda_{ j_2 } $, then $ \omega_1 = 2 | \lambda_{ i_1 } | = 2 $, which contradicts Lemma \ref{elem_lemma}. The remaining cases are simpler or similar to the previous example.

\medskip

Finally, we consider the Case 4  in \eqref{Case4p}. Without loss of generality, we consider the case where $e^{i\varphi_0}$ is in the closure of the first quadrant of the plane, see Fig. \ref{4b}. Then $\omega_1$ and $\omega_2$ are exactly equal to the lengths of two sides of the triangle with base points  $(1+\omega)$, $e^{i\varphi_0}$, and $\frac1{1+\omega}$, or $\frac1{1+\omega }$, $e^{i\varphi_0}$, and $-\frac1{1+ \omega }$;  from which among them there are always  $ \lambda_{ j_0 }, \lambda_{ j_1 }$, and $ \lambda_{ j_2 } $ such that $ \omega_1 = | \lambda_{ j_0 } - \lambda_{ j_1 } | $, $ \omega_2 = | \lambda_{ j_0 } - \lambda_{ j_2 } | $.

\begin{figure}[!h]
\minipage{0.45\textwidth}
\begin{tikzpicture}[
	>=stealth',
	axis/.style={semithick,->},
	coord/.style={dashed, semithick},
	yscale = 0.75,
	xscale = 0.75]
	\newcommand{\xmin}{-4};
	\newcommand{\xmax}{4};
	\newcommand{\ymin}{-3};
	\newcommand{\ymax}{3};
	\newcommand{\ta}{3};
	\newcommand{\fsp}{0.2};
	\filldraw[color=light-gray1] (0,0) circle (2);
	\draw [axis] (\xmin,0) -- (\xmax,0) node [right] {$\re \la$};
	\draw [axis] (0,\ymin) -- (0,\ymax) node [below left] {$\ima \la$};
	\draw (2.2,-0.0) node [above] {$\omega_1$};
	\draw [dashed] (1.4,1.4) -- (-1.4,-1.4);
	\draw [thick,-] (1.5,0) -- (2.9,0);
	\draw [thick,-] (1.5,0) -- (1.4,1.4);
	\draw (1.0,0) node [below] {$\frac 1{1+\omega}=\la_2$};
	\draw (-1.5, 0) node [below] {$\la_5$};
	\draw (2.9,0) node [above right] {$\la_0$};
	\draw (-2.9,0) node [below] {$\la_3$};
	\draw (1.5,1.4) node [above] {$\la_1 =e^{i\varphi_0}$};
	\draw (-1.5,-1.4) node [below] {$\la_4$};
	\draw (1.5,0.2) node [above left] {$\omega_2$};
	\draw (2.9,0) node [below] {\small $1+\omega$};
	\fill (1.4,1.4)  circle[radius=2pt];
	\fill (-1.4,-1.4)  circle[radius=2pt];
	\fill (-1.5,0)  circle[radius=2pt];
	\fill (1.5,0)  circle[radius=2pt];
	\fill (2.9, 0)  circle[radius=2pt];
	\fill (-2.9,-0)  circle[radius=2pt];
\end{tikzpicture}
\caption{In this case, $-u\not\in \mathbb U$ and $\omega_1 =|\la_2-\la_0|$ and $\omega_2=|\la_2-\la_1|$.}\label{2cases}
\endminipage\hfill
\minipage{0.47\textwidth}
\begin{tikzpicture}[
	>=stealth',
	axis/.style={semithick,->},
	coord/.style={dashed, semithick},
	yscale = 0.75,
	xscale = 0.75]
	\newcommand{\xmin}{-5};
	\newcommand{\xmax}{5};
	\newcommand{\ymin}{-3};
	\newcommand{\ymax}{3};
	\newcommand{\ta}{3};
	\newcommand{\fsp}{0.2};
	\filldraw[color=light-gray1] (0,0) circle (2);
	\draw [axis] (\xmin,0) -- (\xmax,0) node [below] {$\re \la$};
	\draw [axis] (0,\ymin) -- (0,\ymax) node [below left] {$\ima \la$};
	\draw [thick,-] (-0.5,0) -- (1.4,1.4);
	\draw [thick,-] (0.5,0) -- (1.4,1.4);
	\draw (2.2,0.0) node [above] {$1$};
	\draw [dashed] (1.4,1.4) -- (-1.4,-1.4);
	\draw (-0.5, 0) node [below left] {$\la_5$};
	\draw (4.3,0) node [above] {$\lambda_0 = 1+\omega$};
	\draw (-4.5,0) node [below] {$\la_3$};
	\draw (1.5,1.4) node [above] {$\la_1 =e^{i\varphi_0}$};
	\draw (-1.5,-1.4) node [below] {$\la_4$};
	\draw (1,0.2) node [above right] {$\omega_1$};
	\draw (1.0,0) node [below] {\small$\lambda_2=\frac 1{1+\omega}$};
	\draw (-0.1,0.7) node [above right] {$\omega_2$};
	\fill (1.4,1.4)  circle[radius=2pt];
	\fill (-1.4,-1.4)  circle[radius=2pt];
	\fill (-0.5,0)  circle[radius=2pt];
	\fill (0.5,0)  circle[radius=2pt];
	\fill (4.5, 0)  circle[radius=2pt];
	\fill (-4.5,-0)  circle[radius=2pt];
\end{tikzpicture}
\caption{In this case, $-u\not\in \mathbb U$ and $\omega_1=|\la_1-\la_2|$ and $\omega_2=|\la_1-\la_5|$.}\label{2cases1}
\endminipage
\end{figure}

\medskip

Thus we can always choose $ \lambda_{ j_0 }, \lambda_{ j_1 }, \lambda_{ j_2 } $ such that $ \omega_1 = | \lambda_{ j_0 } - \lambda_{ j_1 } | $, $ \omega_2 = | \lambda_{ j_0 } - \lambda_{ j_2 } | $.
\end{proof}

\bigskip

Now we need some sharp estimates on the positions of the stationary points $\la_j$.

\subsection{Advanced lemmas}

The first lemma of this subsection measures how far the stationary points are from each other in the simple case where the parameter $-u\in \mathbb U$.

\begin{lemma}
\label{cluster_lemma}
Set $ J_1 = \{ j_0, j_1, j_2 \} $ with $ j_0, j_1, j_2  $ of Lemma \ref{aux_lemma}, and $ J_2 = N_5 \backslash J_1 $.  Then
\begin{enumerate}
\item One has the trivial bound
\begin{equation}
\label{triv_bound}
| \lambda_i - \lambda_j | \leq 2 \omega_2 \quad \forall i, j \in J_1.  
\end{equation}
\item If $ -u \in \mathbb{U} $, then 
\begin{equation}
\label{no_coinc_1}
\forall k \in \{ 0, 1, 2 \} \text{ the point } - \lambda_{ j_k } \text{ does not coincide  with any of } \{ \lambda_{ j_0 }, \lambda_{ j_1 }, \lambda_{ j_2 } \}
\end{equation}
and
\begin{equation}
\label{other_bound}
| \lambda_i - \lambda_j | \leq 2 \omega_2 \quad \forall i, j \in J_2.  
\end{equation}
\end{enumerate}
\end{lemma}

\begin{proof}


Estimate (\ref{triv_bound}) is a consequence of the definition of $ \omega_1 $, $ \omega_2 $, the triangle inequality, and Lemmas \ref{elem_lemma}, \ref{aux_lemma}.

\medskip

Note that property \eqref{no_coinc_1} implies that the set $ \{ \lambda_k, \lambda_l, \lambda_m \} $ with $ \{ k, l, m \} = J_2 $ coincides with the set $ \{ - \lambda_{ j_0 }, - \lambda_{ j_1 }, - \lambda_{ j_2 } \} $, and thus, if \eqref{no_coinc_1} holds, then \eqref{other_bound} is just a consequence of \eqref{triv_bound} and of the invariance of the absolute value under the minus sign change. 

\medskip

Let us prove (\ref{no_coinc_1}). Assuming that $ -u \in \mathbb{U} $, we are in the setting of Cases 1, 2 and 3 described in \eqref{Case1}, \eqref{Case2} and \eqref{Case3}, respectively. 

\medskip

Assume \eqref{Case1}. In this case $ J_1 = \{ 0, 1, 2 \} $ with $ \lambda_0 = \lambda_1 = \lambda_2 $ (or $ J_1 = \{ 3, 4, 5 \} $ with $ \lambda_3 = \lambda_4 = \lambda_5 $) and $ \omega_1 = \omega_2 = 0 $.  Thus  \eqref{no_coinc_1} holds trivially. 

\medskip

Now assume \eqref{Case2}. Here we have  $J_1=\{0,1,2\}$ or $J_1=\{0,2,4\}$ with $ \lambda_0 = \lambda_2 = e^{ i \varphi / 2 } $, $ \lambda_1 = e^{ i \varphi } $, $ \lambda_4 = - e^{ i \varphi } $. (There are other choices for $J_1$, but they are symmetric to the considered cases.)  Therefore $J_2=\{3,4,5\}$ or $J_2=\{1,3,5\}$. Note also that $ \omega_1 = 0 $ in this case.

\medskip

If for some $ k \in \{ 0, 1, 2 \} $ the point $ - \lambda_{ j_k } $ coincides with one of $ \{ \lambda_{ j_0 }, \lambda_{ j_1 }, \lambda_{ j_2 } \} $, then it is only possible if $ e^{ i \varphi / 2 } = e^{ i \varphi } $ or $ e^{ i \varphi  /2 } = - e^{ i \varphi } $ (due to Lemma \ref{elem_lemma} stating that $ \omega_2 < 2 $). 

\medskip

In the first case $ \varphi = \frac{ 4 \pi k }{ 3 } $, in the second case $ \varphi = \frac{ 4 \pi k }{ 3 } + \frac{ 2 \pi }{ 3 } $, $ k = 0, 1, 2 $. In both cases we see that the corresponding value of parameter $ u $ is $ - u = 18 e^{ \frac{ 2 \pi i n }{ 3 } } $, $ n \in \mathbb{N} $, which cannot be possible in the case described by \eqref{Case2} (see Lemma \ref{st_points_lemma}).

\medskip


Now we deal with the general case, where \eqref{Case3} holds. 
Recall that $ \omega_ 1= | \lambda_{ j_0 } - \lambda_{ j_1 } | $, $ \omega_2 = | \lambda_{ j_0 } - \lambda_{ j_2 } | $. Due to Lemma \ref{elem_lemma}, if for some $ k \in \{ 0, 1, 2 \} $ the point $ - \lambda_{ j_k } $ coincides with one of $ \{ \lambda_{ j_0 }, \lambda_{ j_1 }, \lambda_{ j_2 } \} $, then it is only possible that $ - \lambda_{ j_1 } = - \lambda_{ j_2 } $. Points $ ( \lambda_{ j_0 }, -\lambda_{ j_0 }, \lambda_{ j_1 }, - \lambda_{ j_1 } ) $ are four different stationary points. There exists also another pair $ ( \lambda_k, -\lambda_k ) $ of stationary points, different from all the above four points. 

\medskip

Denote by $ \mathbb{S}^1_+ $ the semicircle defined by points $ \lambda_{ j_0 } $, $ \lambda_{ j_1 } $, $ \lambda_{ j_2 } = - \lambda_{ j_1 } $. One of the points $ ( \lambda_k, - \lambda_k ) $, say $ \lambda_k $ for definiteness, necessarily belongs to $ \mathbb{S}^1_+ $. Then we evidently have that 
\begin{equation*}
| \lambda_{ j_0 } - \lambda_k | < | \lambda_{ j_0 } - \lambda_{ j_2 } | = | \lambda_{ j_0 } + \lambda_{ j_1 } | = \omega_2,
\end{equation*} 
which contradicts the definition of $ \omega_2 $ (see also Fig. \ref{semicircle}).

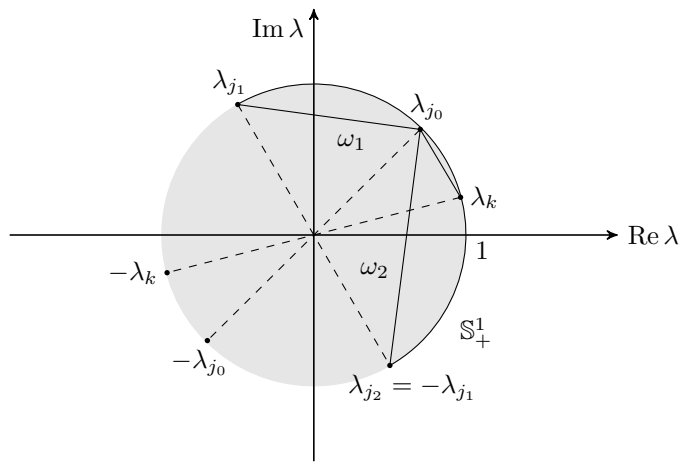
\begin{figure}[!h]
\begin{center}
\begin{tikzpicture}[
	>=stealth',
	axis/.style={semithick,->},
	coord/.style={dashed, semithick},
	yscale = 1,
	xscale = 1]
	\newcommand{\xmin}{-4};
	\newcommand{\xmax}{4};
	\newcommand{\ymin}{-3};
	\newcommand{\ymax}{3};
	\newcommand{\ta}{3};
	\newcommand{\fsp}{0.2};
	\filldraw[color=light-gray1] (0,0) circle (2);
	\draw [axis] (\xmin,0) -- (\xmax,0) node [right] {$\re \la$};
	\draw [axis] (0,\ymin) -- (0,\ymax) node [below left] {$\ima \la$};
	\draw (2,-0.2) node [right] {$1$};
	\draw [dashed] (1.4,1.4) -- (-1.4,-1.4);
	\draw [dashed] (-1,1.73) -- (1,-1.73);
	\draw [dashed] (1.93,0.5) -- (-1.93,-0.5);
	\draw (1.4,1.4) -- (-1,1.73);
	\draw (1.4,1.4) -- (1,-1.73);
	\draw (1.4,1.4) -- (1.93,0.5);
	\draw (1.5,1.4) node [above] {$\la_{j_0}$};
	\draw (-1.5,-1.4) node [below] {$-\la_{j_0}$};
	\draw (-1.1,1.73) node [above] {$\la_{j_1}$};
	\draw (1.3,-1.73) node [below] {$\la_{j_2} = -\la_{j_1}$};
	\draw (1.93,0.5) node [right] {$\la_{k}$};
	\draw (-1.93,-0.5) node [left] {$-\la_{k}$};
	\draw (0.5,1) node [above] {$\omega_1$};
	\draw (0.8,-0.7) node [above] {$\omega_2$};
	\draw (1.8,-1.3) node [right] {$\mathbb{S}^1_+$};
	\fill (1.4,1.4)  circle[radius=1pt];
	\fill (-1.4,-1.4)  circle[radius=1pt];
	\fill (-1,1.73)  circle[radius=1pt];
	\fill (1,-1.73)  circle[radius=1pt];
	\fill (1.93, 0.5)  circle[radius=1pt];
	\fill (-1.93,-0.5)  circle[radius=1pt];
	\draw (1,-1.73) arc (-60:120:2);
\end{tikzpicture}
\end{center}
\caption{The configuration of roots in the case \eqref{Case3}. It is impossible that $ -\lambda_{j_1} = \lambda_{j_2} $, since in that case there is always another stationary point $ \lambda_k $ on the semicircle $ \mathbb{S}^1_+ $ defined by the points $ \lambda_{j_0}, \lambda_{j_1}, \lambda_{j_2} = - \lambda_{j_1} $ such that the distance from $ \lambda_{ j_0 } $ to $ \lambda_k $ is smaller than the distance from $ \lambda_{j_0} $ to $ \lambda_{j_2} $.}\label{semicircle}
\end{figure}

\end{proof}

\begin{lemma}\label{cluster_lemma_1}
Consider, as in the previous lemma, $ J_1 = \{ j_0, j_1, j_2 \} $ with $ j_0, j_1, j_2  $ exactly as in Lemma \ref{aux_lemma}, and $ J_2 = N_5 \backslash J_1 $.  If now $ -u \in \mathbb{C} \backslash \mathbb{U} $ and for all $j\in N_5$ we have that $ \lambda_j \in B_K( 0 ) $, with $ K = \sqrt{ \sqrt{2} + 1 } $, then
\eqref{no_coinc_1} and  (\ref{other_bound}) hold.

\end{lemma}

\begin{proof}
Let us assume that $ J_2$ is given by the set $J_2 = \{ k_0, k_1, k_2 \} $. Note that by the symmetry $ \lambda \mapsto - \lambda $ of the roots, the property \eqref{no_coinc_1} can be restated as follows: 
\begin{equation}
\label{no_coinc_2}
\forall i \in\{ 0, 1, 2 \} \text{ the point } - \lambda_{ k_i } \text{ does not coincide with any of } \lambda_{ k_0 }, \lambda_{ k_1 }, \lambda_{ k_2 }.
\end{equation}
If \eqref{no_coinc_2} holds, then it means that the set $ \{ - \lambda_{ k_0 }, - \lambda_{ k_1 }, - \lambda_{ k_2 }  \} $ coincides with the set $ \{ \lambda_{ j_0 }, \lambda_{ j_1 }, \lambda_{ j_2 }  \} $ and thus, in this case, (\ref{other_bound}) is just a consequence of (\ref{triv_bound}) and the invariance of the absolute value under the minus sign change.

\medskip

Suppose now the general case. Although $ - \lambda_{ k_i } $ may coincide with one of the roots $ \{ \lambda_{ k_0 }, \lambda_{ k_1 }, \lambda_{ k_2 } \} $ if $\omega$ is large enough (see Fig. \ref{weird_case}), we will see that this does not happen for $\omega<K-1$.

\begin{figure}[!h]
\begin{center}
\begin{tikzpicture}[
	>=stealth',
	axis/.style={semithick,->},
	coord/.style={dashed, semithick},
	yscale = 1,
	xscale = 1]
	\newcommand{\xmin}{-5};
	\newcommand{\xmax}{5};
	\newcommand{\ymin}{-3};
	\newcommand{\ymax}{3};
	\newcommand{\ta}{3};
	\newcommand{\fsp}{0.2};
	\filldraw[color=light-gray1] (0,0) circle (2);
	\draw [axis] (\xmin,0) -- (\xmax,0) node [right] {$\re \la$};
	\draw [axis] (0,\ymin) -- (0,\ymax) node [below left] {$\ima \la$};
	\draw (2.2,-0.2) node [left] {$1$};
	\draw [dashed] (1.4,1.4) -- (-1.4,-1.4);
	\draw (0.5,0) node [above] {$\la_2$};
	\draw (-0.5, 0) node [below] {$\la_5$};
	\draw (4.5,0) node [above] {$\la_0$};
	\draw (-4.5,0) node [below] {$\la_3$};
	\draw (1.5,1.4) node [above] {$\la_1 =e^{i\varphi_0}$};
	\draw (-1.5,-1.4) node [below] {$\la_4$};
	\draw (0.5,0) node [below] {\small$\frac 1{1+\omega}$};
	\draw (4.5,0) node [below] {\small $1+\omega$};
	\fill (1.4,1.4)  circle[radius=1pt];
	\fill (-1.4,-1.4)  circle[radius=1pt];
	\fill (-0.5,0)  circle[radius=1pt];
	\fill (0.5,0)  circle[radius=1pt];
	\fill (4.5, 0)  circle[radius=1pt];
	\fill (-4.5,-0)  circle[radius=1pt];
\end{tikzpicture}
\end{center}
\caption{The configuration of roots in the case where $\omega$ is very large. In this case $\omega_1 = |\la_1-\la_2|$ and $\omega_2= |\la_1-\la_5|$, so that we can choose $J_1=\{1,2,5\}$ and $J_2=\{0,3,4\}$. In particular, $-\la_3=\la_0$, meaning that the opposite point to $\la_3$ coincide with $\la_0$, with $0\in J_2$. Note also that $-\la_{5}=\la_2$, with both $2$ and $5$ indexes in $J_1$. This phenomenon is avoided if we choose $\omega$ not too large.}\label{weird_case}
\end{figure}
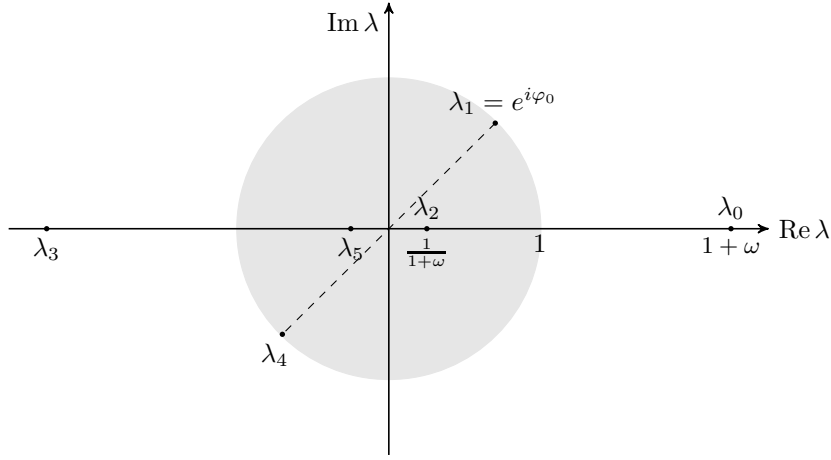

\medskip

Suppose now that 
\begin{center}
($\star$) \quad There exists some $ i \in \{ 0, 1, 2 \} $ such that $ - \lambda_{ k_i } $ coincides with one of the points $ \{ \lambda_{ k_0 }, \lambda_{ k_1 }, \lambda_{ k_2 } \} $.
\end{center}
(This is the situation of Fig. \ref{2cases1}, for example.) Recall that from \eqref{Case4p} we can assume that the points are of the form
\be\label{Case4pq}
(\la_0,\la_1,\la_2 ~ ; ~ \la_3,\la_4,\la_5) = \Big( (1+\omega), e^{i\varphi_0}, \frac{1}{1+\omega}  ~ ; ~ -(1+\omega) ,-e^{i\varphi_0},-\frac{1}{1+\omega} \Big),
\ee
for some $\varphi_0\in \R$, see Fig. \ref{4b}.  Also, with no loss of generality, we may assume $\varphi_0 \in [-\frac\pi2,\frac\pi2]$.

\medskip

In view of ($\star$), by the symmetry $\la \mapsto -\la$ of the roots there exists also $   r \in \{ 0, 1, 2 \} $ such that $- \lambda_{ j_r } $ coincides with one of the points $ \{ \lambda_{ j_0 }, \lambda_{ j_1 }, \lambda_{ j_2 } \} $ (In Fig. \ref{2cases1} for example, $j_0=1$, $j_1=2$ and $j_2=5$, and $-\la_{j_2} = -\la_5 = \la_2 = \la_{j_1}$). 
\begin{Cl}\label{Claim00}
We can only have that $ - \lambda_{ j_1 } = \lambda_{ j_2 } $ (or, equivalently, $ - \lambda_{ j_2 } = \lambda_{ j_1 } $). In consequence, 
\begin{equation}
\label{coinc_relation}
\omega_1 = | \lambda_{ j_0 } - \lambda_{ j_1 } |, \quad \omega_2 = | \lambda_{ j_0 } + \lambda_{ j_1 } |. 
\end{equation} 
\end{Cl}
\begin{proof}
Assume that $r=0$. If $-\la_{j_0} = \la_{j_1}$, then from Lemma \ref{aux_lemma} one has $\omega_1 =|\la_{j_0} + \la_{j_0}| =2$, a contradiction of Lemma \ref{elem_lemma}. If $-\la_{j_0} = \la_{j_2}$, using Lemma \ref{aux_lemma} we also obtain a contradiction to Lemma \ref{elem_lemma}. In conclusion, $-\la_{j_1} = \la_{j_2}$.
\end{proof}

%
%
%
Recall that $ -u \in \mathbb{C} \backslash \mathbb{U} $ and for all $j\in N_5$ we have that $ \lambda_j \in B_K( 0 ) $, with $ K = \sqrt{ \sqrt{2} + 1 } $. Then, for $ \omega $ given by Lemma \ref{st_points_lemma}, item 4, we have
\be\label{w_K}
\begin{aligned}
 \omega & < K -1=\sqrt{ \sqrt{2} + 1 } - 1>0, \\
\omega^2 +2\omega & <  ( \sqrt{2} + 1 -2\sqrt{ \sqrt{2} + 1 } +1 )+2\sqrt{ \sqrt{2} + 1 } -2 =\sqrt{2}.
\end{aligned}
\ee
Let 
\be\label{d1_d2}
 d_1 := | \lambda_0 - \lambda_2 | , \quad \hbox{and} \quad d_2 := \min( | \lambda_2 - \lambda_1 |, | \lambda_2 + \lambda_1 | ).
\ee
We also define
\be\label{d12_d01_d02}
\begin{aligned}
& d_{ 12 } := \max( | \lambda_2 - \lambda_1 |, | \lambda_2 + \lambda_1 | ) , \\
& d_{ 01 } := \max( | \lambda_1 - \lambda_0 |, | \lambda_1 + \lambda_0 |  ) , \\
& d_{ 02 } := \max(  | \lambda_0 - \lambda_2 |, | \lambda_0 + \lambda_2 |  ) .
\end{aligned}
\ee
We prove now two simple claims.

\begin{Cl}\label{Claim0}
One has
\[
\max( d_1, d_2 ) \geq \omega_2.
\]

\begin{proof}
We have $d_2 \geq \omega_1$, $d_1 \geq \omega_1$, so that  $\max( d_1, d_2 ) \geq \omega_2$, since $\omega_2$ is the second minimal distance (maybe equal to $\omega_1$) between pairs of points which are not of the form $(\la_2,-\la_2)$. 
\end{proof}

\end{Cl}
\begin{Cl}\label{Claim1}
One of the values $d_{01}, d_{12}, d_{02}$ above {\bf must coincide} with $\omega_2$. 
\end{Cl}

\begin{proof}
This is a consequence of \eqref{coinc_relation} and the definitions in \eqref{d12_d01_d02}.
\end{proof}

Clearly
\[
d_{ 12 }  ~{\color{black} \geq  }~ d_2.
\]
However, we can have a better estimate: 
\be\label{d12_d2}
d_{12}>d_2.
\ee
Indeed, if $d_{12} =d_2$, then $| \lambda_2 - \lambda_1 | =| \lambda_2 + \lambda_1 |$, which implies that 
\be\label{square0}
\re (\la_1 \overline{\la}_2) =0.
\ee
Since $\la_2$ is real-valued and nonzero, we get $\re (\la_1) =0$, which implies that $\la_1$ is either $\pm i$, see Fig. \ref{square}.
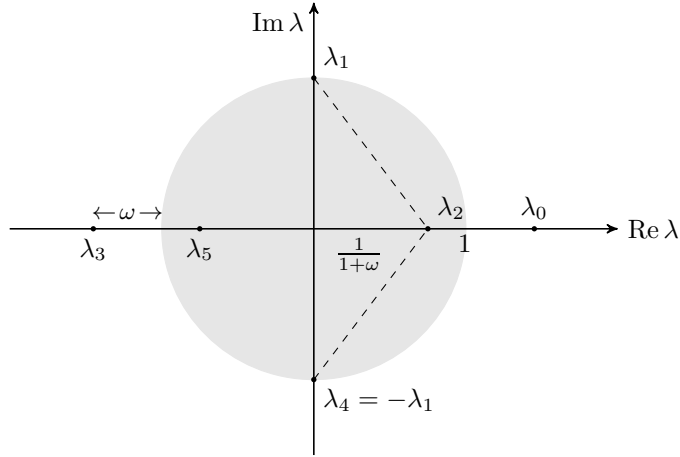
\begin{figure}[!h]
\begin{center}
\begin{tikzpicture}[
	>=stealth',
	axis/.style={semithick,->},
	coord/.style={dashed, semithick},
	yscale = 1,
	xscale = 1]
	\newcommand{\xmin}{-4};
	\newcommand{\xmax}{4};
	\newcommand{\ymin}{-3};
	\newcommand{\ymax}{3};
	\newcommand{\ta}{3};
	\newcommand{\fsp}{0.2};
	\filldraw[color=light-gray1] (0,0) circle (2);
	\draw [axis] (\xmin,0) -- (\xmax,0) node [right] {$\re \la$};
	\draw [axis] (0,\ymin) -- (0,\ymax) node [below left] {$\ima \la$};
	\draw (2.2,-0.2) node [left] {$1$};
	\draw [dashed] (0,2) -- (1.5,0);
	\draw [dashed] (0,-2) -- (1.5,0);
	\draw (1.5,0) node [above right] {$\la_2$};
	\draw (-1.5, 0) node [below] {$\la_5$};
	\draw (2.9,0) node [above] {$\la_0$};
	\draw (-2.9,0) node [below] {$\la_3$};
	\draw (0,2) node [above right] {$\la_1$};
	\draw (0,-2) node [below right] {$\la_4 =-\la_1$};
	\draw (-2.45,0) node [above] {\small $\leftarrow\!\omega\!\rightarrow$};
	\draw (0.6,0) node [below] {$\frac1{1+\omega}$};
	\fill (0,2)  circle[radius=1pt];
	\fill (0,-2)  circle[radius=1pt];
	\fill (-1.5,0)  circle[radius=1pt];
	\fill (1.5,0)  circle[radius=1pt];
	\fill (2.9, 0)  circle[radius=1pt];
	\fill (-2.9,-0)  circle[radius=1pt];
\end{tikzpicture}
\end{center}
\caption{The roots described in the case where \eqref{square0} holds.}\label{square}
\end{figure}

In this case,
\[
| \lambda_0 - \lambda_2 | = 1+\omega -\frac1{1+\omega},  
\]
and
\be\label{la_1la_2}
|\la_1-\la_2| = \sqrt{1+ \frac1{(1+\omega)^2}} = |\la_1+ \la_2| .
\ee
However, from \eqref{w_K},
\be\label{Aux_11}
1 + \omega - \frac{ 1 }{ 1 + \omega } = \frac{ \omega^2 + 2 \omega }{ 1 + \omega } = \frac{ \omega( \omega + 2 ) }{ 1 + \omega }  < \frac{ \sqrt{2} }{ 1 + \omega },
\ee
but on the other hand, using \eqref{la_1la_2} and the fact that $\omega >0$,
\be\label{Aux_22}
|\la_1-\la_2|  = \sqrt{ 1 + \frac{ 1 }{ ( 1 + \omega )^2 } } = \frac{\sqrt{1+(1+\omega )^2}}{1+\omega}> \frac{ \sqrt{2} }{ 1 + \omega },
\ee
which implies that 
\[
| \lambda_0 - \lambda_2 | < | \lambda_1- \lambda_2 | ,
\]
but also (recall that both $\omega_1$ and $\omega_2$ do not take into account the difference $|\la_2-\la_5| =\frac 2{1+\omega}$),
\[
\omega_1=| \lambda_0 - \lambda_2 | , \qquad  | \lambda_1- \lambda_2 | =\omega_2<2.
\]
Consequently, $ J_1 = \{ 0, 1, 2 \} $, $J_2 =\{3,4,5\}$, and $\{-\la_3,-\la_4,-\la_5\} = \{\la_0,\la_1,\la_2 \} = J_1$, so they do not coincide with  $\{\la_3,\la_4,\la_5\}$. This is a contradiction to our main assumption, and \eqref{d12_d2} holds.

Additionally, from \eqref{Aux_11} and \eqref{Aux_22},
\[
d_1  < \frac{ \sqrt{2} }{ 1 + \omega }, \quad d_{ 12 }  > d_1.
\]
From this point and \eqref{d12_d2} we conclude that 
\[
d_{12} >\max(d_1,d_2).
\]

\medskip

Now we claim that 
\[
 d_{ 01 } ~ {\color{black}>}~ \max( d_1, d_2 ).
\]

Indeed, this fact follows from the following chains of inequalities:
\begin{gather*}
d_{01} \geq \sqrt{ 1 + ( 1 + \omega )^2 } > 1 + \omega > 1 + \omega - \frac{ 1 }{ 1 + \omega } = d_1, \\
d_{01} \geq \sqrt{ 1 + ( 1 + \omega )^2 } > \sqrt{2} > d_2
\end{gather*}
Thus, we obtain that
\be\label{d01d1d2}
 d_{ 01 } ~ {\color{black}>}~ \max( d_1, d_2 ).
\ee

A similar computation as in the previous steps, leads to the following result. We have
\[
| \lambda_0 - \lambda_2 | = 1+w -\frac1{1+w} ,  \quad  | \lambda_0 + \lambda_2 | = 1+w + \frac1{1+w}, 
\]
so that 
\be\label{d02_d1}
d_{ 02 }= | \lambda_0 + \lambda_2 | > d_1.
\ee

On the other hand,
\[
\begin{aligned}
d_{ 02 } & = 1 + \omega + \frac{ 1 }{ 1 + \omega } \\
& = ( 1 + \omega ) \left( 1 + \frac{ 1 }{ ( 1 + \omega )^2 } \right) \\
& > \sqrt{ 1 + \frac{ 1 }{ ( 1 + \omega )^2 } }.
\end{aligned}
\]
Now we use the fact that $d_2$ in \eqref{d1_d2} is always bounded above by $\sqrt{ 1 + \frac{ 1 }{ ( 1 + \omega )^2 } }$ (the case where $\la_1=-\la_4 =e^{i\pi/2}$) to conclude that 
\[
d_2 < d_{02}.
\]
Due to \eqref{d02_d1}, we conclude that
\[
d_{02}>\max(d_1,d_2).
\]
Thus we have that $d_{12}, d_{01}$ and $d_{02}$, defined in \eqref{d12_d01_d02}, satisfy the inequalities
\[
d_{ 12 } > \max( d_1, d_2 ),\quad d_{ 01 } > \max( d_1, d_2 ), \quad \hbox{and} \quad d_{ 02 } > \max( d_1, d_2 ),
\]
However, from Claims \ref{Claim0} and \ref{Claim1} one of the values $d_{01}, d_{12}, d_{02}$ must coincide with $\omega_2$, which is $\max( d_1, d_2 )$, a contradiction. 
\end{proof}

\begin{cor}
\label{cor_coinc}
If one of the following holds:
\begin{enumerate}
\item $ - u \in \mathbb{U} $,
\item $ - u \in \mathbb{C} \backslash \mathbb{U} $ and $ \lambda_j \in B_K( 0 ) $ with $ K = \sqrt{ \sqrt{2} + 1 } $ for all $ j \in N_5 $, 
\end{enumerate}
then $ \forall i \in J_2 $ the point $ - \lambda_i $ coincides with a point $ \lambda_j $, where $ j \in J_1  $.
\end{cor}
This corollary is just a restatement of property \eqref{no_coinc_1} of Lemmas \ref{cluster_lemma} and \ref{cluster_lemma_1}.

\medskip
%
%

Now we derive an estimate on the distance between the stationary points and the unit circle. Recall that $\Ss$ is the unit circle in the complex plane, i.e $ \Ss^1 =\{ \lambda \in \mathbb{C}: | \lambda | = 1 \} $, and $\mathbb U$ is the set defined in \eqref{UU}.

\begin{lemma}
\label{dist_unit_circle_lemma} 
Suppose that $ \lambda_j \in B_K( 0 ) $, for all $ j \in N_5 $, and with $K$ as in Lemma \ref{cluster_lemma_1}. Then there exists  $C = C( K )>0 $ such that 
\begin{equation*}
\dist( \lambda_j, \Ss^1 ) \leq C \omega_1, \quad \forall j \in N_5.
\end{equation*}
\end{lemma}

\begin{proof}
We proceed by splitting the proof into different cases. 

\medskip

Fist of all, if $ -u \in\mathbb{U} $, then from Lemma \ref{st_points_lemma}, for each $j \in N_5 $ one has $ \lambda_j \in \Ss^1 $. Hence, the statement of the Lemma is trivially true.

\medskip

Let now $ -u \in \mathbb{C} \backslash \mathbb{U} $. From \eqref{la_j} with $\widetilde u=-u$, and the simplification made in  \eqref{Case4p}, $ \lambda_1$ and $\lambda_4 =-\la_1$ are in $ \Ss^1 $. Moreover,
\[
 | \lambda_0 - \lambda_1 |  \geq    | \lambda_2 - \lambda_1 |,
\]
and
\[
 | \lambda_0 - \lambda_4 |    \geq | \lambda_2 - \lambda_4 |.
\]
Indeed, using \eqref{Case4p}, we are led to show that 
\[
 | 1+\omega \mp e^{i\varphi_0} |^2  \geq    \Big|\frac1{1+\omega} \mp e^{i\varphi_0} \Big|^2.
\] 
Expanding the squares, we get
\[
(1+\omega \mp \cos\varphi_0)^2 + \sin^2 \varphi_0 \geq  \Big(\frac1{1+\omega} \mp \cos \varphi_0\Big)^2 + \sin\varphi_0^2,
\] 
which leads to 
\[
(1+\omega)^2 \mp 2 (1+\omega) \cos \varphi_0 \geq \frac1{(1+\omega)^2} \mp 2\frac{\cos\varphi_0}{1+\omega}.
\]
We have then
\[
\Big(1+\omega +\frac1{1+\omega} \Big) \Big(1+\omega -\frac1{1+\omega} \Big)  \geq \pm 2 \Big(1+\omega - \frac1{1+\omega} \Big)  \cos\varphi_0,
\]
or
\[
1+\omega +\frac1{1+\omega} \geq \pm 2\cos\varphi_0,
\]
which is evidently true. Moreover, the left hand side above is always greater than 2, because $\omega>0$. Therefore,
\[
 | \lambda_0 - \lambda_1 | > | \lambda_2 - \lambda_1 | , \qquad \hbox{and}\qquad   | \lambda_0 - \lambda_4 | > | \lambda_2 - \lambda_4 |. 
\]
Now it is not difficult to check (by using symmetry and previous estimates) that $\omega_1$ is among the following distances:
\[
|\la_0 -\la_2|, \quad |\la_1-\la_2|,  \quad |\la_2-\la_4|.
\]
In the first case it is clear that $ \dist( \lambda_j, \Ss^1 ) \leq \omega_1 $, $ j = 0, 2, 3$ and $5 $, enough to conclude. 

\medskip

In the other two cases we only have that $ \dist( \lambda_j, \Ss^1 ) \leq \omega_1 $, for $ j = 2, 5 $. However, for $i=0$ or $i=3$, 
\[
\begin{aligned}
\dist( \lambda_i, \Ss^1 ) & = \omega\\
&  \leq   \frac{ K\omega }{ 1 + \omega } \qquad \hbox{(see \eqref{w_K}),} \\
& = K \dist( \lambda_j, \Ss^1 ), \qquad  j=2,5, \\
&  \leq K \omega_1,
\end{aligned}
\]
where $ \omega $ is defined in Lemma \ref{st_points_lemma}, item 4. The proof is complete.
\end{proof}

\begin{lemma}
\label{dist_deg_points_lemma}
Suppose that $ \lambda_j \in B_K( 0 ) $ for all $ j \in N_5 $ and for $K$ given in Lemma \ref{cluster_lemma_1}. Then there exists $ C = C( K ) $ such that the following holds: for all $j\in N_5$  one can find $k=k(j)\in N_5$ satisfying
\begin{equation*}
| \lambda_j - \lambda_k^* | \leq C \omega_2,
\end{equation*}
where $\la_k^*$ are the degenerate stationary points defined in \eqref{deg_points}.
\end{lemma}

\begin{proof}

We start by noticing that the roots $ \zeta_0 $, $ \zeta_1 $, $ \zeta_2 $ of $ Q( u, \zeta ) $ in \eqref{Q_u} are the corresponding roots of equation
\begin{equation*}
\zeta^3 - \frac{ \bar u }{ 6 } \zeta^2 + \frac{ u }{ 6 } \zeta - 1 = 0,
\end{equation*}
thus we have that $ \zeta_0 \zeta_1 \zeta_2 = 1 $. We will suppose that the complex-valued roots $ \lambda_0 = \sqrt{ \zeta_0 } $, $ \lambda_1 =\sqrt{ \zeta_1 } $ and $ \lambda_2 = \sqrt{ \zeta_2 } $ are taken in such a way that
\begin{equation}
\label{one_equality}
\lambda_0 \lambda_1 \lambda_2 = 1.
\end{equation}
Choose $ j_0, j_1, j_2 $ as in Lemma \ref{aux_lemma}. Define $ J_1 = \{ 0, 1, 2 \} $, $ J_2 = \{ 3, 4, 5 \} $. Note that at least two out of three indexes $ j_0, j_1, j_2 $ belong to the same group $ J_k $, $ k =1 $ or $ k =2 $. 

\medskip

We start by considering the case when all of the three values belong to the same group $ J_k $. Without loss of generality we can suppose that they belong to $ J_1 $. Then, due to (\ref{one_equality}), we have that
\begin{equation}
\label{new_one_eq}
\lambda_{ j_0 } \lambda_{ j_1 } \lambda_{ j_2 } = 1.
\end{equation}
Let $ \lambda_{ j_0 } = e^{ i \varphi } + \eta $, for some $\eta\in \Com$. Due to Lemma \ref{dist_unit_circle_lemma} we have that there exists $ C =  C( K )>0 $ such that  $ | \eta | \leq C  \omega_1  $. If we put $ \lambda_{ j_1 } = e^{ i \varphi } + \eta + \xi_1 $ , $ \lambda_{ j_2 } = e^{ i \varphi } + \eta + \xi_2 $, then 
\[
\lambda_{ j_1 } = \lambda_{ j_0 } +\xi_1, \quad \la_{j_2} = \la_{j_0} +\xi_2,
\]
and from Lemma \ref{aux_lemma}, $ | \xi_1 | = \omega_1 $, $ | \xi_2 | = \omega_2 $. Now, from (\ref{new_one_eq}) we have that there is a constant $  \tilde C( K )>0 $ such that
\[
| e^{ 3 i \varphi } - 1|  \leq  \tilde C \omega_2.
\]
Indeed,
\begin{equation*}
\begin{aligned}
| e^{ 3 i \varphi } - 1 | & = | e^{ 3 i \varphi } -\lambda_{ j_0 } \lambda_{ j_1 } \lambda_{ j_2 } |\\
& =| e^{ 3 i \varphi } -\lambda_{ j_0 }( \lambda_{ j_0 } +\xi_1) (  \lambda_{ j_0 } +\xi_2 ) | \\
& =| e^{ 3 i \varphi } -\lambda_{ j_0 }^3 -(\xi_1+\xi_2) \la_{j_0}^2 -\xi_1\xi_2 \la_{j_0} | \\
& = | e^{ 3 i \varphi } -(e^{ i \varphi } + \eta)^3 -(\xi_1+\xi_2) (e^{ i \varphi } + \eta)^2 -\xi_1\xi_2 (e^{ i \varphi } + \eta) |    \\
& = | e^{ 3 i \varphi } - e^{ 3i \varphi } -3e^{2i\varphi} \eta - 3e^{i\varphi} \eta^2  -\eta^3 -(\xi_1+\xi_2) (e^{ 2i \varphi }+ 2e^{i\varphi} \eta + \eta^2) -\xi_1\xi_2 (e^{ i \varphi } + \eta) |    \\
& \leq  3|\eta| + 3 \eta^2  +|\eta|^3 +(|\xi_1|+|\xi_2|) (1+ 2|\eta| + \eta^2) +|\xi_1||\xi_2| (1 + |\eta|)     \\
& \leq  C\omega_1 + C \omega_1^2  +C\omega_1^3 +(\omega_1+\omega_2) (1+ C\omega_1 + C\omega_1^2) +\omega_1\omega_2 (1 + C\omega_1)     \\
& \leq C_1\omega_2 + C_1\omega_2^2 + C _1\omega_2^3   \\
&\leq \tilde C \omega_2 \qquad \hbox{(using Lemma \ref{elem_lemma})},
\end{aligned}
\end{equation*}
where $ C_1 $ is some intermediate constant depending only on $  K$.

On the other hand, 
\begin{equation*}
| e^{ 3 i \varphi } - 1 | = | ( e^{ i \varphi } - 1 ) ( e^{ i \varphi } - e^{ \frac{ 2 i \pi }{ 3 } } ) ( e^{ i \varphi } - e^{ \frac{ 4 i \pi }{ 3 } } ) |.
\end{equation*}
Denote by $ d $ the minimum of distances between $  e^{ i \varphi }$ and $\{ 1,  e^{ \frac{ 2 i \pi }{ 3 } }, e^{ \frac{ 4 i \pi }{ 3 } } \}$. Without loss of generality, we can assume that $d= |e^{ i \varphi } - 1|$, and $\varphi \in [-\frac\pi3,\frac\pi3]$, see Fig. \ref{d_min}. 

\begin{figure}[!h]
\begin{center}
\begin{tikzpicture}[
	>=stealth',
	axis/.style={semithick,->},
	coord/.style={dashed, semithick},
	yscale = 1,
	xscale = 1]
	\newcommand{\xmin}{-4};
	\newcommand{\xmax}{4};
	\newcommand{\ymin}{-3};
	\newcommand{\ymax}{3};
	\newcommand{\ta}{3};
	\newcommand{\fsp}{0.2};
	\filldraw[color=light-gray1] (0,0) circle (2);
	\draw [axis] (\xmin,0) -- (\xmax,0) node [right] {$\re \la$};
	\draw [axis] (0,\ymin) -- (0,\ymax) node [below left] {$\ima \la$};
	\draw (2.2,-0.2) node [below left] {$1$};
	\draw [dashed] (2,0) -- (-1,-1.73);
	\draw [dashed] (2,0) -- (-1,1.73);
	\draw [dashed] (-1,1.73) -- (-1,-1.73);
	\draw [dashed] (0,0) -- (1,1.73);
	\draw [dashed] (0,0) -- (1,-1.73);
	\draw [thick] (1.4,1.4) -- (2,0);
	\draw (-1.2,1.44) node [above left] {$e^{2i\pi/3}$};
	\draw (-1.2,-1.44) node [below left] {$e^{4i\pi/3}$};
	\draw (0.2,0.05) node [above right] {\small$\frac\pi3$};
	\draw (1.5,1.4) node [above right] {$e^{i\varphi}$};
	\draw (1.7,0.7) node [right] {$d$};
	\fill (1.4,1.4)  circle[radius=2pt];
	\fill (-1, -1.73)  circle[radius=2pt];
	\fill (-1, 1.73)  circle[radius=2pt];
	\fill (2,0)  circle[radius=2pt];
\end{tikzpicture}
\end{center}
\caption{In this figure, $d$ is the minimal distance between $  e^{ i \varphi }$ and $\{ 1,  e^{ \frac{ 2 i \pi }{ 3 } }, e^{ \frac{ 4 i \pi }{ 3 } } \}$.}\label{d_min}
\end{figure}
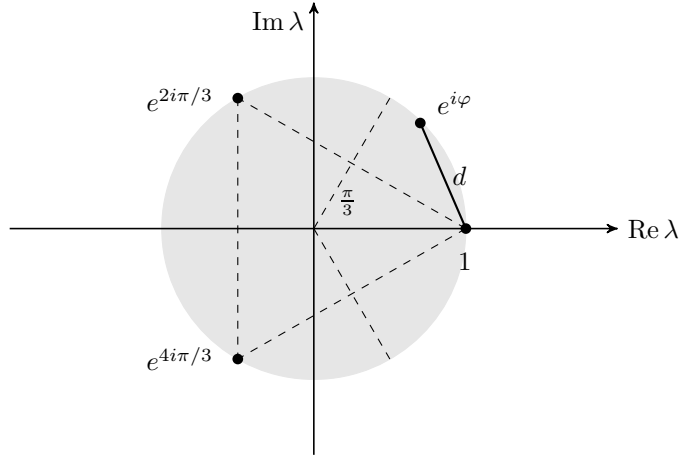
Then
\[
|e^{ i \varphi } - e^{ \frac{ 2 i \pi }{ 3 } }| \geq 1, \qquad |e^{ i \varphi } - e^{ \frac{ 4 i \pi }{ 3 } }|\geq 1,
\]
with equality if either $\varphi =\pm \frac\pi3$. Consequently,
\begin{equation*}
| ( e^{ i \varphi } - 1 ) ( e^{ i \varphi } - e^{ \frac{ 2 i \pi }{ 3 } } ) ( e^{ i \varphi } - e^{ \frac{ 4 i \pi }{ 3 } } ) | \geq d
\end{equation*}
and thus $ d \leq C( K ) \omega_2 $. From this property the statement of the Lemma easily follows. 

\medskip

We now consider the case when two out of three indexes $ j_0, j_1, j_2 $ (call them $ m, n $) belong to the same group and the third one (call it $ l $) belongs to the other group. Without loss of generality we can assume that $ l \in J_1 $, $ m, n \in J_2 $. If $ - \lambda_l $ does not coincide with $ \lambda_m $, $ \lambda_n $, then we have that $ \lambda_m ,  \lambda_n$ are two elements in $\{-\la_0,-\la_1,-\la_2\}$, and $-\la_l $ is the third element in this list. Therefore, 
\begin{equation*}
\lambda_l ( -\lambda_m ) ( -\lambda_n ) = 1,
\end{equation*}
and the reasoning of the previous case applies. If finally $ -\lambda_l $ coincides with either $ \lambda_m $ or $ \lambda_n $, then it is easy to see, due to Claim \ref{Claim00} in the proof of Lemma \ref{cluster_lemma_1}, that it can only be $ \lambda_{ j_1 } $ coinciding with $ - \lambda_{ j_2 } $. Thus we have that $ \omega_1 = | \lambda_{ j_0 } - \lambda_{ j_1 } | $, $ \omega_2 = | \lambda_{ j_0 } + \lambda_{ j_1 } | $. It is not difficult to check that in this case one of the points $ \lambda_{ j_0 } $, $ \lambda_{ j_1 } $ necessarily lies inside the unit circle (more precisely, in the set $ D = \{ \lambda \in \mathbb{C} \colon | \lambda | \leq 1 \} $) and the other one is necessarily on the unit circle (i.e. on $ \mathbb{S}^1 $). It follows then that $ \omega_2 \geq 1 $ and thus
\begin{equation*}
\forall j \in N_5, \quad \forall k \in N_5, \quad | \lambda_j - \lambda_k^* | \leq K \leq K \omega_2.
\end{equation*}
\end{proof}


\section{The integral inside the ball $B_2$: estimate of $I_{in}$}\label{Sect_5}

This section is devoted to the estimate of $I_{in}$ in \eqref{I_split}, which is required in order to finish the proof of estimate \eqref{xi_alpha}.

\subsection{Scheme}
Let $ \lambda_0 = \lambda_0( -u ) $, where $ \lambda_0( -u ) $ is given by Lemma \ref{st_points_lemma}. Let $\omega_1=\omega_1(-u)$, $\omega_2=\omega_2(-u)$ be defined by Definitions \ref{omega1}, \ref{omega2} correspondingly. Also let $ K $ be defined as in Lemma \ref{cluster_lemma_1}.

\medskip
   
For each $ t $ big enough we will consider the following four cases for the values of parameter $ u $.
\begin{itemize}
\item[\bf Case I] The set of values of $ u $ such that $\omega_1$, $\omega_2$ and $\lambda_0$ satisfy
\be\label{case1}
 \omega_1 \leq \frac{ 4 }{ t^{ 1/4 } }, \quad  \omega_2 \leq \frac{ 4 }{ t^{ 1/4 } }\quad \hbox{and}\quad  | \lambda_0 | < K.
 \ee
This is the set of values of $ u $ for which the stationary points of $ S $ are close to each other (and thus close to one of the degenerate points $ \lambda_k^* $ from \eqref{deg_points}). It is the case when $ u $ lies in a small neighborhood of the points $ u_k^* =- 18 e^{ \frac{ 2 \pi i k }{ 3 } } $.

\item[\bf Case II] The set of values of $ u $ such that 
\be\label{case2}
 \omega_1 \leq \frac{ 4 }{ t^{ 1/4 } } , \quad \omega_2 > \frac{ 4 }{ t^{ 1/4 } } \quad\hbox{and}\quad | \lambda_0 | < K.
\ee
This is the case when two out of three stationary points are close to each other, but far enough from the third one. In this case $ u $ belongs to a neighborhood of the curve $ \mathcal{U} $.

\item[\bf Case III] The set of values of $ u $ such that 
\be\label{case3}
\omega_1 > \frac{ 4 }{ t^{ 1/4 } } , \quad \omega_2 > \frac{ 4 }{ t^{ 1/4 } } \quad \hbox{and}\quad | \lambda_0| < K.
\ee
This is the case when the three stationary points are sufficiently separated from each other, but all belong to $ B_2( 0 ) $. In this case $ u $ lies in a neighborhood of the set $ \mathbb{U} $.

\item[\bf Case IV] The set of values of $ u $ such that $ | \lambda_0 | \geq K $. This is the case when parameter $ u $ lies outside a neighborhood of the set $ \mathbb{U} $.
\end{itemize}

Now we prove the estimate separately for each case.

\subsection{Case I} For this case, we prove the estimates in several steps.
\begin{enumerate}
\item[{\bf Step 1}] Set $ \varepsilon = \frac{ 1 }{ t^{ 1/4 } } $. Note that from \eqref{case1}, $ \omega_1 \leq 4 \varepsilon $, $ \omega_2 \leq 4 \varepsilon $.

Let $ j_0, j_1, j_2 $ be the indices from Lemma \ref{aux_lemma}. Set $ J_1 = \{ j_0, j_1, j_2 \} $, $ J_2 = N_5 \backslash J_1 $. 
Denote
\begin{gather*}
B_{ \varepsilon }^1 := \bigcup\limits_{ j \in J_1 } \{ \lambda \in \mathbb S^1 \colon | \lambda - \lambda_j | \leq \varepsilon \}, \\ B_{ \varepsilon }^2 := \bigcup\limits_{ j \in J_2 } \{ \lambda \in \mathbb S^1 \colon | \lambda - \lambda_j | \leq \varepsilon \}.
\end{gather*}
and
\begin{gather*}
B_{ \varepsilon } := B_{ \varepsilon }^1 \cup B_{ \varepsilon }^2, \quad D_{ \varepsilon } := B_{ \varepsilon } \times B_{ \varepsilon }.
\end{gather*}
In other words, $B_\ve$ represents the ensemble of points in the complex plane that are at $\ve$ distance from a stationary point $\la_j$, with the index $j$ belonging either to $J_1$ or to $J_2$. On the other hand, $D_\ve$ is a subset of $\Com^2$ where both coordinates are close to (maybe different) stationary points $(\la_j,\la_k)$. 

Note that $ B_{ \varepsilon } $ is a union of arcs on the unit circle. Note also that, if one uses the angular representation of the points on the unit circle $ \lambda = e^{ i \varphi } $, then $ B_{ \varepsilon } $ can be viewed as a union of segments of $ \mathbb{R} / 2 \pi \mathbb{Z} $ (see Fig. \ref{arcs_segms}). We will write
\begin{equation*}
B_{ \varepsilon } = \bigcup\limits_{ i } I_{ a_i, b_i },
\end{equation*}
where $ I_{ a_i, b_i } $ are the segments forming $ B_{ \varepsilon } $.  

Note also that $ D_{ \varepsilon } $ can be viewed as a union of rectangles on $(\mathbb{R} / 2 \pi \mathbb{Z}) \times (\mathbb{R} / 2 \pi \mathbb{Z})$ (see Fig. \ref{arcs_segms}). We will write 
\begin{equation}
\label{rectangles}
D_{ \varepsilon } = \bigcup\limits_{ i,j } \Pi_{ a_i, b_i }^{ a_j, b_j },
\end{equation}
where $ \Pi_{ a_i, b_i }^{ a_j, b_j } $ are the rectangles forming $ D_{ \varepsilon } $.  

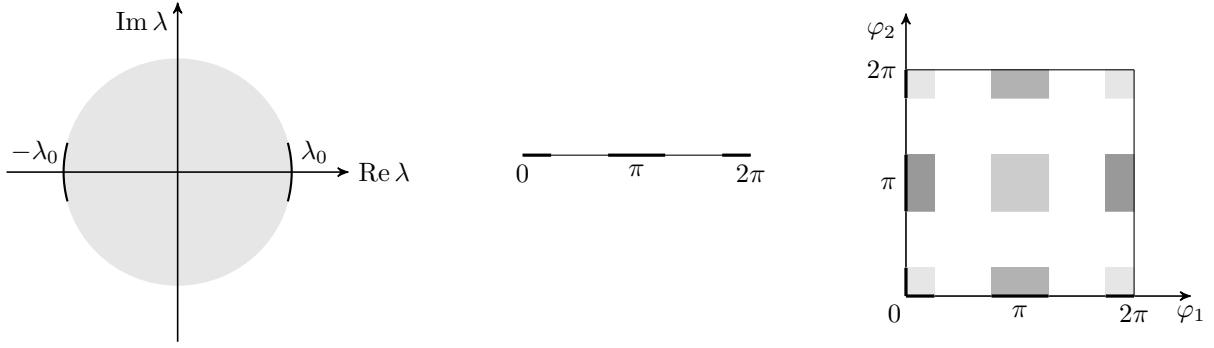
\begin{figure}[!h]
\minipage{0.35\textwidth}
\begin{tikzpicture}[
	>=stealth',
	axis/.style={semithick,->},
	coord/.style={dashed, semithick},
	yscale = 0.75,
	xscale = 0.75]
	\newcommand{\xmin}{-3};
	\newcommand{\xmax}{3};
	\newcommand{\ymin}{-3};
	\newcommand{\ymax}{3};
	\newcommand{\ta}{3};
	\newcommand{\fsp}{0.2};
	\filldraw[color=light-gray1] (0,0) circle (2);
	\draw [axis] (\xmin,0) -- (\xmax,0) node [right] {$\re \la$};
	\draw [axis] (0,\ymin) -- (0,\ymax) node [below left] {$\ima \la$};
	\draw (2.4,0.0) node [above] {$\lambda_0$};
	\draw (-2.5,0.0) node [above] {$-\lambda_0$};
	\draw[thick] (2,0) arc (0:15:2);
	\draw[thick] (2,0) arc (0:-15:2);
	\draw[thick] (-2,0) arc (180:195:2);
	\draw[thick] (-2,0) arc (180:165:2);
\end{tikzpicture}
\endminipage\hfill
\minipage{0.30\textwidth}
\centering
\begin{tikzpicture}[
	>=stealth',
	axis/.style={semithick,->},
	coord/.style={dashed, semithick},
	yscale = 0.75,
	xscale = 0.75]
	\newcommand{\xmin}{-3};
	\newcommand{\xmax}{3};
	\newcommand{\ymin}{-3};
	\newcommand{\ymax}{3};
	\newcommand{\ta}{3};
	\newcommand{\fsp}{0.2};
	\draw (-2,0) -- (2,0);
	\draw [very thick,-] (-2,0) -- (-1.5,0);
	\draw [very thick,-] (-0.5,0) -- (0.5,0);
	\draw [very thick,-] (1.5,0) -- (2,0);
	\draw (-2,0) node [below] {$0$};
	\draw (0,0) node [below] {$\pi$};
	\draw (2,0) node [below] {$2 \pi$};
\end{tikzpicture}
\endminipage
\hfill
\minipage{0.30\textwidth}
\centering
\begin{tikzpicture}[
	>=stealth',
	axis/.style={semithick,->},
	coord/.style={dashed, semithick},
	yscale = 0.75,
	xscale = 0.75]
	\newcommand{\xmin}{-3};
	\newcommand{\xmax}{3};
	\newcommand{\ymin}{-3};
	\newcommand{\ymax}{3};
	\newcommand{\ta}{3};
	\newcommand{\fsp}{0.2};
	\draw [axis] (-2,-2) -- (3,-2) node [below] {$\varphi_1$};
	\draw [axis] (-2,-2) -- (-2,3) node [below left] {$\varphi_2$};
	\draw (-2.2,-2) node [below] {$0$};
	\draw (0,-2) node [below] {$\pi$};
	\draw (2,-2) node [below] {$2 \pi$};
	\draw (-2,0) node [left] {$\pi$};
	\draw (-2,2) node [left] {$2 \pi$};
	\draw [fill=light-gray1, light-gray1] (-2,-2) rectangle (-1.5,-1.5);
	\draw [fill=light-gray1, light-gray1] (-2,1.5) rectangle (-1.5,2);
	\draw [fill=light-gray1, light-gray1] (1.5,-2) rectangle (2,-1.5);
	\draw [fill=light-gray1, light-gray1] (1.5,1.5) rectangle (2,2);
	\draw [fill=light-gray2, light-gray2] (-0.5,-0.5) rectangle (0.5,0.5);
	\draw [fill=light-gray3, light-gray3] (-2,-0.5) rectangle (-1.5,0.5);
	\draw [fill=light-gray3, light-gray3] (1.5,-0.5) rectangle (2,0.5);
	\draw [fill=light-gray4, light-gray4] (-0.5,-2) rectangle (0.5,-1.5);
	\draw [fill=light-gray4, light-gray4] (-0.5,1.5) rectangle (0.5,2);
	\draw (2,-2) -- (2,2);
	\draw (-2,2) -- (2,2);
	\draw [very thick,-] (-2,-2) -- (-1.5,-2);
	\draw [very thick,-] (-0.5,-2) -- (0.5,-2);
	\draw [very thick,-] (1.5,-2) -- (2,-2);
	\draw [very thick,-] (-2,-2) -- (-2,-1.5);
	\draw [very thick,-] (-2,-0.5) -- (-2,0.5);
	\draw [very thick,-] (-2,1.5) -- (-2,2);
\end{tikzpicture}
\endminipage
\caption{On the left: $ B_{ \varepsilon } $ viewed as a union of arcs on the unit circle; in the middle: $ B_{ \varepsilon } $ viewed as a union of segments of $ \mathbb{R} / 2 \pi \mathbb{Z} $; on the right: $ D_{ \varepsilon } $ viewed as a union of rectangles on $(\mathbb{R} / 2 \pi \mathbb{Z}) \times (\mathbb{R} / 2 \pi \mathbb{Z})$}\label{arcs_segms}
\end{figure}

\item[{\bf Step 2}]  Represent $ I_{in}$ as the sum of integrals over $ D_{ \varepsilon } $ and over $ \mathbb{T}^2 \backslash D_{ \varepsilon } $, where $ \mathbb{T} = \Ss^1 $:
\begin{align*}
& I_{in} = I_{in}^{ int } + I_{in}^{ ext }, \\
& I_{in}^{ int } = \int\limits_{ D_{ \varepsilon } } f( \varphi_1, \varphi_2 ) e^{ i t S( u, e^{ i \varphi_1 }, e^{ i \varphi_2 } ) } d \varphi_1 d \varphi_2 \text{ and } \\
& I_{in}^{ ext } = \int\limits_{ \mathbb{T}^2 \backslash D_{ \varepsilon } } f( \varphi_1, \varphi_2 ) e^{ i t S( u, e^{ i \varphi_1 }, e^{ i \varphi_2 } ) } d \varphi_1 d \varphi_2,
\end{align*}
where $ f( \varphi_1, \varphi_2 ) = | e^{ i \varphi_1 } + e^{ i \varphi_2 } |^{ \alpha } | \sin( \varphi_1 - \varphi_2 ) | $ (see (see \eqref{J_new}) and \eqref{lambda_phi}).

\medskip

\item[{\bf Step 3}]  Note that $ \mu( D_{ \varepsilon } ) \eqsim \varepsilon^2 $, where $ \mu $ denotes the Lebesgue measure.

Now we estimate $ f $ on $ D_{ \varepsilon } $. Note that for $ ( \lambda, \lambda' ) \in B_{ \varepsilon }^1 \times B_{ \varepsilon }^1 $ we have that there exist $ i, j \in J_1 $ such that $ | \lambda - \lambda_i | \leq \varepsilon $, $ | \lambda' - \lambda_j | \leq \varepsilon $. 

\medskip

By Lemma \ref{cluster_lemma} we also have that $ | \lambda_i - \lambda_j | \lesssim \omega_2 \leq 4 \varepsilon  $. Thus we get
\begin{equation*}
\begin{aligned}
| \sin( \varphi_1 - \varphi_2 ) | &  \lesssim | \varphi_1 - \varphi_2 |  \\
& \lesssim | \lambda - \lambda_i | + | \lambda_i - \lambda_j | + | \lambda' - \lambda_j |\\
&  \lesssim \varepsilon.
\end{aligned}
\end{equation*}

Reasoning similarly (and using Lemma \ref{cluster_lemma} and Lemma \ref{cluster_lemma_1} to estimate the difference between $ \lambda_i $ and $ \lambda_j $ for $ i, j \in J_2 $), we obtain the same estimate for $ | \sin( \varphi_1 - \varphi_2 ) | $ on $ B_{ \varepsilon }^2 \times B_{ \varepsilon }^2 $.

Consider now $ ( \lambda, \lambda' ) \in B_{ \varepsilon }^1 \times B_{ \varepsilon }^2 $. There exist $ i \in J_1, j \in J_2 $ such that $ | \lambda - \lambda_i | \leq \varepsilon $, $ | \lambda' - \lambda_j | \leq \varepsilon $. Now note that for $ \forall j \in J_2 $ we have that $ - \lambda_j = \lambda_k $ with $ k \in J_1 $ (see Corollary \ref{cor_coinc}). Thus on $ B_{ \varepsilon }^1 \times B_{ \varepsilon }^2 $ we have
\[
\begin{aligned}
| \sin( \varphi_1 - \varphi_2 ) | & = | \sin( \pi + \varphi_1 -  ( \varphi_2 + \pi ) ) | \\
& = | \sin( \varphi_1 -  ( \varphi_2 + \pi ) ) | \\
& \lesssim | \varphi_1 -  ( \varphi_2 + \pi ) | \\
& \lesssim | \lambda - \lambda_i | + | \lambda_i + \lambda_j | + | \lambda' - \lambda_j | \\
& \lesssim \varepsilon + | \lambda_i - \lambda_k | \\
&  \lesssim \varepsilon.
\end{aligned}
\]
A similar reasoning gives the same estimate for $ | \sin( \varphi_1 - \varphi_2 ) | $ on $ B_{ \varepsilon }^2 \times B_{ \varepsilon }^1 $.

Thus
\begin{equation}
\label{f_estimate}
| f | \lesssim \varepsilon \text{ on }  D_{ \varepsilon }
\end{equation}
and we obtain the following estimate for $ I_{in}^{ int } $:
\begin{equation*}
| I_{in}^{ int } | \lesssim \varepsilon^3.
\end{equation*}

\item[{\bf Step 4}] 
\begin{align*}
& D_{ \varepsilon }( \varphi_1 ) = \{ \varphi_2 \in [ 0, 2 \pi ): ( \varphi_1, \varphi_2 ) \in D_{ \varepsilon } \}, \\
& D_{ \varepsilon }( \varphi_2 ) = \{ \varphi_1 \in [ 0, 2 \pi ): ( \varphi_1, \varphi_2 ) \in D_{ \varepsilon } \}.
\end{align*}

If for a certain $ \varphi_1 $ the set $ D_{ \varepsilon }( \varphi_1 ) $ is emplty, then we set $ \left. g \right|_{ \varphi_2 \in D_{ \varepsilon }( \varphi_1 )  } = 0 $ for any function $ g $. Similarly, $ \left. g \right|_{ \varphi_1 \in D_{ \varepsilon }( \varphi_2 )  } = 0 $ for any $ g $, if $ D_{ \varepsilon }( \varphi_2 ) = \varnothing $ for a certain $ \varphi_2 $.

Let $ C $ denote the ``corner points'' of the rectangles forming $  D_{ \varepsilon }$ (see representation \ref{rectangles}). We introduce the following notation
\begin{equation*}
\left. f \right|_{ ( \varphi_1, \varphi_2 ) \in C } = \sum\limits_{ i, j } f( a_i, a_j ) - f( a_i, b_j ) - f( b_i, a_j ) + f( b_i, b_j )
\end{equation*}
(see also Fig. \ref{point_orientation}).

\begin{figure}[!h]
\centering
\begin{tikzpicture}[
	>=stealth',
	axis/.style={semithick,->},
	coord/.style={dashed, semithick},
	yscale = 0.75,
	xscale = 0.75]
	\newcommand{\xmin}{-3};
	\newcommand{\xmax}{3};
	\newcommand{\ymin}{-3};
	\newcommand{\ymax}{3};
	\newcommand{\ta}{3};
	\newcommand{\fsp}{0.2};
	\draw [axis] (-2,-2) -- (3,-2) node [below] {$\varphi_1$};
	\draw [axis] (-2,-2) -- (-2,3) node [below left] {$\varphi_2$};
	\draw (-2.2,-2) node [below] {$0$};
	\draw (0,-2) node [below] {$\pi$};
	\draw (2,-2) node [below] {$2 \pi$};
	\draw (-2,0) node [left] {$\pi$};
	\draw (-2,2) node [left] {$2 \pi$};
	\draw [fill=light-gray1, light-gray1] (-2,-2) rectangle (-1.5,-1.5);
	\draw [fill=light-gray1, light-gray1] (-2,1.5) rectangle (-1.5,2);
	\draw [fill=light-gray1, light-gray1] (1.5,-2) rectangle (2,-1.5);
	\draw [fill=light-gray1, light-gray1] (1.5,1.5) rectangle (2,2);
	\draw [fill=light-gray2, light-gray2] (-0.5,-0.5) rectangle (0.5,0.5);
	\draw [fill=light-gray3, light-gray3] (-2,-0.5) rectangle (-1.5,0.5);
	\draw [fill=light-gray3, light-gray3] (1.5,-0.5) rectangle (2,0.5);
	\draw [fill=light-gray4, light-gray4] (-0.5,-2) rectangle (0.5,-1.5);
	\draw [fill=light-gray4, light-gray4] (-0.5,1.5) rectangle (0.5,2);
	\draw (2,-2) -- (2,2);
	\draw (-2,2) -- (2,2);
	\draw [very thick,-] (-2,-2) -- (-1.5,-2);
	\draw [very thick,-] (-0.5,-2) -- (0.5,-2);
	\draw [very thick,-] (1.5,-2) -- (2,-2);
	\draw [very thick,-] (-2,-2) -- (-2,-1.5);
	\draw [very thick,-] (-2,-0.5) -- (-2,0.5);
	\draw [very thick,-] (-2,1.5) -- (-2,2);
	\draw (-1.4,-1.7) node [above] {$+$};
	\draw (-0.6,-1.7) node [above] {$-$};
	\draw (0.6,-1.7) node [above] {$+$};
	\draw (1.4,-1.7) node [above] {$-$};
	\draw (-1.4,0.3) node [above] {$+$};
	\draw (-0.6,0.3) node [above] {$-$};
	\draw (0.6,0.3) node [above] {$+$};
	\draw (1.4,0.3) node [above] {$-$};
	\draw (-1.4,-0.3) node [below] {$-$};
	\draw (-0.6,-0.3) node [below] {$+$};
	\draw (0.6,-0.3) node [below] {$-$};
	\draw (1.4,-0.3) node [below] {$+$};
	\draw (-1.4,1.7) node [below] {$-$};
	\draw (-0.6,1.7) node [below] {$+$};
	\draw (0.6,1.7) node [below] {$-$};
	\draw (1.4,1.7) node [below] {$+$};
\end{tikzpicture}
\caption{The signs in front of the values $ f( \varphi_1, \varphi_2 ) $ in the expression $ \left. f \right|_{ ( \varphi_1, \varphi_2 ) \in C } $, where $ C $ is the set of the ``corner points'' of $ D_{ \varepsilon } $.}\label{point_orientation}
\end{figure}
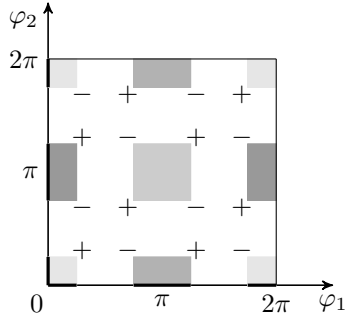

Integrating $ I_{in}^{ ext } $ by parts we obtain the following representation for $ I_{in}^{ ext } $:
\begin{multline*}
I_{in}^{ ext } = - \left. \frac{ 1 }{ t^2 } \frac{ e^{ i t S } f }{ S_{ \varphi_1 } S_{ \varphi_2 } } \right|_{ ( \varphi_1, \varphi_2 ) \in C } + \frac{ 1 }{ t^2 } \int\limits_{ 0 }^{ 2 \pi } \left.\frac{ e^{ i t S } }{ S_{ \varphi_1 } } \partial_{ \varphi_2 } \left( \frac{ f }{ S_{ \varphi_2 } } \right) \right|_{ \varphi_1 \in \partial D_{ \varepsilon }( \varphi_2 ) } d \varphi_2 + \\
+ \frac{ 1 }{ t^2 } \int\limits_{ 0 }^{ 2 \pi } \left. \frac{ e^{ i t S } }{ S_{ \varphi_2 } } \partial_{ \varphi_1 } \left( \frac{ f }{ S_{ \varphi_1 } } \right) \right|_{ \varphi_2 \in \partial D_{ \varepsilon }( \varphi_1 ) } d \varphi_1 - \frac{ 1 }{ t^2 } \iint\limits_{ \mathbb{T}^2 \backslash D_{ \varepsilon } } e^{ i t S } \partial_{ \varphi_1 } \partial_{ \varphi_2 } \left( \frac{ f }{ S_{ \varphi_1 } S_{ \varphi_2 } } \right) d \varphi_1 d \varphi_2 = \\
= :I_{in}^{ ext ,1} + I_{in}^{ ext ,2} +I_{in}^{ ext ,3} + I_{in}^{ ext ,4}.
\end{multline*}

\item[{\bf Step 5}]  In this step our goal is to find an estimate for each term $ I_{in}^{ ext ,k} $.

Note that for $ \varphi_1 \in \partial B_{ \rho } $, with $ \rho $ sufficiently small and $ \rho \geq \varepsilon $, we have
\begin{equation}
\label{s_estimate}
| S_{ \varphi_1 } | \eqsim \prod\limits_{ j = 0 }^{ 2 }| \lambda - \lambda_j | | \lambda + \lambda_j | \geq \rho^3.
\end{equation}

In the same manner we obtain that for $ \varphi_2 \in \partial B_{ \rho } $ the following estimate holds: $ | S_{ \varphi_2 } | \geq \rho^3 $.

Finally, using (\ref{f_estimate}) and (\ref{s_estimate}), we get
\begin{equation*}
| I_{in}^{ext,1} | \leq \abs{ \left. \frac{ 1 }{ t^2 } \frac{ e^{ i t S } f }{ S_{ \varphi_1 } S_{ \varphi_2 } } \right|_{ ( \varphi_1, \varphi_2 ) \in C }}  \lesssim \frac{ \varepsilon }{ t^{ 2 } \varepsilon^3 \cdot \varepsilon^3 } = \frac{ 1 }{ t^{ 2 } \varepsilon^5 }.
\end{equation*}

In order to estimate $ J_2 $, $ J_3 $, $ J_4 $ we divide $ \mathbb{T} \backslash D_{ \varepsilon } $ into the following regions ``around'' the points $ (\lambda_k, \lambda_l $):
\[
\begin{aligned}
\mathbb{T} \backslash D_{ \varepsilon } & = \bigcup_{ k, l \in N_5 } \Omega_{ k,l }, \qquad \hbox{where}  \\
\Omega_{ k,l } & = \Big\{ (\lambda, \lambda') \in \mathbb{T} \backslash D_{ \varepsilon }  \colon ~ | \lambda - \lambda_k | \leq | \lambda - \lambda_j | \quad \forall j \in N_5 \backslash \{ k \}, \\
& \qquad\qquad    | \lambda' - \lambda_l | \leq | \lambda - \lambda_j | \quad \forall j \in N_5 \backslash \{ l \} \Big\}.
\end{aligned}
\]
On $ \Omega_{ k,l } $, we perform the following change of variables
\begin{equation*}
\varphi_1 = \varphi_k + \rho_1, \quad \varphi_2 = \varphi_l + \rho_2.
\end{equation*}

We also note that on $ \partial B_{ \rho } $ we have that
\begin{equation}
\label{ss_estimate}
\left| \frac{ S_{ \varphi_k \varphi_k } }{ ( S_{ \varphi_k } )^2 } \right| \lesssim \frac{ 1 }{ \rho^4 }, \quad k = 1, 2.
\end{equation}
Using (\ref{f_estimate}), (\ref{s_estimate}) and (\ref{ss_estimate}), we obtain
\[
\begin{aligned}
 | I_{in}^{ext,2} |  & \leq \abs{\frac{ 1 }{ t^2 } \int\limits_{ 0 }^{ 2 \pi } \left.\frac{ e^{ i t S } }{ S_{ \varphi_1 } } \partial_{ \varphi_2 } \left( \frac{ f }{ S_{ \varphi_2 } } \right) \right|_{ \varphi_1 \in \partial D_{ \varepsilon }( \varphi_2 ) } d \varphi_2} \\
 & \lesssim \frac{ 1 }{ t^2 \varepsilon^3 } \int\limits_{ \varepsilon }^{ \varepsilon_0 } \left( \frac{ 1 }{ \rho_2^3 } + \frac{ \rho_2 + \varepsilon }{ \rho_2^4 } \right) d \rho_2\\
 &  \lesssim \frac{ 1 }{ t^{ 2 } \varepsilon^5 }, 
\end{aligned}
\]
\[
\begin{aligned}
 |  I_{in}^{ext,3} |  & \leq \abs{ \frac{ 1 }{ t^2 } \int\limits_{ 0 }^{ 2 \pi } \left. \frac{ e^{ i t S } }{ S_{ \varphi_2 } } \partial_{ \varphi_1 } \left( \frac{ f }{ S_{ \varphi_1 } } \right) \right|_{ \varphi_2 \in \partial D_{ \varepsilon }( \varphi_1 ) } d \varphi_1} \\
&  \lesssim \frac{ 1 }{ t^2 \varepsilon^3 } \int\limits_{ \varepsilon }^{ \varepsilon_0 } \left( \frac{ 1 }{ \rho_1^3 } + \frac{ \rho_1 + \varepsilon }{ \rho_1^4 } \right) d \rho_1 \\
&  \lesssim \frac{ 1 }{ t^{ 2 } \varepsilon^5 },
\end{aligned}
\]
and
\[
\begin{aligned}
 |  I_{in}^{ext,4} |  &  \leq \abs{\frac{ 1 }{ t^2 } \iint\limits_{ \mathbb{T}^2 \backslash D_{ \varepsilon } } e^{ i t S } \partial_{ \varphi_1 } \partial_{ \varphi_2 } \left( \frac{ f }{ S_{ \varphi_1 } S_{ \varphi_2 } } \right) d \varphi_1 d \varphi_2} \\
&  \lesssim \frac{ 1 }{ t^2 } \int\limits_{ \varepsilon }^{ \varepsilon_0 } \int\limits_{ \varepsilon }^{ \varepsilon_0 } \left( \frac{ 1 }{ \rho_1^3 \rho_2^3 } +  \frac{ 1 }{ \rho_1^3 \rho_2^4 } +  \frac{ 1 }{ \rho_1^4 \rho_2^3 } + \frac{ \rho_1 + \rho_2 }{ \rho_1^4 \rho_2^4 } \right) d \rho_1 d \rho_2 \\
&  \lesssim \frac{ 1 }{ t^{ 2 } \varepsilon^5 }.
\end{aligned}
\]
Thus
\begin{equation*}
| I_{in}^{ext} | \lesssim \frac{ 1 }{ t^{ 2 } \varepsilon^5 }.
\end{equation*}

\item[{\bf Step 6}]  Finally, since $ \varepsilon = \frac{ 1 }{ t^{ 1/4 } } $, we obtain that
\begin{equation*}
| I_{in}^{ext}| \lesssim \frac{ 1 }{ t^{ 3/4 } }.
\end{equation*}

\end{enumerate}

\subsection{Case II} We remind that, by \eqref{case2}, Case II corresponds to the conditions
\[
 \omega_1 \leq \frac{ 4 }{ t^{ 1/4 } } , \quad \omega_2 > \frac{ 4 }{ t^{ 1/4 } } \quad\hbox{and}\quad | \lambda_0 | < K.
\]
We set $ \displaystyle{\varepsilon_1 = \frac{ 1 }{ t^{ 1/2 } \omega_2 } }$ and $\displaystyle{ \varepsilon_2 = \frac{ 1 }{ t^{ 1/4 } } }$. Note that $ 16 \varepsilon_1 < \omega_2 $, $ 4 \varepsilon_2 \geq \omega_1 $ and $ \omega_2 > 4 \varepsilon_2 $. For this case we perform the following procedure:

\begin{enumerate}
\item Using Lemma \ref{aux_lemma}, we divide the stationary points into two sets:
\begin{equation*}
D := \{ \lambda_{j_0}, - \lambda_{j_0}, \lambda_{j_1}, - \lambda_{j_1} \}, \quad N: = \{ \lambda_{j_2}, - \lambda_{j_2} \},
\end{equation*}
with $ D $ standing for (possibly) degenerate and $ N $ standing for nondegenerate. Here $ j_0, j_1, j_2 $ are given by Lemma \ref{aux_lemma}. Note that, by Lemmas \ref{cluster_lemma} $, \ref{cluster_lemma_1}, D \cup N $ is the set of all the stationary points.

We set
\begin{gather*}
P_{ DD } := D \times D, \quad P_{ ND } := N \times D, \quad P_{ DN } := D \times N, \quad P_{ NN } := N \times N, \\
P := ( D \cup N ) \times ( D \cup N ).
\end{gather*}
Now let
\begin{align*}
& \forall ( \lambda, \mu ) \in P_{ DD }, \quad D_{ \lambda \mu } := \{ \lambda_1, \lambda_2 \in \mathbb{S}^1 : | \lambda_1 - \lambda | \leq \varepsilon_2, | \lambda_2 - \mu | \leq \varepsilon_2 \}, \\
& \forall ( \lambda, \mu ) \in P_{ ND }, \quad D_{ \lambda \mu } := \{ \lambda_1, \lambda_2 \in \mathbb{S}^1 : | \lambda_1 - \lambda | \leq \varepsilon_2, | \lambda_2 - \mu | \leq \varepsilon_1 \}, \\
& \forall ( \lambda, \mu ) \in P_{ DN }, \quad D_{ \lambda \mu } := \{ \lambda_1, \lambda_2 \in \mathbb{S}^1 : | \lambda_1 - \lambda | \leq \varepsilon_1, | \lambda_2 - \mu | \leq \varepsilon_2 \}, \\
& \forall ( \lambda, \mu ) \in P_{ NN }, \quad D_{ \lambda \mu } := \{ \lambda_1, \lambda_2 \in \mathbb{S}^1 : | \lambda_1 - \lambda | \leq \varepsilon_2, | \lambda_2 - \mu | \leq \varepsilon_2 \}.
\end{align*}

Finally, $ D_{ \varepsilon }: = \bigcup\limits_{ ( \lambda, \mu ) \in P } D_{ \lambda \mu } $.

\item We represent $ I_{in} = I_{in}^{ int } + I_{in}^{ ext } $, where $ I_{in}^{ int } $ is the integral over $ D_{ \varepsilon } $ and $ I_{in}^{ ext } $ is the integral over $ \mathbb{T}^2 \backslash D_{ \varepsilon } $.

\item The estimate of $ I_{in}^{ int } $. We set 
\be\label{sep_I_in_int}
 I_{in}^{ int } =  :I_{in}^{ int ,1} +  I_{in}^{ int ,2} ,
\ee
with $  I_{in}^{ int ,1} $ being the integral over the set $ \bigcup\limits_{ ( \lambda, \mu ) \in P_{ DD } \cup P_{ NN } } D_{ \lambda \mu } $ and $ I_{in}^{ int ,2} $ being the integral over $ \bigcup\limits_{ ( \lambda, \mu ) \in P_{ ND } \cup P_{ DN } } D_{ \lambda \mu } $. Then we can estimate
\begin{align*}
& |  I_{in}^{ int ,1} | \lesssim \varepsilon_2^3 = \frac{ 1 }{ t^{ 3/4 } }, \\
& |  I_{in}^{ int ,2} | \lesssim \varepsilon_1 \varepsilon_2 \omega_2 = \frac{ 1 }{ t^{ 3/4 } }.
\end{align*}

\item Define
\begin{align*}
& D_{ \varepsilon }( \varphi_1 ) = \{ \varphi_2 \in [ 0, 2 \pi ): ( \varphi_1, \varphi_2 ) \in D_{ \varepsilon } \}, \\
& D_{ \varepsilon }( \varphi_2 ) = \{ \varphi_1 \in [ 0, 2 \pi ): ( \varphi_1, \varphi_2 ) \in D_{ \varepsilon } \}.
\end{align*}

Integrating by parts we obtain the following representation for $ I_{in}^{ext} $:
\begin{multline*}
I_{in}^{ext} = - \left. \frac{ 1 }{ t^2 } \frac{ e^{ i t S } f }{ S_{ \varphi_1 } S_{ \varphi_2 } } \right|_{ ( \varphi_1, \varphi_2 ) \in C } + \frac{ 1 }{ t^2 } \int\limits_{ 0 }^{ 2 \pi } \left.\frac{ e^{ i t S } }{ S_{ \varphi_1 } } \partial_{ \varphi_2 } \left( \frac{ f }{ S_{ \varphi_2 } } \right) \right|_{ \varphi_1 \in \partial D_{ \varepsilon }( \varphi_2 ) } d \varphi_2  \\
\qquad + \frac{ 1 }{ t^2 } \int\limits_{ 0 }^{ 2 \pi } \left. \frac{ e^{ i t S } }{ S_{ \varphi_2 } } \partial_{ \varphi_1 } \left( \frac{ f }{ S_{ \varphi_1 } } \right) \right|_{ \varphi_2 \in \partial D_{ \varepsilon }( \varphi_1 ) } d \varphi_1 - \frac{ 1 }{ t^2 } \iint\limits_{ \mathbb{T}^2 \backslash D_{ \varepsilon } } e^{ i t S } \partial_{ \varphi_1 } \partial_{ \varphi_2 } \left( \frac{ f }{ S_{ \varphi_1 } S_{ \varphi_2 } } \right) d \varphi_1 d \varphi_2 = \\
 =: I_{in}^{ext,1} + I_{in}^{ext,2} + I_{in}^{ext,3}+ I_{in}^{ext,4}. 
\end{multline*}
Here, as in the previous case, $ C $ denotes the ``corner points'' of rectangles forming $ D_{ \varepsilon } $ and the notation $ \left.f\right|_{ ( \varphi_1, \varphi_2 ) \in C } $ is introduced similarly to the Case I (note, however, that in this case $ D_{ \varepsilon } $ cannot be represented as cartesian square of a certain set $ B_{ \varepsilon } $).

\item We only show how to estimate $  I_{in}^{ext,1}  $. For estimating the remaining cases we use similar arguments and the scheme of reasoning presented in the previous case.

For $ ( \lambda, \mu ) \in P_{ DD } $ we have
\begin{equation*}
|  I_{in}^{ext,1}  | \lesssim \frac{ \varepsilon_2 }{ t^2 \varepsilon_2^2 \omega_2 \varepsilon_2^2 \omega_2 } = \frac{ 1 }{ t^2 \varepsilon_2^3 \omega_2^2 } = \frac{ 1 }{ t^{5/4} \omega_2^2 } \lesssim \frac{ 1 }{ t^{ 3/4 } }.
\end{equation*}
For $ ( \lambda, \mu ) \in P_{ ND } \cup P_{ DN } $ we have
\begin{equation*}
|  I_{in}^{ext,1}  | \lesssim \frac{ \omega_2 }{ t^2 \varepsilon_1^2 \omega_2 \varepsilon_2 \omega_2^2 } = \frac{ 1 }{ t \varepsilon_2 } = \frac{ 1 }{ t^{ 3/4 } }.
\end{equation*}
For $ ( \lambda, \mu ) \in P_{ NN } $ we have
\begin{equation*}
|  I_{in}^{ext,1}  | \lesssim \frac{ \omega_2 }{ t^2 \varepsilon_2 \omega_2^2 \varepsilon_2 \omega_2^2 } = \frac{ 1 }{ t^2 \varepsilon_2^2 \omega_2^3 } = \frac{ 1 }{ t^{3/2} \omega_2^3 } \lesssim \frac{ 1 }{ t^{ 3/4 } }.
\end{equation*}

\item Finally, we obtain that $ |  I_{in}^{ext}  | \lesssim \frac{ 1 }{ t^{ 3/4 } } $, and the same bound holds for $|I_{in}|$.
\end{enumerate}

\subsection{Case III}
In this case we follow closely the argument of the previous case, with some slight modifications. In particular, we split the set of points  $ D \times D $ as follows:
\begin{equation*}
D \times D = P_{ sym } ~ \sqcup~  P_{ assym }, \text{ with } P_{ sym }: = \{ ( \lambda_{j_0}, \lambda_{j_0} ), ( \lambda_{j_0}, -\lambda_{j_0} ), ( \lambda_{j_1}, \lambda_{j_1} ), ( \lambda_{j_1}, -\lambda_{j_1} ) \},
\end{equation*}
where the symbol $ \sqcup $ denotes the disjoint union. Note that  because of the disjoint union, the set $ P_{ assym }$ is well-defined . Further, for all $ \lambda, \mu \in P \backslash P_{ assym } $, the set $ D_{ \lambda \mu } $ is defined as in the previous case and
\begin{equation*}
\forall \lambda, \mu \in P_{ assym }, \quad D_{ \lambda \mu } := \{ \lambda_1, \lambda_2 \in S^1 : | \lambda_1 - \lambda | \leq \varepsilon_3, | \lambda_2 - \mu | \leq \varepsilon_3 \},
\end{equation*}
with $ \varepsilon_3 = \frac{ 1 }{ t^{ 3/8 } \omega_1^{ 1/2 } } $.

The integral $  I_{in}^{int,1}  $ (cf. \eqref{sep_I_in_int}) is represented as $  I_{in}^{int,1}  =  I_{in}^{int,1,old}  + I_{in}^{int,1,new}  $, where $  I_{in}^{ext,1, old} $ is the integral over the set
\[
 \bigcup\limits_{ ( \lambda, \mu ) \in ( P_{ DD } \cup P_{ NN } ) \backslash P_{ assym } } D_{ \lambda \mu } 
\]
(and it is estimated as in the previous case), and $  I_{in}^{ext,1,new}  $ is the integral over 
\[
\bigcup\limits_{ ( \lambda, \mu ) \in P_{ assym } } D_{ \lambda \mu },
\]
and is estimated as follows:
\begin{equation*}
|  I_{in}^{int,1,new}  | \lesssim \varepsilon_3^2 ~\omega_1 = \frac{ 1 }{ t^{ 3/4 } }.
\end{equation*}

In a similar way we split $I_{in}^{ext,1}  $ into the sum $ I_{in}^{ext, 1 } =  I_{in}^{ext,1,old}  + I_{in}^{ext,1,new} $. The integral $ I_{in}^{ext,1,old}   $ is estimated as in the previous case and for the integral $ I_{in}^{ext,1,new}   $ we have
\begin{equation*}
| I_{in}^{ext,1,new}   | \lesssim \frac{ \omega_1 }{ t^2\varepsilon_3^2 \omega_1^2 \omega_2^2 } = \frac{ 1 }{ t^{5/4} \omega_2^2 } \lesssim \frac{ 1 }{ t^{ 3/4 } }.
\end{equation*}

\subsection{Case IV} In this case we follow very closely the scheme of Case I.
We set now $ \varepsilon = \frac{ 1 }{ t^{ 1/2 } } $ and
\begin{gather*}
D_{ \varepsilon } := B_{ \varepsilon } \times B_{ \varepsilon }, \quad B_{ \varepsilon } := \bigcup\limits_{ k \in N_5 } B_{ \varepsilon }^{ k }, \\
B_{ \varepsilon }^{ k } := \{ \lambda \in \Ss^1 \colon | \lambda - \lambda_k | \leq \varepsilon \},
\end{gather*}
and we represent $ I_{in} = I_{in}^{ int } + I_{in}^{ ext } $, where $ I_{ in}^{int} $ is the integral over $ D_{ \varepsilon } $ and $ I_{in}^{ext} $ is the integral over $ \mathbb{T}^2 \backslash D_{ \varepsilon } $.
The integral $ I_{ in}^{int}$ is estimated as
\begin{equation*}
|I_{ in}^{int} | \lesssim \varepsilon^2 = \frac{ 1 }{ t }.
\end{equation*}

To estimate $ I_{ in}^{ext} $ we integrate it by parts as we did in the Case I. Using the fact that on $ \partial B_{ \rho }^{ k } $ with $ \rho \geq \varepsilon $ we have $ | S_{ \varphi_j } | \gtrsim \rho $, we get that
\begin{equation*}
|  I_{ in}^{ext} | \lesssim \frac{ 1 }{ t^2 \varepsilon^2 } = \frac{ 1 }{ t }.
\end{equation*}
This last estimate completes the proof of \eqref{xi_alpha}.

\bigskip

\section{Smoothing estimate for large frequencies}
\label{large_freq}

In this section, our goal is to prove Proposition \ref{Smoothing_Large}. We consider two regimes, depending on small and large time $t$. The key point that makes the estimate of Proposition \ref{Smoothing_Large} better than that of Proposition \ref{all_freq_prop} is the fact that, at small frequencies, if $E<0$, some cancellations appear when estimating $I$ \eqref{I}, because of the term $|\xi|^\al$. This property does not repeat if $E>0$, however, if we consider only large frequencies, we do not need to invoke any cancellation in the integral $I$.  

\subsection{Small time}

For small $ t $ the estimate can be obtained as follows. Recall the integral $I_R$ in \eqref{I_R} and the cut-off function $\psi_R$ from \eqref{Phi_R}. Note first that 
\begin{equation*}
I_R = \int\limits_{ \mathbb{C} } | \xi |^{ \alpha }  e^{ i t \tilde S( u, \xi ) } d \Re \xi d \Im \xi - \int\limits_{ \mathbb{C} } | \xi |^{ \alpha } ( 1 - \psi_R( \xi ) )  e^{ i t \tilde S( u, \xi ) } d \Re \xi d \Im \xi.
\end{equation*}
We have that 
\begin{equation*}
\abs{ \int\limits_{ \mathbb{C} } | \xi |^{ \alpha }  e^{ i t \tilde S( u, \xi ) } d \Re \xi d \Im \xi  } \lesssim \frac{ 1 }{ t^{ \frac{ \alpha + 2 }{ 3 } } },
\end{equation*}
uniformly for small $ t $, where the implicit constant depends only on $ \alpha $. The proof of this estimate has been carried out in \cite[Section 3.5]{KM} for the case of negative energy; it can be repeated without any change in the case of positive energy (see also Lemma \ref{i_out_lemma}).
Further,
\begin{equation*}
\left| \int\limits_{ \mathbb{C} } | \xi |^{ \alpha } ( 1 - \psi_R( \xi ) )  e^{ i t \tilde S( u, \xi ) } d \Re \xi d \Im \xi \right| \leq \pi ( R + 1 )^{ \alpha + 2 },
\end{equation*}
from which it follows that for small $ t $ we have $ | I_R | \lesssim \frac{ 1 }{ t^{ \frac{ \alpha + 2 }{ 3 } } } $.

\subsection{Large time}

We start by performing the standard change of variables \eqref{ChVar}. The integral $ I_R $ then becomes 
\begin{equation*}
I_R = \int\limits_{ \mathbb{C} } \frac{ | \lambda \bar \lambda + 1 |^{ \alpha } }{ | \lambda |^{ \alpha + 4 } } ( |\lambda|^4 - 1 ) \psi_R\left( \lambda + \frac{ 1 }{ \bar \lambda } \right) e^{ i t S( u, \lambda ) } d \Re \lambda d \Im \lambda,
\end{equation*}
where the phase $ S $ is defined by \eqref{S_u}.

\medskip

Note that $\supp \psi_R\left( \lambda + \frac{1}{ \bar \lambda } \right) \subset \mathbb{C} \backslash B_{ \theta }( 0 ) $ with $ \theta = \frac{ R + \sqrt{ R^2 - 4 } }{ 2 } > 1 $. Thus the stationary points $ \pm \lambda_1 $, $ \pm \lambda_2 $ of the phase $ S $ are strictly separated from the domain of integration, i.e. from $\supp \psi_R\left( \lambda + \frac{1}{ \bar \lambda } \right) $ (see Lemma \ref{st_points_lemma} for the definition of $ \lambda_0 $, $ \lambda_1 $, $ \lambda_2 $). The only possible stationary points in the domain of integration are the nondegenerate points $ \pm \lambda_0 $ of the Case 4 of Lemma \ref{st_points_lemma}.

\medskip

We consider the following two cases: $ \pm \lambda_0 \in B_{1 + \frac{ \theta - 1 }{ 2  } }( 0 ) $, and $ \pm \lambda_0 \in \mathbb{C} \backslash B_{1 + \frac{ \theta - 1 }{ 2 } }( 0 ) $.
\begin{enumerate}
\item {\bf Case $ \pm \lambda_0 \in B_{1 + \frac{ \theta - 1 }{ 2 } }( 0 ) $.}

In this case all the stationary points are uniformly outside the domain of integration. 

Define 
\begin{equation*}
f( \lambda ) = \frac{ | \lambda \bar \lambda + 1 |^{ \alpha } }{ | \lambda |^{ \alpha + 4 } } ( |\lambda|^4 - 1 ).
\end{equation*}
By the Stokes formula we have that 
\begin{equation*}
I_R = \sum_{ j = 2 }^{ 4 } I_j,
\end{equation*}
with
\begin{equation*}
I_2 := - \frac{ 1 }{ it } \int\limits_{ \mathbb{C} } \frac{ f_{ \lambda }( \lambda ) \exp( i t S( u, \lambda ) ) }{ S_{ \lambda }( u, \lambda ) } \psi_R\left( \lambda + \frac{ 1 }{ \bar \lambda } \right) ,
\end{equation*}
and
\begin{equation*}
I_3 := \frac{ 1 }{ it } \int\limits_{ \mathbb{C}  } \frac{ f( \lambda ) \exp( i t S( u, \lambda ) ) S_{ \lambda \lambda }( u, \lambda ) }{ ( S_{ \lambda }( u, \lambda ) )^2 } \psi_R\left( \lambda + \frac{ 1 }{ \bar \lambda } \right) ,
\end{equation*}
\begin{equation*}
I_4 = - \frac{ 1 }{ i t } \int\limits_{ \mathbb{C} } \frac{ f( \lambda ) e^{ i t S( u, \lambda ) } }{ S_{ \lambda }( u, \lambda ) } \left( \partial_{ \xi } \psi_R \left( \lambda + \frac{ 1 }{ \bar \lambda } \right) - \partial_{ \bar \xi } \psi_R \left( \lambda + \frac{ 1 }{ \bar \lambda } \right) \frac{ 1 }{ \lambda^2 } \right) 
\end{equation*}
We notice that, due to assumptions on $ \lambda_0 $,
\begin{equation*}
| S_{ \lambda } | \gtrsim | \lambda |^2, \quad \frac{ | S_{ \lambda \lambda } | }{ | S_{ \lambda } |^2 } \lesssim \frac{ 1 }{ | \lambda |^3 }, \quad \text{ for } \lambda \in \supp \psi_R\left( \lambda + \frac{1}{ \bar \lambda } \right).
\end{equation*}
Besides,
\begin{equation*}
\supp \partial_{ \xi } \psi_R\left( \lambda + \frac{1}{ \bar \lambda } \right) ~ \bigcup  ~ \supp \partial_{ \bar \xi } \psi_R\left( \lambda + \frac{1}{ \bar \lambda } \right)  \subset B_{ \theta_1 }( 0 ) \backslash B_{ \theta }( 0 ),
\end{equation*}
where $ \theta_1 = \frac{ R + 1 + \sqrt{ ( R + 1 )^2 - 4 } }{ 2 } $. Finally, for $ \lambda \in \supp \psi_R\left( \lambda + \frac{1}{ \bar \lambda } \right) $ we have that
\begin{equation*}
| f( \lambda ) | \sim | \lambda |^{ \alpha }, \quad | f_{ \lambda }( \lambda ) \sim |\lambda  |^{ \alpha - 1 }.
\end{equation*}

Thus we get the following estimate
\begin{equation*}
| I_R | \lesssim \frac{ 1 }{ t }.
\end{equation*}

\item {\bf Case $ \pm \lambda_0 \in \mathbb{C} \backslash B_{1 + \frac{ \theta - 1 }{ 2 } }( 0 ) $.}

In this case the points $ \pm \lambda_0 $ may lie in the domain of integration. However, they are well separated from the unit disk and thus from the other stationary points. Consequently, in this case the reasoning follows closely the study of Case 1 of  \cite[Lemma 3.1]{KM}. Below, we present the scheme of the proof of the estimate in the considered case. For details the reader may refer to \cite[Lemma 3.1, Case 1]{KM}.

We take $ D_{ \varepsilon } $ to be the union of disks with radius $ \varepsilon $ (to be chosen later) and with centers in points $ \lambda_0 $, $ -\lambda_0 $. Define
\begin{equation*}
r:= |\la_0|.
\end{equation*}
We represent $ I_R $ in the following way
\begin{equation*}
\begin{aligned}
& I_R = I_{ int } + I_{ ext }, \quad \text{ where } \\
& I_{ int } = \int\limits_{ D_{ \varepsilon } } f( \lambda ) \exp( i t S( u, \lambda ) ) \psi_R\left( \lambda + \frac{1}{ \bar \lambda } \right), \qquad \hbox{ and }\\
& I_{ ext } = \int\limits_{ \mathbb{ C } \backslash D_{ \varepsilon } } f( \lambda ) \exp( i t S( u, \lambda ) ) \psi_R\left( \lambda + \frac{1}{ \bar \lambda } \right) .
\end{aligned}
\end{equation*}
It is easy to see that we have the following estimate for $ I_{ int } $:
\begin{equation*}
| I_{ int } | \lesssim \varepsilon^2 r^{ \alpha }. 
\end{equation*} 

To estimate $ I_{ ext } $ we first use the Stokes formula:
\begin{equation*}
I_{ ext } = \sum_{ j = 1 }^{ 4 } I_j,
\end{equation*}
with
\begin{equation*}
I_1 := \frac{ 1 }{ 2t } \int\limits_{ \partial D_{ \varepsilon } } \frac{ f( \lambda ) \exp( i t S( u, \lambda ) ) }{ S_{ \lambda }( u, \lambda ) } \psi_R\left( \lambda + \frac{ 1 }{ \bar \lambda } \right) d \bar \lambda
\end{equation*}
\begin{equation*}
I_2 := - \frac{ 1 }{ it } \int\limits_{ \mathbb{C} \backslash D_{ \varepsilon } } \frac{ f_{ \lambda }( \lambda ) \exp( i t S( u, \lambda ) ) }{ S_{ \lambda }( u, \lambda ) } \psi_R\left( \lambda + \frac{ 1 }{ \bar \lambda } \right) ,
\end{equation*}
and
\begin{equation*}
I_3 := \frac{ 1 }{ it } \int\limits_{ \mathbb{C} \backslash D_{ \varepsilon } } \frac{ f( \lambda ) \exp( i t S( u, \lambda ) ) S_{ \lambda \lambda }( u, \lambda ) }{ ( S_{ \lambda }( u, \lambda ) )^2 } \psi_R\left( \lambda + \frac{ 1 }{ \bar \lambda } \right) ,
\end{equation*}
\begin{equation*}
I_4 = - \frac{ 1 }{ it } \int\limits_{ \mathbb{C} \backslash D_{ \varepsilon } } \frac{ f( \lambda ) e^{ i t S( u, \lambda ) } }{ S_{ \lambda }( u, \lambda ) } \left( \partial_{ \xi } \psi_R \left( \lambda + \frac{ 1 }{ \bar \lambda } \right) - \partial_{ \bar \xi } \psi_R \left( \lambda + \frac{ 1 }{ \bar \lambda } \right) \frac{ 1 }{ \lambda^2 } \right). 
\end{equation*}

Now let 
\begin{equation*}
\Omega = \{ | \lambda | < 2 \} \cup \{ | \lambda - \lambda_0 | < 1 \} \cup \{ | \lambda + \lambda_0 | < 1 \}.
\end{equation*}
We represent each $ I_k $, $ k = 2, 3, 4 $, as $ I_k = I_k^{ + } + I_k^{ - } $, where the integrands of $  I_k^{ + } $, $  I_k^{ - } $ are the same as that of $ I_{ k } $ and the domains of integration are $ \mathbb{C} \backslash \Omega $ for $  I_k^{ + } $ and $ \Omega \backslash D_{ \varepsilon } $ for $  I_k^{ - } $.

The integrals $ I_2^{ + } $, $ I_3^{ + } $ are treated absolutely in the same way as the corresponding integrals in the case of negative energy (see Case 1, Lemma 3.1 of \cite{KM}). 
In particular, it can be shown that 
\begin{equation*}
| I_2^{ + } | \lesssim \frac{ 1 }{ t }, \quad | I_3^{ + } | \lesssim \frac{ 1 }{ t }.
\end{equation*}

To estimate $ I_4^{ + } $, we note that
\begin{equation*}
\supp \partial_{ \xi } \psi_R\left( \lambda + \frac{1}{ \bar \lambda } \right) \cup \supp \partial_{ \bar \xi } \psi_R\left( \lambda + \frac{1}{ \bar \lambda } \right)  \subset B_{ \theta_1 }( 0 ) \backslash B_{ \theta }( 0 ),
\end{equation*}
and $ \partial_{ \xi } \psi_R \left( \lambda + \frac{ 1 }{ \bar \lambda } \right) - \partial_{ \bar \xi } \psi_R \left( \lambda + \frac{ 1 }{ \bar \lambda } \right) \frac{ 1 }{ \lambda^2 }  $  is bounded in this domain. Thus, we have the following estimate
\begin{equation*}
| I_4^{ + } | \lesssim \frac{ 1 }{ t }.
\end{equation*}

Now we estimate integrals $ I_1 $, $ I_k^{ - } $, $ k = 2, 3, 4 $. First, we denote
\begin{equation*}
\lambda^{ ( 1 ) } = \lambda_0, \quad \lambda^{ ( 2 ) } = - \lambda_0
\end{equation*}
and we split $ \Omega \backslash D_{ \varepsilon } $ into two parts:
\begin{align*}
& \Omega \backslash D_{ \varepsilon } = \Omega^{ ( 1 ) } \cup \Omega^{ ( 2 ) }, \text{ where } \\
& \Omega^{ (k) } = \{ \lambda \in \Omega \backslash D_{ \varepsilon } ~ \colon ~ | \lambda - \lambda ^{ (k) } | \leq | \lambda - \lambda^{ ( j ) } |, ~k \neq j \}, \quad k = 1, 2.
\end{align*}
Then we put
\begin{gather*}
I_1^k =  \frac{ 1 }{ 2t } \int\limits_{ \partial D_{ \varepsilon } \cap \Omega^{ (k) } } \frac{ f( \lambda ) e^{ i t S( u, \lambda ) } }{ S_{ \lambda }( u, \lambda ) } \psi_R\left( \lambda + \frac{ 1 }{ \bar \lambda } \right) d \bar \lambda, \\
I_2^{-,k} = \frac{ 1 }{ it } \int\limits_{ \Omega^{ ( k ) } } \frac{ f_{ \la }( \lambda ) e^{ i t S( u, \lambda ) } }{ S_{ \lambda }( u, \lambda ) } \psi_R \left( \lambda + \frac{ 1 }{ \bar \lambda } \right), \\
I_3^{-,k} = \frac{ 1 }{ it } \int\limits_{ \Omega^{ (k) } } \frac{ f( \lambda ) e^{ i t S( u, \lambda ) } S_{ \lambda \lambda }( u, \lambda ) }{ ( S_{ \lambda }( u, \lambda ) )^2 } \psi_R\left( \lambda + \frac{ 1 }{ \bar \lambda } \right), \\
I_4^{-,k} = - \frac{ 1 }{ it } \int\limits_{ \Omega^{ (k) } } \frac{ f( \lambda ) e^{ i t S( u, \lambda ) } }{ S_{ \lambda }( u, \lambda ) } \left( \partial_{ \xi } \psi_R \left( \lambda + \frac{ 1 }{ \bar \lambda } \right) - \partial_{ \bar \xi } \psi_R \left( \lambda + \frac{ 1 }{ \bar \lambda } \right) \frac{ 1 }{ \lambda^2 } \right)
\end{gather*}
For every integral over domain $ \Omega^{ ( k ) } $, we use the following polar coordinates at the stationary point $\la^{(k)}$:
\[
\lambda = \lambda^{ (k) } + \rho~ e^{ i \varphi }, \quad \varphi \in [0,2\pi),
\]
and where $\ve\leq \rho \leq \rho_0(\varphi)$. Note that $\rho_0$ is always bounded by a fixed constant $\ve_0$, that can be chosen e.g. equal to 4, uniformly on $\varphi$.

Taking into account that on  $ \partial B_{ \rho }(\lambda^{ (k) })  \cap \Omega^{ (k) } \cap \supp \psi_R\left( \lambda + \frac{ 1 }{ \bar \lambda } \right) $ we have the following estimates:
\begin{gather*}
\left| \frac{ | \lambda |^4 - 1 }{ | \lambda |^4 } \right| \lesssim 1, \quad \left| \frac{ | \lambda |^2 - 1 }{ | \lambda | } \right| \lesssim r, \\
| S_{ \lambda } | \gtrsim \rho r, \quad \left| \frac{ S_{ \lambda \lambda } }{ ( S_{ \lambda } )^2 } \right| \lesssim \frac{ 1 }{ \rho r^2 } + \frac{ 1 }{ \rho^2 r },
\end{gather*}
it is not difficult to get the following bounds:
\begin{gather*}
| I_1^k  | \lesssim \frac{ 1 }{ t } \frac{r^{ \alpha } \varepsilon}{ \varepsilon r } = \frac{ r^{ \alpha - 1 } }{ t }, \\
| I_2^{-,k} | \lesssim \frac{ 1 }{ t } \int\limits_{ \varepsilon }^{ \varepsilon_0 } \frac{ r^{ \alpha - 1 } \rho \, d \rho }{ \rho r } \lesssim \frac{ 1 }{ t }r^{ \alpha - 2 }, \\
| I_3^{-,k} | \lesssim \frac{ 1 }{ t } \int\limits_{ \varepsilon }^{ \varepsilon_0 } \frac{ r^{ \alpha } \rho d \rho }{ \rho r^2 } + \frac{ 1 }{ t } \int\limits_{ \varepsilon }^{ \varepsilon_0 } \frac{ r^{ \alpha } \rho d \rho }{ \rho^2 r } 
\lesssim  \frac{ 1 }{ t } r^{ \alpha - 1 } |\ln \varepsilon|.
\end{gather*}
Finally, we also have 
\begin{equation*}
| I_4^{-,k} | \lesssim \frac{ 1 }{ t } \int\limits_{ \varepsilon }^{ \varepsilon_0 } \frac{ r^{ \alpha } \rho d \rho }{ \rho r } \lesssim \frac{ r^{ \alpha -1 } }{ t }.
\end{equation*}

Putting $ \varepsilon = \frac{ 1 }{ \sqrt{ t r } } $, we obtain the following bound for $ I_{ R } $:
\begin{equation*}
| I_R | \lesssim \frac{ \ln t }{ t }.
\end{equation*}
\end{enumerate}

\bigskip

\section{Strichartz estimates}\label{Sect_7}

\medskip

Now it is time to collect all previous estimates from Sections 2 to 6. Let $E>0$, 
\be\label{I_beta}
I (t,u;E):= \int\limits_{ \mathbb{C} } | \xi |^{ \alpha+i\bt }  e^{ i t \tilde S( u, \xi ;E) } d \Re \xi d \Im \xi,   \qquad (\hbox{cf. } \eqref{I})
\ee
where
\[
 u = \frac{z}t,
\]
and (cf. \eqref{st_phase})
\[
\tilde S( u, \xi ; E) = ( \xi^3 + \bar \xi^3 )\left( 1 - \frac{ 3 E }{ |\xi|^2 } \right) + \frac{ 1 }{ 2 }( \bar u \xi + u \bar \xi ).
\]
Note that
\begin{equation*}
p( \xi ) = ( \xi^3 + \bar \xi^3 )\left( 1 - \frac{ 3 E }{ |\xi|^2 } \right)
\end{equation*}
is the symbol of the linear part of the NV equation \eqref{NV}. In Section \ref{disp_section} (see Propositions \ref{all_freq_prop}, \ref{Smoothing_Large}) we have proved the following result.

\begin{lem}[Dispersion estimate for positive energy]
\label{asp_lin_estimate_lemma}
The following two dispersive estimates are satisfied.
\ben
\item Let $ 0 \leq \alpha \leq \frac{1}{4} $ and $\bt \in \R$. Then for all $ t > 0 $ and for any fixed $\ve > 0$ small, the following estimate is valid:
\begin{equation}\label{I_decrease}
| I (t,u;1)| \lesssim  \frac{(1+|\bt|) }{ t^{3/4 -\ve } },
\end{equation}
uniformly on $ u \in \mathbb{C} $. The implicit constant depends on $\al$, $\ve$ only.

\medskip

\item Let $ 0 \leq \alpha < 1 $ and $\bt \in \R$. Then for all $ t > 0 $ and for any fixed $\ve > 0$ small, the following estimate is valid:
\begin{equation}\label{I_decrease_R}
| I_R (t,u;1)| \lesssim  \frac{(1+|\bt|) }{ t^{1 -\ve } },
\end{equation}
uniformly on $ u \in \mathbb{C} $. The implicit constant depends on $\al$, $\ve$, $ R $ only.
\een
\end{lem}

\begin{rem}
Note that the addition of the constant $i\bt$ in \eqref{I_beta} is required to use complex interpolation theory; however, from the point of view of Propositionw \ref{all_freq_prop}, \ref{Smoothing_Large} and their consequence, Lemma \ref{asp_lin_estimate_lemma}, it requires no additional essential changes.
\end{rem}

\begin{rem}
By making a suitable change of variables, we easily prove a dispersion estimate in the case of arbitrary positive energy $E>0$. We have
\be\label{Ineq_E}
| I (t,u;E)| \, \lesssim  \,  \frac{(1+|\bt|) E^{(4\al-1)/8 -3\ve/2} }{ t^{ 3/4-\ve }  },
\ee
with constants independent of $E$. Indeed, given \eqref{I} with phase \eqref{st_phase}, we are reduced to the case $E=1$ by performing the change of variables  $\xi= E^{1/2} \xi_0 $, $t_0= E^{3/2}t$, and $ z_0=E^{1/2} z$, which converts \eqref{I} into
\[
\int\limits_{ \mathbb{C} } E^{\al/2}| \xi _0|^{ \alpha }  e^{ i t_0 \tilde S( u_0, \xi_0;1) } E ~d \Re \xi_0 \, d \Im \xi_0,\quad u_0 := \frac{z_0}{t_0}.
\]
Using \eqref{I_decrease} we are lead to
\bee
|I(t,u;E)|  & \lesssim  & \frac{  (1+|\bt|) E^{\frac{\al+2}{2}} }{ t_0^{\frac{3}4 -\ve}} \\
& \sim & (1+|\bt|)  E^{\frac{\al+2}{2}}  \frac{ E^{-\frac{9}{8} -\frac32\ve}}{ t^{ \frac{3}4-\ve } } \\
&  \lesssim & (1+|\bt|) \frac{ E^{\frac{4\al-1}8 -\frac 32\ve} }{ t^{ \frac{3}4 -\ve} },
\eee
as desired. Note that  in the formal limit $E\to 0$, estimate \eqref{Ineq_E} becomes singular.
\end{rem}

\medskip

Let us consider now \eqref{I_decrease_R}. With this purpose, define
\begin{equation*}
I_R( t, u; E ) = \int\limits_{ \mathbb{C} } | \xi |^{ \alpha+i\bt } \psi_R( | E |^{ - 1/2 } \xi ) e^{ i t \tilde S( u, \xi ;E) } d \Re \xi d \Im \xi.
\end{equation*}
By a similar argument as in the computation of \eqref{Ineq_E}, we obtain the following estimate:
\begin{equation*}
| I_R( t, u; E ) | \lesssim \frac{ ( 1 + | \beta | ) | E |^{ ( \alpha - 1 )/2 - 3 \ve / 2 } }{ t^{ 1 - \ve } }.
\end{equation*}

\medskip
Some useful consequences of Lemma \ref{asp_lin_estimate_lemma} are stated in the following lines. Define the operator $ P_R $ as follows:
\begin{equation*}
P_R u = \mathcal{F}^{-1}[ \psi_R( | E |^{ -1/2 } \xi ) \mathcal{F}[u]( \xi ) ],
\end{equation*}
where $ \psi_R $ is defined in section \ref{large_freq}, for some fixed $ R > 2 $.

\begin{cor}[Smoothing and Strichartz estimates]
Let $U(t)=U(t;E)$ be the associated $NV_+$ linear group, namely for $\xi =\xi_1 +i\xi_2$,
\[
U(t)v_0 :=  \int_{\R^2} e^{ i t \tilde S( u, \xi ,E) } \hat{v}_0(\xi_1,\xi_2)d\xi_1 d\xi_2,
\]
$($cf. \eqref{I} and \eqref{st_phase}$)$.  Then for any $0\leq \al\leq \frac{1}{4}$, $\theta\in [0,1]$, we have the following smoothing decay and Strichartz estimates:
\be\label{Dispersion}
\| |\partial_z|^{\al \theta} U(t)v_0 \|_{L^p_{x,y}} \lesssim_{\al, \theta}  \frac{ E^{\theta \big( \frac{4 \al - 1}8- \frac32 \ve\big)} }{ t^{ \theta\big(\frac{3 }4 - \ve \big)} } \|v_0\|_{L^{p'}_{x,y}},
\ee
\be\label{Smoothing}
\| |\partial_z|^{\al \theta/2} U(t)v_0 \|_{L^q_tL^p_{x,y}} \lesssim_{\al, \theta}   E^{ \frac{\theta}{2}\big(\frac{4 \al -1}{8}-\frac32\ve\big)} \|v_0\|_{L^{2}_{x,y}},
\ee
with $p=\frac{2}{1-\theta}$, $\frac1p+\frac1{p'}=1$, and $\frac 2q= \theta \big( \frac{3}4 -\ve \big)$.

\medskip

In addition, for any $ 0 \leq \alpha < 1 $, $ \theta \in [0,1] $ the following estimates are valid:
\be\label{Dispersion_large}
\| |\partial_z|^{\al \theta} P_R U(t)v_0 \|_{L^p_{x,y}} \lesssim_{\al,\theta}  \frac{ E^{\theta \big( \frac{\al-1}8- \frac32 \ve\big)} }{ t^{ \theta(1 - \ve)} } \|v_0\|_{L^{p'}_{x,y}},
\ee
\be\label{Smoothing_large}
\| |\partial_z|^{\al \theta/2} P_R U(t)v_0 \|_{L^q_tL^p_{x,y}} \lesssim_{\al,\theta}   E^{ \frac{\theta}{2}\big(\frac{\al-1}{8}-\frac32\ve\big)} \|v_0\|_{L^{2}_{x,y}},
\ee
with $p=\frac{2}{1-\theta}$, $\frac1p+\frac1{p'}=1$, and $\frac 2q= \theta ( 1 -\ve )$.

\end{cor}

\begin{rem}
Note that, in comparison with analogous estimates for Zakharov-Kuznetsov equation in \cite{LP}, Proposition 2.2, we have less smoothness in general ($1/8$ of derivative instead of $1/4$ for ZK), but more smoothness if we exclude low frequencies (almost $1/2$ of derivative).
\end{rem}

\begin{proof}
The proofs  are by now standard, but we present them for the sake of completeness. See e.g. \cite{SW} Theorem 4.1, \cite{KPV_Indiana} and references therein for detailed proofs.

\medskip

First,  we prove \eqref{Dispersion}. This is just a consequence of the standard complex interpolation theorem. We have from \eqref{Ineq_E} and Young's inequality,
\[
\| |\partial_z|^{\al +i\bt} U(t)v_0 \|_{L^\infty}  \lesssim  \, \frac{ (1+|\bt|) E^{\frac{4\al-1}8 -\frac32\ve} }{ t^{ \frac{3}4 -\ve } } \|v_0\|_{L^1},
\]
for any $0\leq \al \leq \frac{1}{4}$ and $\bt\in \mathbb{R}$. We interpolate against the trivial estimate
\[
\| |\partial_z|^{i\bt} U(t)v_0 \|_{L^2} = \|v_0\|_{L^2},
\]
to get \eqref{Dispersion}.

The estimate \eqref{Dispersion_large} is obtained similarly by interpolating 
\begin{equation*}
\| |\partial_z|^{\al +i\bt} P_R U(t)v_0 \|_{L^\infty}  \lesssim  \, \frac{ (1+|\bt|) E^{\frac{\al-1}2 -\frac32\ve} }{ t^{ 1 -\ve } } \|v_0\|_{L^1},
\end{equation*}
and 
\begin{equation*}
\| |\partial_z|^{i\bt} P_R U(t)v_0 \|_{L^2} \leq \|v_0\|_{L^2}.
\end{equation*}
\medskip

Now we prove \eqref{Smoothing}. By duality, we are lead to prove that
\[
\int {\phi}(t,x,y)  |\partial_z|^{\al \theta/2} U(t)v_0 \, dtdxdy \lesssim  E^{\frac{\theta}{2} \big(\frac{\al-1}{8} -\frac32 \ve\big)} \|v_0\|_{L^2} \|\phi\|_{L^{q'}_t L_{x,y}^{p'}},
\]
for all $\phi\in C^\infty_0(\R^3)$ and with $ \frac1q + \frac1{q'} = 1 $. However,
\bee
\int {\phi}(t,x,y)  |\partial_z|^{\al \theta/2} U(t)v_0  \, dtdxdy & =&  \int v_0 \Big[ \int {  |\partial_z|^{\al \theta/2}  U(-t)\phi(t,x,y)}  dt \Big]dxdy\\
&\leq &  \|v_0\|_{L^2} \norm{ \int {  |\partial_z|^{\al \theta/2}  U(-t) {\phi(t)}}  dt}_{L^2}.
\eee
Now,
\bee
 \norm{ \int {  |\partial_z|^{\al \theta/2}  U(-t)\phi(t)}  dt}_{L^2}^2 &=& \int_t \!\int_{t'} \!\int_{x,y} {  |\partial_z|^{\al \theta/2}  U(-t')\phi(t',x,y)}    |\partial_z|^{\al \theta/2}  U(-t)\phi(t,x,y) dxdy dt' dt  \\
 & =&  \int_{x,y} \!\int_{t} \!\int_{t'}   |\partial_z|^{\al \theta}  U(t'-t)\phi(t',x,y) \phi(t,x,y)  dt' dt dxdy \\
 & \leq & \norm{ \int_{t'}   |\partial_z|^{\al \theta}  U(t'-\cdot )\phi(t')dt' }_{L^q_t L^p_{x,y}} \|\phi\|_{L^{q'}_t L^{p'}_{x,y}}.
\eee
On the other hand, we have from \eqref{Dispersion} and the Hardy-Littlewood-Sobolev's inequality,
\bee
 \norm{ \int_{t'}   |\partial_z|^{\al \theta}  U(t'-\cdot )\phi(t')dt' }_{L^q_t L^p_{x,y}} & \leq &   \norm{ \int_{t'} \norm{  |\partial_z|^{\al \theta}  U(t' - \cdot)\phi(t') }_{L^p_{x,y}} dt' }_{L^q_t}\\
 & \lesssim &  E^{\theta\big(\frac{4\al-1}{8} -\frac32\ve\big)} \norm{ \int_{ t' } \frac{  \norm{ \phi(t') }_{L^{p'}_{x,y}} }{ |t-t'|^{ \theta \big(\frac{3}4 -  \ve\big)} } dt' }_{L^q_t}\\
& \lesssim &   E^{\theta \big(\frac{4\al-1}{8}-\frac32\ve\big)} \norm{ \phi }_{L^{q'}_t L^{p'}_{x,y} }.
\eee
Estimate \eqref{Smoothing_large} is obtained in a similar way using the fact that $ U(t) $ and $ P_R $ commute.
\end{proof}

Another set of standard but useful estimates is the following (see  \eqref{Xsb} for a definition  of $X_E^{s,b}$ spaces)
\begin{cor}
We have
\be\label{Smoothing1}
\| f \|_{L^{4}_{t,x,y}} \lesssim  E^{-1/32^+} \| f \|_{X_E^{0,\frac7{16}+}}, \quad \forall f \in X_E^{0,\frac7{16}+},
\ee
and
\be\label{Smoothing2}
\| |\partial_z|^{1/4^-}P_R f \|_{L^{4}_{t,x,y}} \lesssim E^{-0^+} \| f \|_{X_E^{0,\frac12^-}}, \quad \forall f \in X_E^{0,\frac12^-}.
\ee
The constant in the last inequality becomes singular as $\frac14^-$ approaches $\frac14.$
\end{cor}
\begin{proof}
Put $\al=0$, $\theta= \big( \frac74-\ve\big)^{-1} =\frac47^+$ in \eqref{Smoothing}, we get
\[
\| U(t)v_0 \|_{L^{r}_{t,x,y}} \lesssim E^{ \gamma}  \|v_0\|_{L^{2}_{x,y}},
\]
with $ r = \frac{ 14 }{ 3 } \left( \frac{ 1 - \frac{ 4 }{ 7 } \ve }{ 1 - \frac{ 4 }{ 3 } \ve } \right) = \frac{ 28 }{ 3 }^+ $ and $ \gamma = - \frac{ 1 }{ 28 } \left(\frac{ 1- 12 \ve }{ 1 - \frac{ 4 }{ 7 } \ve } \right) =  -\frac{ 1 }{ 28 }^+ $.

From the transfer principle (see e.g. \cite[Lemma 3.3]{G}), we have
\[
\| f \|_{L^{r}_{t,x,y}} \lesssim  E^{\gamma}  \| f \|_{X_E^{0,\frac12+}}, \quad \forall f \in X_E^{0,\frac12+}.
\]
By interpolation with the trivial estimate for $\| f \|_{L^{2}_{t,x,y}}$, using the interpolation parameter $\lambda = \frac{ 7 }{ 8 }( 1 - \frac{ 4 }{ 7 } \ve ) = \frac{ 7 }{ 8 }^-$, so that $ \lambda \frac{ 1 }{ r } + ( 1 - \lambda ) \frac{1}{2} = \frac{1}{4} $, we obtain \eqref{Smoothing1}.

\medskip

Now we deal with \eqref{Smoothing2}. Taking $\al=1- 16\ve$ and $\theta=\frac{1}{2-5\ve}=\frac12^+$ in \eqref{Smoothing_large} we have $p=q =4 \frac{1-\frac52 \ve}{1-5\ve} =4^+$. We also have $\frac12\al\theta = \frac{1-16\ve}{4(1-\frac52\ve)} =\frac14^-$ and
\[
\| |\partial_z|^{\frac14^-} U(t)v_0 \|_{L^{r}_{t,x,y}} \lesssim E^{-0^+} \|v_0\|_{L^{2}_{x,y}},
\]
with $ r = 4 \frac{1-\frac52 \ve}{1-5\ve} $, from where
\[
\| |\partial_z|^{\frac14^-} f \|_{L^{r}_{t,x,y}} \lesssim  E^{-0^+}   \| f \|_{X_E^{0,\frac12+}}.
\]
We interpolate with $L^2_{t,x,y}$-estimate with the interpolation parameter $\lambda= 1 - \frac52 \ve $ to obtain
\[
\| |\partial_z|^{1/4^-} f \|_{L^{4}_{t,x,y}} \lesssim  E^{-0^+} \| f \|_{X_E^{0, \frac12^-}}, \quad \forall f \in X_E^{0, \frac12^-}.
\]
\end{proof}

\section{Bilinear Estimates}\label{Sect_8}

%

\subsection{Some notation}

We use the following notations. We denote by $ \langle f\rangle $ the japanese bracket:
\[
\langle f\rangle := (1+|f|^2)^{1/2};
\]
we also have $ A \land B : = \min( A, B ) $ and $ A \lor B : = \max( A, B ) $. Variables $ N, \tilde N, \check N, L, \tilde L, \check L $ of this section are \emph{dyadic}, i.e. their range is $ \{ 2^k, k \in \mathbb{N} \} $. \emph{In this section only and unless otherwise specified} $\mathcal F[ u ]$ and $ \hat u $ both denote the Fourier transform of $ u $ in $(t,x,y)$.

\medskip

Now we introduce the associated $X^{s,b}$ spaces \cite{B1} for the NV dynamics.  Let $\tilde\varphi \in C_0^\infty(\R;[0,1])$ be a cutoff function such that
\[
\tilde\varphi(s) =0 \quad \hbox{ for } \quad |s| \geq 1, \qquad \tilde\varphi(s) = 1 \quad \hbox{ for } \quad |s|\leq \frac12.
\]
We define
\[
\varphi(s):= \tilde\varphi(s) -\tilde\varphi(2s).
\]
We introduce the frequency projection operators at a dyadic frequency $N >1$, as follows:
\[
\varphi_N (s):= \varphi(N^{-1} s),
\]
and for $N=1$,
\[
\varphi_1 (s):= \tilde\varphi(s).
\]
Using these multipliers, we have for the Fourier transform in the $\xi$ variable,
\[
P_N u := \mathcal F^{-1}\Big[\varphi_N (E^{-1/2} |\xi|) \mathcal F[u](\xi)\Big].
\]
Recall the definition of the phase $\tilde S(u,\xi)$ in \eqref{st_phase}. We define (recall that $E>0$)
\[
w(\xi,\bar\xi;E) := ( \xi^3 + \bar \xi^3 )\left( 1 - \frac{ 3 E }{ |\xi|^2 } \right).
\]
Note that the above expression is real-valued. In order to perform some Fourier analysis, we need $w$ in terms of real-valued coordinates. Put $\xi=\xi_1+i\xi_2$. We have the rational symbol
\be\label{w0}
w (\xi, \bar \xi) := w(\xi,\bar\xi; E) = 2( \xi_1^3 -3\xi_1\xi_2^2)\left( 1 - \frac{ 3E}{ \xi_1^2+\xi_2^2 } \right).
\ee
We also define
\[
\sigma(\tau,\xi_1,\xi_2):= \tau -w(\xi_1,\xi_2).
\]
We introduce then
\[
Q_L u := \mathcal F^{-1}[ \varphi_L(E^{-3/2}|\sigma|) \mathcal F[u](\tau,\xi)].
\]
Finally, for a fixed energy $E$, we say that $u=u(t,x,y) \in X_E^{s,b}$ for $s,b\in \R$ if $u\in L^2(\R^3)$ and its Fourier transform $\hat u$ satisfies the integral condition
\be\label{Xsb}
\|u\|^2_{X_E^{s,b}} := \int_{\R^3} \langle \sigma\rangle^{2b} \langle |\xi| \rangle^{2s} |\hat u(\tau,\xi_1,\xi_2)|^2 d\tau d\xi_1d\xi_2 <+\infty.
\ee

\subsection{Statement of the result and proof}

\begin{prop}
Assume that $E>0$ is a fixed level of energy. Then we have for $\ve>0$ small and $s>\frac12$,
\be\label{Bilinear}
\|\partial_z(vw)\|_{X_E^{s,-1/2 + 2\ve}} \lesssim  E^{(7^--8s)/16}  ~\|v\|_{X_E^{s,1/2+ \ve}} \|w\|_{X_E^{s,1/2+ \ve}},
\ee
for all $v,w$ such that the right hand side makes sense.
\end{prop}

\begin{rem}
The exponent $(7^--8s)/16$ is an artefact of the proof, and we believe that better exponents can be obtained by estimating in a different form some crucial terms below.
\end{rem}

\begin{proof}
We follow closely the ideas from \cite{MP}, using  a modified version of the original ideas by Bourgain \cite{B1} and Kenig, Ponce and Vega \cite{KPV1}. As usual, by duality we are lead to prove that
\[
J:= \int_{\R^6} K[\tau,\tilde\tau,\xi,\tilde\xi]\hat u(\tau,\xi_1,\xi_2)\hat v(\check\tau,\check\xi_1,\check\xi_2)\hat w(\tilde\tau,\tilde\xi_1,\tilde\xi_2)\, d\xi_1 \, d\tilde\xi_1\, d\xi_2 \, d\tilde\xi_2 \, d\tau \, d\tilde\tau
\]
satisfies
\[
|J| \lesssim \|u\|_{L^2} \|v\|_{L^2} \|w\|_{L^2},
\]
for any $u,v,w\in L^2(\R^3)$, and where $\tilde\sigma := \sigma( \tilde\tau,\tilde \xi)$,
\be\label{hat_tilde}
\check\tau := \tilde\tau -\tau, \quad \check \xi: =\tilde \xi-\xi,\quad \check\sigma:= \tilde\sigma -\sigma.
\ee
Here the kernel $K=K[\tau,\tilde\tau,\xi,\tilde\xi]$ is explicitly given by
\[
K:=|i\tilde\xi_1+\tilde\xi_2| |i\check\xi_1+\check \xi_2||-i \check \xi_1+\check \xi_2|^{-1} \langle \tilde\sigma\rangle^{-1/2+2\ve}\langle \sigma\rangle^{-1/2-\ve}\langle \check \sigma\rangle^{-1/2-\ve}\langle |\tilde \xi| \rangle^{s} \langle |\xi|\rangle^{-s}\langle |\check \xi|\rangle^{-s}.
\]
A further simplification leads to
\be\label{KernelK}
K=|\tilde \xi|  \langle \tilde\sigma\rangle^{-1/2+2\ve}\langle \sigma\rangle^{-1/2-\ve}\langle \check\sigma\rangle^{-1/2-\ve}\langle |\tilde \xi| \rangle^{s} \langle |\xi|\rangle^{-s}\langle |\check \xi|\rangle^{-s}.
\ee
Now we use dyadic decompositions to split  $J$ into several pieces. We put
\[
J= \sum_{N,\tilde N,\check N} J_{N,\tilde N,\check N},
\]
where
\be\label{JNNN}
J_{N,\tilde N,\check N} := \int_{\R^6} K[\tau,\tilde\tau,\xi,\tilde\xi] \widehat{P_Nu}(\tau,\xi_1,\xi_2)\widehat{P_{\check N}v}(\check\tau,\check\xi_1,\check\xi_2)\widehat{P_{\tilde N}w}(\tilde\tau,\tilde\xi_1,\tilde\xi_2)d\xi_1d\tilde\xi_1d\xi_2d\tilde\xi_2 d\tau d\tilde\tau.
\ee

\noindent
\emph{Estimate for low-low to low frequencies.} This is the simplest case. Here we have $N\sim \tilde N \sim \check N$, where $N\sim 1$.  From \eqref{JNNN} we have to estimate the quantity
\[
\sum_{N\sim \tilde N \sim \check N \sim 1} J_{N,\tilde N,\check N}.
\]
Note that we also have $|\xi|\sim |\tilde \xi| \sim |\check \xi| \sim E^{1/2}$. Therefore, using \eqref{KernelK} we have
\bee
|K|  & \lesssim &   E^{(1-s)/2} \langle \tilde\sigma\rangle^{-1/2+2\ve}\langle\sigma\rangle^{-1/2-\ve}\langle \check\sigma\rangle^{-1/2-\ve}\\
& \lesssim &  E^{(1-s)/2} \langle\sigma\rangle^{-1/2-\ve}\langle \check\sigma\rangle^{-1/2-\ve}.
\eee
Therefore, from \eqref{JNNN}, Plancherel, and using Cauchy-Schwarz,
\[
|J_{N,\tilde N,\check N}| \lesssim  E^{(1-s)/2}  \norm{\mathcal F^{-1}[\langle\sigma\rangle^{-1/2-\ve}  \widehat{P_Nu}]   }_{L^4}\norm{\mathcal F^{-1}[\langle\check \sigma\rangle^{-1/2-\ve}   \widehat{P_{\check N}v} ] }_{L^4} \|P_{\tilde N}w\|_{L^2}.
\]
From \eqref{Smoothing1} and the definition of $X_E^{s,b}$ space norm, we obtain for $\ve>0$ small,
\bee
\norm{\mathcal F^{-1}[\langle\sigma\rangle^{-1/2-\ve}  \widehat{P_Nu}]   }_{L^4} & \lesssim &  E^{-1/32^+}  \norm{\mathcal F^{-1}[\langle\sigma\rangle^{-1/2-\ve}  \widehat{P_Nu}]   }_{X_E^{0,\frac7{16}+}} \\
& \lesssim &   E^{-1/32^+} \| \langle\sigma\rangle^{\frac7{16}^+} \langle\sigma\rangle^{-1/2-\ve}\widehat{P_Nu} \|_{L^2} \\
&  \lesssim &  E^{-1/32^+} \|P_Nu \|_{L^2}.
\eee
We conclude that
\[
\sum_{N\sim \tilde N \sim \check N \sim 1} J_{N,\tilde N,\check N} \lesssim  E^{(7^--8s)/16} \|u \|_{L^2}\|v \|_{L^2}\|w \|_{L^2}.
\]

\medskip

\noindent
\emph{Estimate for low-high to high frequencies.} The worst case here corresponds to $N\ll \tilde N$ and $\tilde N \sim \check N$, where $\tilde N \gg 1$.  From \eqref{JNNN} we have to estimate the quantity
\[
\sum_{N\ll \tilde N, \, \tilde N \sim \check N   \gg 1} J_{N,\tilde N,\check N}.
\]
Unfortunately, a crude estimate for $K$ in \eqref{KernelK} shows that it is not possible to counterbalance the term $|\tilde \xi| \sim E^{1/2} \tilde N$.  For this reason a dyadic decomposition on the modulation variables is necessary. Using the frequency localization operator $Q_{L}$, $Q_{\tilde L}$ and $Q_{\check L}$ on dyadic shells $\sigma \sim E^{3/2} L$, $\tilde \sigma \sim E^{3/2}\tilde L$ and so on,\footnote{The power $3/2$ of the energy in front of the modulation variable $\sigma$ is dictated by the time scaling.} we have
\[
J_{N,\tilde N,\check N} =\sum_{L,\tilde L,\check L} J_{N,\tilde N,\check N}^{L,\tilde L,\check L},
\]
where
\bea\label{JNNNLLL}
J_{N,\tilde N,\check N}^{L,\tilde L,\check L} &  := & \int_{\R^6} K[\tau,\tilde\tau,\xi,\tilde\xi] \mathcal F[P_NQ_L u](\tau,\xi_1,\xi_2) \times \nonu\\
& & \qquad \times \mathcal F[P_{\check N}Q_{\check L}v](\check\tau,\check\xi_1,\check\xi_2)\mathcal F[P_{\tilde N}Q_{\tilde L}w](\tilde\tau,\tilde\xi_1,\tilde\xi_2)\, d\xi_1d\tilde\xi_1d\xi_2d\tilde\xi_2 d\tau d\tilde\tau.
\eea
We readily have in the considered region
\[
|K[\tau,\tilde\tau,\xi,\tilde\xi]| \lesssim E^{-(s+7)/4} \tilde N N^{-s}  \tilde L^{-1/2+2\ve}L^{-1/2-\ve}\check L^{-1/2-\ve},
\]
and using Cauchy-Schwarz, the fact that
\[
\int_{\tau,\xi,\tilde \tau,\tilde \xi}\mathcal F[P_NQ_L u](\tau,\xi_1,\xi_2)\mathcal F[P_{\check N}Q_{\check L}v](\check\tau,\check\xi_1,\check\xi_2)
\]
represents a convolution in Fourier variables, and Plancherel, we get
\be\label{Aux_J}
|J_{N,\tilde N,\check N}^{L,\tilde L,\check L}| \lesssim  E^{-(s+7)/4} \tilde N  N^{-s}\tilde L^{-1/2+2\ve}L^{-1/2-\ve}\check L^{-1/2-\ve} \| P_NQ_L u ~ P_{\check N}Q_{\check L}v\|_{L^2}\|P_{\tilde N}Q_{\tilde L}w\|_{L^2}.
\ee
Now we use the following estimate (see \cite{MP} for a similar statement), valid under the assumptions that we work with low-high to high frequencies:
\be\label{Strichartz_Mod}
\| P_NQ_L u ~ P_{\check N}Q_{\check L}w \|_{L^2} \lesssim E^{5/4} N^{1/2} \check N^{-1} L^{1/2} \check L^{1/2}\| P_NQ_L u\|_{L^2} \| P_{\check N}Q_{\check L}w\|_{L^2}.
\ee
This estimate is proved several lines below. For now we assume the validity of this estimate and we continue with the estimation of \eqref{Aux_J}. Note that \eqref{Strichartz_Mod} allows to cancel out the bad frequency $\tilde N$ in \eqref{Aux_J}. We get
\[
|J_{N,\tilde N,\check N}^{L,\tilde L,\check L}| \lesssim  E^{-(2+s)/4} N^{1/2-s}\tilde L^{-1/2+2\ve}L^{-\ve}\check L^{-\ve}\| P_NQ_L u\|_{L^2} \| P_{\check N}Q_{\check L}v\|_{L^2}\|P_{\tilde N}Q_{\tilde L}w\|_{L^2}.
\]
Adding up on $N$ (recall that $s>\frac12$), $\tilde L$, $L$ and $\check L$ we obtain
\bee
\sum_{N\ll \tilde N, \,\tilde N \sim \check N \gg 1} J_{N,\tilde N,\check N} & \lesssim  & E^{-(2+s)/4} \|u\|_{L^2} \sum_{\tilde N\sim \check N}  \| P_{\check N}v\|_{L^2}\|P_{\tilde N}w\|_{L^2} \\
& \lesssim  & E^{-(2+s)/4} \|u\|_{L^2}\|v\|_{L^2}\|w\|_{L^2}.
\eee
Let us prove \eqref{Strichartz_Mod}. Following \cite{MP} (see some previous ideas in \cite{ST}, for the periodic KPI case), and using the Plancherel's identity, together with Young's inequality for convolutions, we have
\bea
\| P_NQ_L u ~ P_{\check N}Q_{\check L}w \|_{L^2} & \sim & \| \mathcal F[P_NQ_L u] \star \mathcal F[ P_{\check N}Q_{\check L}w] \|_{L^2} \nonu\\
& \lesssim & \sup_{(\tilde \tau,\tilde\xi)\in\R^3} (\meas A_E(\tilde\tau,\tilde\xi))^{1/2}~ \| P_NQ_L u\|_{L^2} \|P_{\check N}Q_{\check L}w\|_{L^2}, \label{Est_1}
\eea
where $A_E(\tilde\tau,\tilde\xi)$ is the set (see \eqref{hat_tilde})
\[
A_E(\tilde\tau,\tilde\xi):=\big\{(\tau,\xi) \in \R^3 \ : \ |\sigma| \sim E^{3/2} L, \; |\check \sigma| \sim  E^{3/2} \check L, \;  |\xi| \sim  E^{1/2} N,  \; |\check \xi| \sim  E^{1/2} \check N \big\}.
\]
The measure of this set can be estimated as follows:
\be\label{meas_A}
\meas A_E(\tilde\tau,\tilde\xi) \lesssim  E^{3/2}(L \land \check L)  \meas B_E(\tilde\tau,\tilde\xi),
\ee
where
\[
B_E(\tilde\tau,\tilde\xi) := \big\{ (\tau,\xi) \in \R^3 \ : \ |\tilde \sigma + H[\xi,\tilde \xi] | \lesssim  E^{3/2}(L\lor \check L), \;  |\xi| \sim  E^{1/2} N,  \; |\check \xi| \sim  E^{1/2} \check N \big\},
\]
and $H$ is the standard  resonance function (see \eqref{w0})
\[
H[\xi,\tilde \xi] := w(\tilde \xi, \overline{\tilde \xi}) - w(\xi, \overline \xi) -w(\tilde \xi -\xi,\overline{\tilde \xi} -\overline\xi).
\]
We use now the idea from \cite[Lemma 3.8]{MP}: we estimate the measure of $B_E$ by finding lower bounds on the derivatives of the function $\xi \mapsto  H[\xi, \tilde \xi]$ in the considered region.

\medskip

As in \cite{A}, we have to distinguish between two cases: $|\tilde \xi_1-\xi_1|\sim |\tilde \xi_2-\xi_2| \sim E^{1/2} \tilde N \sim E^{1/2} \check N$ and $|\tilde \xi_1 -\xi_1|\gg |\tilde \xi_2 -\xi_2| $ (the remaining case is identical). For the first case we have
\bee
\partial_{\xi_2} H[\xi,\tilde \xi] & =&  -2\partial_{\xi_2} \Bigg[  ( \xi_1^3 -3\xi_1\xi_2^2)\left( 1 - \frac{ 3E}{ \xi_1^2+\xi_2^2 } \right) \\
&& \qquad + ( (\tilde\xi_1-\xi_1)^3 -3(\tilde\xi_1-\xi_1)(\tilde\xi_2-\xi_2)^2)\left( 1 - \frac{ 3E}{ (\tilde\xi_1-\xi_1)^2+(\tilde\xi_2-\xi_2)^2 } \right)\Bigg] \\
& =& -12\Bigg[  -\xi_1\xi_2\left( 1 - \frac{ 3E}{ \xi_1^2+\xi_2^2 } \right)  +(\tilde\xi_1-\xi_1)(\tilde\xi_2-\xi_2)\left( 1 - \frac{ 3E}{ (\tilde\xi_1-\xi_1)^2+(\tilde\xi_2-\xi_2)^2 } \right) \\
& &\qquad    + \frac{ E\xi_1 \xi_2( \xi_1^2 -3\xi_2^2)}{( \xi_1^2+\xi_2^2)^2 }   - \frac{E(\tilde\xi_1-\xi_1)(\tilde\xi_2-\xi_2)  ( (\tilde\xi_1-\xi_1)^2 -3(\tilde\xi_2-\xi_2)^2)}{( (\tilde\xi_1-\xi_1)^2+(\tilde\xi_2-\xi_2)^2)^2 } \Bigg].
\eee
Using the estimate
\[
\abs{\frac{a b }{ a^2+ b^2 }} \leq \frac12,
\]
and similar other estimates for the fractional terms appearing from the fact that we work with nonzero energies, and valid for all $(a,b)\in \R^2$ (the limit at the origin is not well-defined, but the functions are always bounded), we get
\[
|\partial_{\xi_2} H[\xi,\tilde \xi]| \sim E \check N^2.
\]
In the second case, we have no problems since
\bee
\partial_{\xi_1} H[\xi,\tilde \xi] & =& -2\partial_{\xi_1} \Bigg[  ( \xi_1^3 -3\xi_1\xi_2^2)\left( 1 - \frac{ 3E}{ \xi_1^2+\xi_2^2 } \right) \\
&& \qquad + ( (\tilde\xi_1-\xi_1)^3 -3(\tilde\xi_1-\xi_1)(\tilde\xi_2-\xi_2)^2)\left( 1 - \frac{ 3E}{ (\tilde\xi_1-\xi_1)^2+(\tilde\xi_2-\xi_2)^2 } \right)\Bigg] \\
& =& -6 \Bigg[  ( \xi_1^2 - \xi_2^2)\left( 1 - \frac{ 3E}{ \xi_1^2+\xi_2^2 } \right) \\
&&\qquad  - ( (\tilde\xi_1-\xi_1)^2 - (\tilde\xi_2-\xi_2)^2) \left( 1 - \frac{ 3E}{ (\tilde\xi_1-\xi_1)^2+(\tilde\xi_2-\xi_2)^2 } \right) \\
& &\qquad    + \frac{ 2E\xi_1^2( \xi_1^2 -3\xi_2^2)}{( \xi_1^2+\xi_2^2)^2 }   - \frac{2E(\tilde\xi_1-\xi_1)^2  ( (\tilde\xi_1-\xi_1)^2 -3(\tilde\xi_2-\xi_2)^2)}{( (\tilde\xi_1-\xi_1)^2+(\tilde\xi_2-\xi_2)^2)^2 } \Bigg],
\eee
so that
\[
|\partial_{\xi_1} H[\xi,\tilde \xi]| \sim E \check N^2.
\]
We conclude (see \cite{MP} for example) that
\[
\meas B_E(\check\tau,\check\xi) \lesssim E N  \check N^{-2} (L\lor \check L).
\]
Finally, from \eqref{meas_A} and \eqref{Est_1} we conclude.

\medskip

\noindent
\emph{Estimate for high-high to low $J_{HH\to L}$, and for high-high to high frequencies $J_{HH\to H}$.} This is the difficult part of the proof, because for obtaining \eqref{Bilinear} with $s>\frac12$ we do not have the corresponding Carbery-Kenig-Ziesler \cite{CKZ} result. Instead, we use the corresponding smoothing estimate Lemma \ref{asp_lin_estimate_lemma} which suffices for the case of positive energies.

\medskip

We prove the most difficult case, the one for  high-high to high frequencies (see below for a comment on the case high-high to low). Here we have $N\sim \tilde N \sim \check N$, where $N\gg 1$.  From \eqref{JNNN} we have to estimate the quantity
\be
\sum_{N\sim \tilde N \sim \check N \gg 1} J_{N,\tilde N,\check N}.
\ee
Note that we also have $|\xi|\sim |\tilde \xi| \sim |\check \xi| \sim E^{1/2}N$. Therefore, using \eqref{KernelK} we have
\bee
|K|  & \lesssim &  E^{-s/2}  |\xi| N^{-s}  \langle \tilde\sigma\rangle^{-1/2+2\ve}\langle\sigma\rangle^{-1/2-\ve}\langle \check\sigma\rangle^{-1/2-\ve}\\
& \lesssim &  E^{-s/2 +1/4^+}  N^{1/2^+-s}  |\xi|^{1/4^-}\langle\sigma\rangle^{-1/2-\ve}|\check\xi|^{1/4^-}\langle \check\sigma\rangle^{-1/2-\ve}.
\eee
Therefore, from \eqref{JNNN}
\bee
|J_{N,\tilde N,\check N}| & \lesssim & E^{-s/2 +1/4^+} N^{1/2^+-s} \norm{|\partial_z|^{1/4^-}\mathcal F^{-1}[\langle\sigma\rangle^{-1/2-\ve}  \widehat{P_Nu}]   }_{L^4} \times\\
& & \qquad \times \norm{|\partial_z|^{1/4^-}\mathcal F^{-1}[\langle\check \sigma\rangle^{-1/2-\ve}   \widehat{P_{\check N}v} ] }_{L^4} \|w\|_{L^2}.
\eee
From \eqref{Smoothing2} we get
\bee
|J_{N,\tilde N,\check N}| & \lesssim &E^{-s/2 +1/4^+}   N^{1/2^+-s} \norm{\mathcal F^{-1}[\langle\sigma\rangle^{-1/2-\ve}  \widehat{P_Nu}]   }_{X_E^{0,\frac12^-}}  \times \\
&& \qquad \times \norm{\mathcal F^{-1}[\langle\check \sigma\rangle^{-1/2-\ve}   \widehat{P_{\check N}v} ] }_{X_E^{0,\frac12^-}} \|w\|_{L^2} \\
& \lesssim &  E^{-s/2 + 1/4^+ - 0^+}  N^{1/2^+-s} \|P_N u\|_{L^2}\|P_{\check N} v\|_{L^2}\|w\|_{L^2}.
\eee
Adding on $N\sim \check N$, we conclude.

\medskip

Finally, some words about the case high-high to low frequencies. In this regime one has $N \sim \check N \gg \tilde N$. Note that we also have $|\xi|\sim |\check \xi| \sim E^{1/2}N$, and $ |\tilde \xi| \sim E^{1/2} \tilde N$.  Now it is enough to consider the following estimate: 
\[
\frac{ | \tilde \xi | \tilde N^s }{ N^s \check N^s } \lesssim \frac{ | \xi | }{ N^s }.
\]
Therefore, using \eqref{KernelK} we have
\bee
|K|  & \lesssim &  E^{-s/2}  |\xi| N^{-s}  \langle \tilde\sigma\rangle^{-1/2+2\ve}\langle\sigma\rangle^{-1/2-\ve}\langle \check\sigma\rangle^{-1/2-\ve}\\
& \lesssim &  E^{-s/2 +1/4^+}  N^{1/2^+-s}  |\xi|^{1/4^-}\langle\sigma\rangle^{-1/2-\ve}|\check\xi|^{1/4^-}\langle \check\sigma\rangle^{-1/2-\ve},
\eee
and the rest of the proof is similar to the previous case.
\end{proof}

\bigskip

\section{Local well-posedness}\label{Sect_9}
\medskip

In this section we prove our main result, Theorem \ref{LWP_pos}. Using the standard iteration scheme based on $X_E^{s,b}$ spaces, (see e.g. \cite{KM} for more details) we will show the following

\begin{thm}\label{Cauchy}
Fix $E>0$. The Cauchy problem for \eqref{NV} is locally well-posed in $H^s(\R^2)$ for $s>\frac12$. Moreover, the existence lifetime $T_E>0$ of a solution $v(t)$ with initial data $v_0$ satisfies
\be\label{Existence_E}
T_E\| v_0\|_{H^s}^\al  \gtrsim  E^{\al(8s-7^-)/16},
\ee
for some positive exponent $\al.$ Finally, one has for all $\eta \in [0,1]$ standard cut-off supported in the interval $[-2,2]$, $\eta \equiv 1$ on $[-1,1]$, and $s>\frac12$, and $T<T_E$,
\be\label{Global_local}
\| \eta(t/T) v(t) \|_{X_E^{s,\frac12+\ve}} \lesssim \| v_0\|_{H^s}.
\ee
\end{thm}

\begin{rem}
Note that in this result we give explicit dependence on the energy $E$, for further developments.
\end{rem}

\begin{rem}
It will be clear from the proof of Theorem \ref{Cauchy} that the local well-posedness does not depend on the sign of $E$; however, obtaining the corresponding bilinear estimates with the sufficient gain of derivatives crucially depends on the sign of $E$. For the case $E=0$, estimates can be handled using \cite{CKZ}, while the case $E<0$ requires only one change of variables in order to get a desired decay. Finally, the case $E>0$ requires a new change of variables and a detailed description of all possible cases arising in decay estimates.
\end{rem}

\begin{proof}[Proof of Theorem \ref{Cauchy}]
The proof is standard, we only check the main lines of the proof. Assume that $v_0\in H^s(\R^2)$. From \eqref{NV} we have the local Duhamel's formula for $t\in [0,1]$,
\[
\eta_T(t) v (t)=\eta_T(t) \mathcal T[v](t) := \eta_T(t)  U(t) v_0 + 2\eta_T(t)\int_0^t U(t-t') [\partial_z(vw)+\partial_{\bar z}(v\bar w)]dt',
\]
where $\eta =\eta(t) \in [0,1]$ is a smooth bump function with $\eta(t) =1$ for $t\in[-1,1]$, and $\eta(t) =0$ for $|t| \geq 2$, and $\eta_T(t):= \eta(t/T)$. From this identity, the standard linear estimates (see e.g. Theorem 10 of \cite{A}) and \eqref{Bilinear} we have, for any $s>\frac12$,
\bee
\| \eta_T v\|_{X^{s,1/2+\ve}_E} & \leq & \|v_0\|_{H^s} + C \|\eta_T\partial_z(vw)\|_{X^{s,-1/2+ \ve}_E} \\
& \leq& \|v_0\|_{H^s} +  CT^\ve \|\eta_T \partial_z(v w)\|_{X^{s,-1/2+ 2\ve}_E} \\
& \leq & \|v_0\|_{H^s} +  C\frac{T^\ve}{ E^{(8s-7^-)/16} } \|\eta_T v\|^2_{X^{s,1/2+ \ve}_E}
\eee
Now we fix any time $T\sim  \Big(\frac{ E^{(8s-7^-)/16}}{8C \|v_0\|_{H^s}} \Big)^{1/\ve}$, and the ball
\[
\mathcal B:= \Big\{ v \in X^{s,1/2+\ve}_E  \ : \   \| \eta_T v\|_{X^{s,1/2+\ve}_E} \leq  2\|v_0\|_{H^s}  \Big\}.
\]
For $v \in \mathcal B$, one has
\[
\| \eta_T \mathcal T [v]\|_{X^{s,1/2+\ve}_E} \leq   2\|v_0\|_{H^s}.
\]
The contraction property is proved in a similar fashion. The proof is complete.
\end{proof}

\bigskip

\section{Blow up in infinite time}\label{Sect_10}
\medskip

Suppose that $n\geq 0$ is a fixed integer. In this short section we prove Theorem \ref{Blow}. For this, and following the ideas in \cite{Chang}, we look for a solution of the form 
\[
Q_{n,0}(t,z,\bar z)= -8 \partial_z \partial_{\bar z} \log \big(1+ P_n(t,z)P_n(t,\overline z) \big) =-8 \partial_z \partial_{\bar z} \log \big(1+ |P_n(t,z)|^2 \big). 
\]
Here $P_n(t,z)$ is a polynomial of degree $n$ on $z$, with {\bf real-valued} coefficients depending on time, usually called a \emph{Gould-Hopper} polynomial. For our purposes, these polynomials\footnote{Here the factors differ because Chang's paper \cite{Chang} uses a $NV_0$ version with different constant coefficients.}  are defined in terms of the complex-valued Airy symbol, for $\la\in \R$,
\[
e^{\la z + 8\la^3 t} =: \sum_{n\geq 0} P_n(t,z) \frac{\la^n}{n!}, \qquad P_n(0,z) =z^n.
\] 
They can be explicitly recast as
\[
P_n(t,z) = n! \sum_{k=0}^{[n/3]} \frac{(8t)^k z^{n-3k}}{k! (n-3k)!}.
\]  
The Gould-Hopper polynomials also satisfy the properties
\be\label{Pol_eq}
(z+ 24 t\partial_z^2) P_{n-1}(t,z) = P_n(t,z), \qquad |P_n(t,z)| \to +\infty \quad \hbox{as} \quad |z|\to \infty, ~ n\geq 1,
\ee
and
\be\label{dPz}
\frac{dP_n}{dz}(t,z) = n P_{n-1}(t,z).
\ee
Using this fact, we have
\[
P_0 =1, \quad P_1 = z, \quad P_2= (z+ 24 t\partial_z^2)z = z^2,
\]
\be\label{P3P4}
P_3= (z+ 24 t\partial_z^2)z^2 = z^3 + 48 t,\quad P_4=  (z+ 24 t\partial_z^2)(z^3+48 t) = z^4+192 tz,
\ee
and so on. Note additionally that these polynomials satisfy the Airy equation
\be\label{Airy}
\partial_t P_n(t,z) = 8 \partial_z^3 P_n(t,z).
\ee
Using these facts, we will prove that $Q_{n,0}$ defines a solution to \eqref{NV} with $E=0$. A simple computation shows that (compare with \eqref{Q10} and \eqref{Q20})
\[
Q_{n,0}(t,z,\bar z) = -\frac{8|\partial_z P_n(t,z)|^2}{(1+|P_n(t,z)|^2)^2}.
\]
Since $P_n$ as degree $n$, and $n\geq 1$, $Q_{n,0}$ defines an $L^1$ function, since it decays like $|z|^{2(n-1)-4n} =|z|^{-2(1+n)}$. Additionally, from \eqref{NV} one has
\be\label{WQ}
\begin{aligned}
W_n(t,z,\bar z) & := -3\partial_{\bar z}^{-1}\partial_z Q_{n,0}(t,z,\bar z) \\
&  =24 \partial_z^2 \log \big(1+ |P_n(t,z)|^2 \big) \\
& =24 \bar P_n(t,z)  \frac{(P_n''(t,z) (1+|P_n(t,z)|^2 ) - P_n'^2(t,z) \bar P_n(t,z) )}{(1+|P_n(t,z)|^2)^2}.
\end{aligned}
\ee
It is not difficult to check that if $P_n$ is a  Gould-Hopper polynomial, then $Q_{n,0}(t,z,\bar z)$ solves $NV_0$, as a tedious but direct verification shows \cite{Chang}. Moreover, we have
\[
\int Q_{n,0}(t,z,\bar z) =-8n\pi.
\] 
Indeed, since $Q_{n,0}$ is solution to $NV_0$, it is enough to consider the case where $t=0$, so that $P_n(0,z) =z^n$. We have
\[
\begin{aligned}
\int Q_{n,0}(t,z,\bar z)  & = -8 \int \frac{ n^2 |z|^{2(n-1)} dz d\bar z}{(1+|z|^{2n})^2} \\
& =  -16\pi n^2 \int_0^\infty \frac{ r^{2n-1} dr}{(1+r^{2n})^2} \\
& = -8 n \pi  .
\end{aligned}
\]
Finally, note that $Q_{n,0} \to 0$ pointwise in space as $t\to +\infty$. Now, assume that $n\geq 3$ so that $P_n$ is truly depending on time.

\begin{lem}
If $n\geq 3$, then there exists a root $z_0(t)$ of $P_n(t,\cdot)$ which satisfies, for all $|t|$ large,
\[
|z_0(t)| \sim |t|^{1/3},
\]
with implicit constant independent of time. 
\end{lem}

\begin{proof}
This is just consequence of the fact that for $n\geq 3$ and $t\neq 0$, Gould-Hopper polynomials can be recast as
\[
P_n(t,z) = t^{n/3} \hat P_n(z/t^{1/3}),
\]
where $\hat P_n$ are Appell's polynomials \cite[eqn. (26)]{Chang}. Appell's polynomials have at least one nonzero root, a consequence of the remark below.
\end{proof}

\medskip

\begin{rem}
Assume that $n\geq 3$. Recall that each Appell's polynomial $\hat P_n(z)$ can be written as only one of the three following alternatives:
\ben
\item If $n=3k$, $k=1,2,\ldots$,  then $ \hat P_n(z)= \tilde P_k(z^3)$, with $\tilde P_k$ a nonzero polynomial of degree $k$ in the variable $u=z^3$;
\item If $n=3k+1$, $k=1,2,\ldots$,  then $ \hat P_n(z)= z \tilde P_k(z^3)$, with $\tilde P_k$ as above;
\item If finally $n=3k+2$, $k=1,2,\ldots$,  then $ \hat P_n(z)= z^2 \tilde P_k(z^3)$, with $\tilde P_k$ as above.
\een
\end{rem}

The following result states that any Gould-Hopper may probably have degenerate, nonzero roots, but at least one must have order less or equal than two.

\begin{lem}\label{degeneracy}
For $n\geq 3$, the polynomials $\hat P_n$ have at least one nonzero root of multiplicity at most two.
\end{lem}

\begin{proof}
Assume that every nonzero root  $\hat z_n$ of $\hat P_n$ satisfies
\[
\hat P_n(\hat z_n )=\hat P_n'(\hat z_n )=\hat P_n''(\hat z_n )=0. 
\]
Then from \eqref{dPz},
\[
\hat P_{n-1}(\hat z_n ) = \hat P_{n-1}'(\hat z_n )=\hat P_{n-2}(\hat z_n )=0.
\]
Using \eqref{Pol_eq},
\[
\hat z_n \hat P_{n-1}(\hat z_n) + 24 \hat P_{n-1}''(\hat z_n) = \hat P_n( \hat z_n),
\]
from which $\hat P_{n-1}''(z_n)=0$. Therefore, we have shown that
\[
\hat P_{n-1}(\hat z_n )=\hat P_{n-1}'(\hat z_n )=\hat P_{n-1}''(\hat z_n )=0. 
\]
Proceeding by induction, we find that 
\[
\hat P_{3}(\hat z_n )=\hat P_{3}'(\hat z_n )=\hat P_{3}''(\hat z_n )=0. 
\]
However, from \eqref{P3P4},
\[
\hat P_3 (z) = z^3 +48, 
\]
which has three different roots, a contradiction.
\end{proof}

With all the previous information it is not hard to check that the $L^2$-norm of  $Q_{n,0} $ has to diverge. In a neighborhood of such a root $z_0(t)$, one has 
\[
\int Q_{n,0}^2(t,z,\bar z)dzd\bar z  \geq \int_{B(z_0(t),1)} Q_{n,0}^2(t,z,\bar z)dzd\bar z.
\]
In what follows, we will estimate the terms 
\[
P_n(t, z), \quad \partial_z P_n(t,z),
\]
on the set $B(z_0(t),1)$. Indeed, for $z\in B(z_0(t),1)$, if we write $z=z_0(t) +w$, with $w\in B(0,1)$,
\[
\begin{aligned}
|P_n(t, z_0(t) + w )|^2 & = |P_n(t, z_0(t)) + \partial_z P_n(t, z_0(t) + \xi(t)) w |^2, \qquad (\xi(t) \in B(0,1))   \\
& = |\partial_z P_n(t, z_0(t) +\xi(t)) w |^2 \\
& \sim |P_{n-1}(t,z_0(t) + \xi(t))|^2 |w |^2\\
& = | t|^{\frac 23(n-1)} \abs{\hat P_{n-1}\Big(\frac{z_0(t) +\xi(t)}{t^{1/3}}\Big)}^2 |w|^2\\
& = | t|^{\frac 23(n-1)} \abs{\hat P_{n-1}\Big( \hat z_0 + \frac{\xi(t)}{t^{1/3}}\Big)}^2 |w|^2, \quad \hat z_0 := \frac{z_0(t)}{t^{1/3}}.
\end{aligned}
\]
Now we bound each quantity according to the degree of degeneracy of $\hat z_0$. In the case where $\hat z_0$ is not a root for $P_n$, we get for $t$ large
\be\label{Case_1}
\begin{aligned}
|P_n(t, z_0(t) + w )|^2 & \sim  | t|^{\frac 23(n-1)} \abs{\hat P_{n-1}\Big( \hat z_0 + \frac{\xi(t)}{t^{1/3}}\Big)}^2 |w|^2\\
&  \lesssim |t|^{\frac 23(n-1)}|w|^{2}.
\end{aligned}
\ee
The second option is when $\hat z_0$ is degenerate, in the sense that $\hat P_{n-1}( \hat z_0 )=0$. In this case, from Lemma \ref{degeneracy}, and after choosing a correct  root $z_0(t)$, which is only of second order, we have
\[
\begin{aligned}
|P_n(t, z_0(t) + w )|^2 &  \sim  | t|^{\frac 23(n-1)} \abs{\hat P_{n-1}\Big( \hat z_0 + \frac{\xi(t)}{t^{1/3}}\Big)}^2 |w|^2 \\
&  = | t|^{\frac 23(n-2)} |w|^4 \abs{\tilde P_{n-2}\Big( \hat z_0 + \frac{\xi(t)}{t^{1/3}}\Big)}^2,
\end{aligned}
\]
with $\tilde P_{n-2}$ a polynomial of degree $(n-2)$ and such that $\tilde P_{n-2}(\hat z_0)\neq 0$. If $|t|$ is large enough, we will have for $w\in B(0,1)$,
\be\label{Case_2}
|P_n(t, z_0(t) + w )|^2 \lesssim  | t|^{\frac 23(n-2)} |w|^4.
\ee
Now we deal with the derivative term. We have
\[
\begin{aligned}
|\partial_z P_n(t, z_0(t) + w )|^2 & \sim  |P_{n-1}(t, z_0(t) + w )|^2 \\
& = | t|^{\frac 23(n-1)}  \abs{\hat P_{n-1}\Big(\frac{z_0(t) +w}{t^{1/3}}\Big)}^2\\
& = | t|^{\frac 23(n-1)}  \abs{\hat P_{n-1}\Big( \hat z_0 + \frac{w}{t^{1/3}}\Big)}^2,
\end{aligned}
\]
so that 
\be\label{Case_3}
|\partial_z P_n(t, z_0(t) + w )|^2  \gtrsim | t|^{\frac 23(n-1)} ,
\ee
in the case where the root $\hat z_0$ is not degenerate. The remaining case, where we have $\hat P_{n-1}( \hat z_0)=0$, goes as follows. Replicating the same computations as in the previous estimate for $|P_n(t, z_0(t) + w )|^2$ in \eqref{Case_2},
\[
\begin{aligned}
|\partial_z P_n(t, z_0(t) + w )|^2 & \sim  |P_{n-1}(t, z_0(t) + w )|^2 \\
& = | t|^{\frac 23(n-1)}  \abs{\hat P_{n-1}\Big(\frac{z_0(t) +w}{t^{1/3}}\Big)}^2\\
& = | t|^{\frac 23(n-1)}  \abs{\hat P_{n-1}\Big( \hat z_0 + \frac{w}{t^{1/3}}\Big)}^2\\
&= | t|^{\frac 23(n-2)} |w|^2\abs{\tilde P_{n-2}\Big( \hat z_0 + \frac{w}{t^{1/3}}\Big)}^2,
\end{aligned}
\]
with $\tilde P_{n-2}$ a polynomial of degree $(n-2)$ and such that $\tilde P_{n-2}(\hat z_0)\neq 0$. If $|t|$ is large enough, we will have for $w\in B(0,1)$,
\be\label{Case_4}
|\partial_z P_n(t, z_0(t) + w )|^2 \gtrsim | t|^{\frac 23(n-2)} |w|^2.
\ee
Now we conclude by considering two separate cases. Assume \eqref{Case_1} and \eqref{Case_3}. Using polar coordinates,
\[
\begin{aligned}
 \int_{B(z_0(t),1)} Q_{n,0}^2(t,z,\bar z)dzd\bar z  &  \sim   \int_{B(0,1)} \frac{ |\partial_z P_n(t,z_0(t) +w)|^4 dwd\bar w}{(1+|P_n(t,z_0(t) +w )|^2)^4}  \\
 &    \gtrsim  \int_0^1 \frac{   |t|^{\frac43(n-1)} r dr }{(1+  |t|^{\frac 23(n-1)}r^2)^4},
\end{aligned}
\]
with implicit constants independent of time. Consequently, 
\[
\int Q_{n,0}^2(t,z,\bar z)dzd\bar z  \gtrsim  |t|^{\frac23(n-1)},
\]
which proves the result in the case $n\geq 3$ for nondegenerate Gould-Hopper roots. 

\medskip

Let us consider the remaining case, where \eqref{Case_2} and \eqref{Case_4} hold. Performing the same change to polar coordinates as before,
\[
\begin{aligned}
 \int_{B(z_0(t),1)} Q_{n,0}^2(t,z,\bar z)dzd\bar z  &  \sim   \int_{B(0,1)} \frac{ |\partial_z P_n(t,z_0(t) +w)|^4 dwd\bar w}{(1+|P_n(t,z_0(t) +w )|^2)^4}  \\
 &    \gtrsim  \int_0^1 \frac{   |t|^{\frac43(n-2)} r^5 dr }{(1+ | t|^{\frac 23(n-2)} r^4)^4} \\
 & \sim |t|^{\frac13(n-2)} \int_0^\infty \frac{  r^5 dr }{(1+ r^4)^4},
\end{aligned}
\]
which proves the result as far as $|t|$ is large.


\bigskip
\bigskip



\begin{thebibliography}{100}
\bibitem{AT} Adilkhanov A.N., Taimanov I.A., \emph{On a numerical study of the discrete spectrum of
a two-dimensional Schr\"odinger operator with soliton potential}, preprint arXiv:1507.03971 (2015)
\bibitem{A} Angelopoulos Y. \emph{Well-posedness and ill-posedness results for the Novikov-Veselov equation}, Comm. Pure Applied Anal. 15(3), 727--760 (2016) preprint arXiv:1307.4110
\bibitem{Bog} Bogdanov L. V. The Veselov-Novikov equation as a natural generalization of the Korteweg-de Vries equation. Teoret. Mat. Fiz. 70(2), 309--314 (1987), translation in Theoret. and Math. Phys. 70(2), 219?223 (1987)
\bibitem{BM} de Bouard A., Martel Y. \emph{Nonexistence of $ L^2 $-compact solutions of the Kadomtsev-Petviashvili II equation}. Math. Ann. 328, 525--544 (2004).
\bibitem{BS1} de Bouard A., Saut J.-C. \emph{Solitary waves of generalized Kadomtsev-Petviashvili equations}. Ann. Inst. Henri Poincar\'e, Analyse Non Lin\'eaire. 14(2), 211--236 (1997).
\bibitem{BS2} de Bouard A., Saut J.-C. \emph{Symmetries and decay of the generalized Kadomtsev-Petviashvili solitary waves}. SIAM J. Math. Anal. 28(5), 1064--1085 (1997).
\bibitem{B1} Bourgain J. \emph{On the Cauchy problem for the Kadomtsev-Petviashvili equations}, Geom. Funct. Anal. 3(4), 315--341 (1993) 
\bibitem{CKZ} Carbery A., Kenig C.E., Ziesler, S. \emph{Restriction for homogeneous polynomial surfaces in $\R^3$}. Trans. Amer. Math. Soc.  365(5), 2367--2407 (2013)
\bibitem{Chang} Chang, J.-H., \emph{The Gould-Hopper polynomials in the Novikov-Veselov equation}, J. Math. Phys. 52(9), 092703 (2011)
\bibitem{Perry} Croke R., Mueller J.L., Music M., Perry P., Siltanen S., Stahel A. \emph{The Novikov-Veselov Equation: Theory and Computation}, Contemporary Mathematics 635, 25--70 (2015)
\bibitem{G} Ginibre J. \emph{Le probl\`eme de Cauchy pour des EDP semi-lin\'eaires p\'eriodiques en variables d'{}espace}, S\'eminaire N. Bourbaki, exp. no 796, 163--187 (1994-1995)
\bibitem{Gr} Grinevich P.G. \emph{Rational solitons of the Veselov--Novikov equation are reflectionless potentials at fixed energy}, Teoret. Mat. Fiz. 69(2), 307--310 (1986), translation in Theor. Math. Phys. 69, 1170--1172 (1986).
\bibitem{Gr2} Grinevich P.G. \emph{Scattering transformation at fixed non-zero energy for the two-dimensional Schr\"odinger operator with potential decaying at infinity} Russ. Math. Surv. 55(6), 1015--1083 (2000)
\bibitem{KN} Kazeykina A.V., Novikov R.G. \emph{A large time asymptotics for transparent potentials for the Novikov--Veselov equation at positive energy}. J. Nonlinear Math. Phys. 18(3), 377--400 (2011).
\bibitem{KN2} Kazeykina A.V., Novikov R.G. \emph{Large time asymptotics for the Grinevich--Zakharov potentials}. Bulletin des Sciences Math\'ematiques. 135, 374--382 (2011)
\bibitem{K} Kazeykina A.V. \emph{A large time asymptotics for the solution of the Cauchy problem for the Novikov-Veselov equation at negative energy with non-singular scattering data}. Inverse Problems, 28(5), 055017 (2012).
\bibitem{K2} Kazeykina A.V. \emph{Absence of solitons with sufficient algebraic localization for the Novikov-Veselov equation at nonzero energy}. Funct. Anal. Appl., 48(1), 24--35 (2014)
\bibitem{KM} Kazeykina A., Mu\~noz C. \emph{Dispersive estimates for rational symbols and local well-posedness of the nonzero energy NV equation}, J. Funct. Anal. 270(5), 1744--1791 (2016)
\bibitem{KPV} Kenig C.E, Ponce G., Vega L. \emph{Well-posedness of the initial value problem for the Korteweg-deVries equation}, J. Amer. Math. Soc. 4, 323--347 (1991)
\bibitem{KPV_Indiana} Kenig C.E., Ponce, G., Vega, L. \emph{Oscillatory Integrals and Regularity of Dispersive Equations}, Indiana Univ. Math. J. 40(1), 33--69 (1991)
\bibitem{KPV1} Kenig C.E., Ponce G., Vega L. \emph{A bilinear estimate with applications to the KdV equation}, J. Amer. Math. Soc. 9(2), 573--603 (1996)
\bibitem{Kis} Kiselev O.M. \emph{Asymptotics of solutions of higher-dimensional integrable equations and their perturbations} J. Math. Sci. 138(6), 6067--6230 (2006).
\bibitem{KBook} Konopelchenko B. G. \emph{Introduction to multidimensional integrable equations. The Inverse Scattering Transform in $2+1$ dimensions}. Plenum Press, New York, 1992. x+292 pp. ISBN: 0-306-44220-5.
\bibitem{Lax} Lax P.D. \emph{Integrals of nonlinear equations of evolution and solitary waves}. Comm. Pure Appl. Math. 21, 467--490 (1968)
\bibitem{LP} Linares F., Pastor A. \emph{Well-posedness for the two-dimensional modified Zakharov-Kuznetsov equation}, SIAM J. Math. Anal. 41(4), 1323--1339 (2009)
\bibitem{Manakov}  Manakov S.V. \emph{The inverse scattering method and two-dimensional evolution equations}. Uspekhi Mat. Nauk 31(5), 245--246 (1976) (in Russian).
\bibitem{MST} Manakov S. V., Santini P. M., Takchtadzhyan L. A. \emph{An asymptotic behavior of the solutions of the
Kadomtsev-Petviashvili equations} Phys. Lett. A. 75, 451--454 (1980).
\bibitem{MZBIM} Manakov S.V., Zakharov V.E., Bordag L.A., Its A.R., Matveev V.B. \emph{Two--dimensional solitons of the Kadomtsev--Petviashvili equation and their interaction} Physics Letters A. 63(3), 205--206 (1977).
\bibitem{MRS} Merle F., Rapha\"el P., Szeftel, J. \emph{The instability of Bourgain-Wang solutions for the $L^2$-critical NLS}, Amer. Math. Jour. 135(4), 967--1017 (2013)
\bibitem{MP} Molinet L., Pilod D. \emph{Bilinear Strichartz estimates for the Zakharov-Kuznetsov equation and applications}, Ann. IHP Analyse Nonlin\'eaire 32(2), 347--371 (2015)
\bibitem{MST1} Molinet L., Saut J.-C., Tzvetkov N. \emph{Well-posedness and ill-posedness results for the Kadomtsev-Petviashvili-I equation} Duke Math. J. 115(2), 353--384 (2002)
\bibitem{Murate} Murata, M. \emph{Structure of positive solutions to $(-\Delta +V)\varphi=0$ in $\R^n$}. Duke Math. J. 53(4), 869--943 (1986)
\bibitem{MusicPerry} Music M., Perry P. \emph{Global solutions for the zero-energy Novikov-Veselov equation by inverse scattering}  preprint arXiv:1502.02632 (2015)
\bibitem{Nov} Novikov, R. G. \emph{The inverse scattering problem on a fixed energy level for the two-dimensional Schr\"odinger operator}. J. Funct. Anal. 103(2), 409--463 (1992)
\bibitem{NV} Novikov S.P., Veselov A.P. \emph{Finite-zone, two-dimensional, potential Schr\"odinger operators. Explicit formula and evolutions equations}, Sov. Math. Dokl. 30, 588--591 (1984)
\bibitem{Perry2} Perry, P. \emph{Miura maps and inverse scattering for the Novikov-Veselov equation}. 7(2), 311--343 (2014)
\bibitem{Saut} Saut J.-C. \emph{Remarks on the generalized Kadomtsev-Petviashvili equations}, Indiana Univ. Math. Journal 42(3), 1011--1026 (1993)
\bibitem{ST} Saut J.-C., Tzvetkov N. \emph{On periodic KP-I type equations}. Commun. Math. Phys. 221, 451--476 (2001)
\bibitem{Schottdorf} Schottdorf T., \emph{Global Existence Without Decay}, Ph.D. dissertation U. Bonn, Germany, 2013, retrieved at \url{http://hss.ulb.uni-bonn.de/2013/3430/3430.htm}.
\bibitem{Simon} Simon B., \emph{The bound state of weakly coupled Schr\"odinger operators in one and two dimensions}, Ann. Phys. 97, 279--288 (1976)
\bibitem{SW} Stein E.M., Weiss, G. \emph{Introduction to Fourier analysis on Euclidean spaces}, Princeton University Press, 1971
\bibitem{TS2008} Taimanov I. A., Tsarev, S. P. \emph{Blowing up solutions of the Veselov-Novikov equation}, (Russian) Dokl. Akad. Nauk 420(6), 744--745 (2008); translation in Dokl. Math. 77(3) 467--468 (2008)
\end{thebibliography}
\end{document}